\documentclass[12]{amsart}

\usepackage[english, activeacute]{babel}
\usepackage{algorithm}
\usepackage{algorithmic}
\usepackage{amsmath,amsthm,amsxtra}
\usepackage{fix-cm}
\usepackage{bm}
\usepackage{graphics}
\usepackage[dvipsnames]{xcolor}
\usepackage{graphicx}	
\usepackage{fancyhdr}
\usepackage{float}
\usepackage{thmtools}
\usepackage{stmaryrd}
\usepackage{amssymb,mathrsfs}
\usepackage{arydshln}
\usepackage{enumerate}
\usepackage{mathtools}
\usepackage{epsfig}
\usepackage{amssymb}
\usepackage{float}
\usepackage{tikz-cd}
\usetikzlibrary{calc,backgrounds}
\usepackage{dsfont}
\usetikzlibrary{fit,calc,positioning,decorations.pathreplacing,matrix,automata}
\usepackage{pstricks-add}
\usepackage{pst-node,pst-plot}
\usepackage{latexsym}
\usepackage{amsfonts}
\usepackage{hyperref}
\usepackage{xfrac}
\usepackage[capitalize]{cleveref}
\usepackage{enumitem}
\usepackage{todonotes}

\usepackage{tikz}
\usetikzlibrary{automata,positioning,matrix}

\newtheorem{theorem}{Theorem}[section]
\newtheorem*{theorem*}{Theorem}
\newtheorem{corollary}[theorem]{Corollary}

\newtheorem{lemma}[theorem]{Lemma}
\newtheorem{question}[theorem]{Question}
\newtheorem{proposition}[theorem]{Proposition}

\theoremstyle{definition}
\newtheorem{definition}[theorem]{Definition}
\newtheorem{remark}[theorem]{Remark}

\newtheorem{claim}{Claim}

\theoremstyle{definition}
\newtheorem{example}[theorem]{Example}

\newcommand{\N}{\mathbb{N}}
\newcommand{\Z}{\mathbb{Z}}

\newcommand{\A}{\mathcal{A}}

\newcommand{\R}{\mathbb{R}}
\newcommand{\B}{\mathcal{B}}

\DeclareMathOperator{\Aut}{Aut}
\DeclareMathOperator{\diam}{diam}
\DeclareMathOperator{\End}{End}
\DeclareMathOperator{\Homeo}{Homeo}
\DeclareMathOperator{\supp}{supp}
\DeclareMathOperator{\diag}{diag}
\DeclareMathOperator{\id}{id}

\newcounter{sigmavariable}
\makeatletter
\def\moverlay{\mathpalette\mov@rlay}
\def\mov@rlay#1#2{\leavevmode\vtop{%
		\baselineskip\z@skip \lineskiplimit-\maxdimen
		\ialign{\hfil$\m@th#1##$\hfil\cr#2\crcr}}}
\newcommand{\charfusion}[3][\mathord]{
	#1{\ifx#1\mathop\vphantom{#2}\fi
		\mathpalette\mov@rlay{#2\cr#3}
	}
	\ifx#1\mathop\expandafter\displaylimits\fi}
\makeatother

\newcommand{\cupdot}{\charfusion[\mathbin]{\cup}{\cdot}}

\newtheorem{maintheorem}{Theorem}

\definecolor{zzttqq}{rgb}{0.6,0.2,0.}



\title[]{Decidability of the isomorphism problem between multidimensional substitutive subshifts}

\author{Christopher Cabezas}
\address{University of Li\`ege, Department of Mathematics, All\'ee de la d\'ecouverte 12 (B37), B-4000 Li\`ege, Belgium.}
\email{ccabezas@uliege.be}

\author{Julien Leroy}
\address{University of Li\`ege, Department of Mathematics, All\'ee de la d\'ecouverte 12 (B37), B-4000 Li\`ege, Belgium.}
\email{j.leroy@uliege.be}

\thanks{The authors acknowledge the financial support of ANR project IZES ANR-22-CE40-0011.}

\begin{document}
	
	\begin{abstract}
		An important question in dynamical systems is the classification problem, i.e., the ability to distinguish between two isomorphic systems. In this work, we study the topological factors between a family of multidimensional substitutive subshifts generated by morphisms with uniform support. We prove that it is decidable to check whether two minimal aperiodic substitutive subshifts are isomorphic. The strategy followed in this work consists of giving a complete description of the factor maps between these subshifts. Then, we deduce some interesting consequences on coalescence, automorphism groups, and the number of aperiodic symbolic factors of substitutive subshifts. We also prove other combinatorial results on these substitutions, such as the decidability of defining a subshift, the computability of the constant of recognizability, and the conjugacy between substitutions with different supports.
	\end{abstract}
	\maketitle
	
	\section{Introduction}
	Isomorphic systems are indistinguishable in terms of their dynamical properties, making classification an important problem in dynamical systems. Nevertheless, finding an isomorphism (or conjugation) between two dynamical systems has proven to be highly challenging. We recall that an \emph{isomorphism} between two symbolic systems $(X,S,\Z^{d})$ and $(Y,S,\Z^{d})$ is a continuous and bijective map $\phi:(X,S,\Z^{d})\to (Y,S,\Z^{d})$ commuting with the action, i.e., for any ${\bm n}\in \Z^{d}$, $\phi\circ S^{{\bm n}}=S^{{\bm n}}\circ \phi$. If the map $\phi$ is only surjective, it is called a \emph{factor map}. 
	
	One classic approach to address the classification problem involves identifying \emph{invariants}, which are properties shared by isomorphic systems and are easily determinable. However, in some cases, the existing invariants may not suffice for this purpose. Additionally, the topological factors of a topological dynamical system are rarely used explicitly to unravel the structure of a particular dynamical system. Nonetheless, they contain valuable information for certain aspects and can be employed for concrete computations or the study of specific structures, such as in spectral theory.
	
	We are also concerned with the \emph{decidability} of certain properties. A property is said to be decidable if an algorithm exists that allows one to verify whether the property is satisfied or not. In this article, our focus lies on the decidability of the classification problem within the family of \emph{multidimensional substitutive subshifts}.
	
	One-dimensional substitutive subshifts have been extensively studied for several decades, ever since they were introduced by W.H. Gottschalk in \cite{gottschalk1963substitution}. They represent the simplest nontrivial zero-entropy symbolic systems and are generated in a highly deterministic manner. This simplicity has led to their presence in various fields of mathematics, computer science, and physics, such as combinatorics of words, number theory, dynamics of aperiodic tilings, quasi-crystals, and more (see, for example, \cite{adamczewski2007complexity,allouche2003automatic,mozes1989tilings,nekrashevych2018palindromic}). 
	However, their deep understanding took several decades, with significant contributions made by A. Cobham \cite{cobham1972uniform} (who identified them as so-called \emph{automatic sequences} from a computational perspective), M. Queff\'elec and others \cite{dekking1978spectrum,siegel2005spectral,queffelec2010substitution} (in terms of their spectral properties), B. Moss\'e \cite{mosse1992puissances} (focused on recognizability, which is a sort of invertibility of substitutions), F. Durand \cite{durand1999substitutional} (who classified their topological factor systems), and B. Host and F. Parreau among several others \cite{coven2016computing,cyr2015automorphism,donoso2016lowcomplexity,host1989homomorphismes,lemanczyk1988metric,salo2015blockmaps}  (classification of their automorphism groups). 
	We refer to \cite{queffelec2010substitution,fogg2002substitutions} for extensive bibliographies on the (earlier developments of) the subject. Many of the aspects mentioned above remain largely unexplored in the context of multidimensional substitutive systems.
	
	In the multidimensional setting, substitutive subshifts find their motivation in physical phenomena, particularly through the discovery of the aperiodic structure of quasi-crystals, modeled by the Penrose tiling \cite{penrose1974role}. In these models, symmetries play a fundamental role and are described using finite data. Numerous articles have been dedicated to the study of these tilings (see \cite{baake2013aperiodic} for an extensive bibliography on aperiodic order). Substitutive systems have then emerged as valuable mathematical models within this research direction. Our focus in this article is on substitutions with uniform support, where the shape of any pattern defined by the substitution remains the same (see \cite{cabezas2023homomorphisms} for basic properties on this topic, where we follow the same notation). Within this class of substitutive subshifts, we prove that assuming they have the same combinatorial structure, it is decidable whether a factor map exists between two aperiodic substitutive subshifts.
	
	\begin{maintheorem}\label{Thm:DecidabilityFactorizationProblemIntro}
		Let $\zeta_{1},\zeta_{2}$ be two aperiodic primitive constant-shape substitutions with the same expansion matrix $L$. It is decidable to know whether there exists a factor map $\phi:(X_{\zeta_{1}},S,\Z^{d})\to (X_{\zeta_{2}},S,\Z^{d})$.
	\end{maintheorem}
	
	The strategy followed in this article involves providing a complete description of the factor maps between substitutive subshifts. This approach draws inspiration from a series of works on automorphism groups of symbolic systems, which we proceed to describe. The study of factors and conjugacies between dynamical systems is a classical problem, primarily concerning their algebraic and dynamical properties in relation to the one of the system $(X, T, \Z^{d})$. \emph{Automorphisms}, which are self-conjugacies of a particular system, can be algebraically defined as elements of the centralizer of the action group $\left\langle T\right\rangle$, considered as a subgroup of all homeomorphisms $\Homeo(X)$ from X to itself. Symbolic systems already exhibit significant rigidity properties regarding factor maps and automorphisms. For instance, the famous Curtis-Hedlund-Lyndon theorem \cite{hedlund1969endomorphism} establishes that any factor map between subshifts is a sliding block code, implying that the automorphism group is countable and discrete. Initially studied for \emph{subshifts of finite type} \cite{hedlund1969endomorphism}, these automorphism groups were shown to be infinitely generated and containing various groups, including all finite groups, the free group on two generators, the direct sum of a countable number of copies of $\mathbb{Z}$, and any countable collection of finite groups. The existence of a conjugacy between two subshifts of finite type is known to be equivalent to the notion of \emph{Strong Shift Equivalence} for matrices over $\Z^{+}$ \cite{williams1973classification}, which is not known to be decidable \cite{kim1999williams}.
	
	However, within the rich family of substitutive subshifts, factor maps exhibit strong rigidity properties, as proven by B. Host and F. Parreau in \cite{host1989homomorphismes}. They provided a complete description of factor maps between subshifts arising from certain constant-length substitutions, proving that any measurable factor map induces a continuous one. As a consequence, the automorphism group is virtually generated by the shift action, meaning that there exists a finite set of automorphisms such that any automorphism can be expressed as the composition of an element from this finite set and a power of the shift. Moreover, any finite group can be realized as a quotient group $\Aut(X,S,\Z)/\left\langle S\right\rangle$ for these subshifts, as proven by M. Lema\'nczyk and M. K. Mentzen in \cite{lemanczyk1988metric}. The proof by B. Host and F. Parreau is based on the following fact: there exists a bound (in this case, $r=2$) such that any factor map between these substitutive subshifts is the composition of a sliding block code with a radius less than $r$ and a power of the shift map. Using the self-induced properties of substitutive subshifts, V. Salo and I. Törmä provided in \cite{salo2015blockmaps} a renormalization process for factor maps between two minimal substitutive subshifts of constant-length and for Pisot substitutions, extending the description obtained in \cite{host1989homomorphismes} within a topological framework. More recently, F. Durand and J. Leroy \cite{durand2022decidability} showed the decidability of the factorization problem between two minimal substitutive subshifts, extending the results of V. Salo and I. Törmä, giving a \emph{computable} upper bound $R$ such that every factor map between minimal substitutive subshifts is the composition of a power of the shift map with a factor map having a radius less than $R$. The decidability of the factorization problem in the constant-length case had previously been proved by I. Fagnot in \cite{fagnot1997facteurs} using the first-order logic framework of Presburger arithmetic, without assuming minimality. In \cite{cabezas2023homomorphisms}, an analogous result to that of B. Host and F. Parreau for the multidimensional framework was established. In this article, we further extend the findings in \cite{cabezas2023homomorphisms} to the whole class of aperiodic minimal multidimensional constant-shape substitutive subshifts. 
	
	\begin{maintheorem}\label{Thm:TopologicalRigidityFactorBetweenSubstitutions}
		Let $\zeta_{1},\zeta_{2}$ be two aperiodic primitive constant-shape substitutions with the same expansion matrix. Then, there exists a computable constant $R$ such that every factor map between $(X_{\zeta_{1}},S,\Z^{d})$ and $(X_{\zeta_{2}},S,\Z^{d})$ is the composition of a shift map with a factor map of radius less than $R$.
	\end{maintheorem}
	
	The constant $R$ of the previous theorem depends on the constant of recognizability of the image substitution $\zeta_{2}$. In \cite{cabezas2023homomorphisms} it was already established that aperiodic primitive constant-shape substitutions are recognizable. In this article, we prove that this constant is computable.
	
	\begin{maintheorem}\label{Thm:ComputabilityOfTheRecognizability}
		Let $\zeta$ be an aperiodic primitive constant-shape substitution with expansion matrix $L$ and support $F$. There is a computable upper bound for the constant of recognizability of $\zeta$. This bound can be expressed only by the cardinality of the alphabet $|\A|$, the expansion matrix $L$, the support $F$ and the dimension $d$.
	\end{maintheorem}
	
	This result is an analog of the one proved by F. Durand and J. Leroy in \cite{durand2017constant} for the one-dimensional case.
	
	This article is organized as follows. The basic definitions and background are introduced in \cref{Section:BasicDefinitions}. \cref{Section:DecidabilityFOlner} is devoted to the study of the supports of a constant-shape substitution. We prove the decidability of whether this sequence is F{\o}lner (\cref{thm:decidabilityfolnerporperty}), useful to define the substitutive subshift. Then, we use this proof to get a bound on their complexity function (\cref{Lemma:UpperBoundComplexityFunction}). In \cref{Section:ConjugacySameMatrices} we deal with the conjugacy problem between two aperiodic primitive constant-shape substitutions with the same expansion matrix but different support. We prove that for any pair of different supports $F_{1}, G_{1}$ of an expansion matrix and any constant-shape substitution with support $F_{1}$, there exists a constant-shape substitution with support $G_{1}$ such that the two substitutive subshifts are topologically conjugate (\cref{ConjugationSubstitutionsDifferentFundamentalDomains}). This answers a question raised in \cite{fogg2007substitutions}, where a similar result was shown for the one-dimensional case. \cref{Section:ComputabilityRecognizability} is devoted to the computability of the constant of recognizability of constant-shape substitutions. To do this, we study the computability of the repetitivity function for substitutive subshifts (\cref{GrowthRepetititvtyFunction}).
	Finally, in \cref{Section:RigidityFactorsSubstitutiveSubshifts} we characterize the factor maps between aperiodic primitive substitutive subshifts sharing the expansion matrix (\cref{TopologicalRigidityOfFactors}). Then, we deduce
	the coalescence of substitutive subshifts (\cref{Coalescence}), meaning any endomorphism between the substitutive subshift and itself is invertible. We also prove that the automorphism group of substitutive subshifts is virtually generated by the shift action (\cref{AutomoprhismVirtuallyZd}). Additionally, we use \cref{TopologicalRigidityOfFactors} to conclude the decidability of the factorization problem between substitutive subshifts having the same expansion matrix (\cref{Thm:DecidabilityFactorizationProblem}). Thanks to the coalescence of substitutive subshifts, we also deduce the decidability of the isomorphism problem (\cref{cor:DecidabilityIsomorphismProblem}). We finish this section proving that substitutive subshifts have finitely many aperiodic symbolic factors, up to conjugacy (\cref{Lemma:FinitelyManyAperiodicSymbolicFactors}). Thanks to \cref{TopologicalRigidityOfFactors}, we can provide a list containing these factors. 
	
	\section{Definitions and basic properties}\label{Section:BasicDefinitions}
	\subsection{Basic definitions and notation}
	\subsubsection{Notation} Throughout this article, we will denote by ${\bm n}=(n_{1},\ldots,n_{d})$ the elements of $\Z^{d}$ and by ${\bm x}=(x_{1},\ldots,x_{d})$ the elements of $\R^{d}$. If $F\subseteq \Z^{d}$ is a finite set, it will be denoted by $F\Subset \Z^{d}$, and we use the notation $\Vert F\Vert =\max\limits_{{\bm n}\in F}\Vert {\bm n}\Vert$, where $\Vert\cdot\Vert$ is the standard Euclidean norm of $\R^{d}$. If $L\in \mathcal{M}(d,\R)$ is a $d\times d$-matrix, we denote $\Vert L\Vert=\max\limits_{{\bm x}\in \R\setminus\{{\bm 0}\}} \Vert L({\bm x})\Vert/\Vert {\bm x}\Vert$ as the \emph{matrix norm of} $L$.
	
	A sequence of finite sets $(F_{n})_{n>0}\subseteq \Z^{d}$ is said to be \emph{F}\o\emph{lner}\footnote{In the literature, especially group theory, it is common to also ask that the union of the sequence of sets $(F_{n})_{n>0}$ is equal to $\Z^{d}$ for a sequence to be F\o lner, but we will not use it in this article.} if for any ${\bm n}\in \Z^{d}$
	\[
	\lim\limits_{n\to \infty}\dfrac{|F_{n}\Delta ({\bm n}+F_{n})|}{|F_{n}|}=0,
	\]
	where $|X|$ stands for the cardinality of the set $X$. For any pair of subsets $E, F\subseteq \Z^{d}$, we denote $F^{\circ E}$, as the set of all elements ${\bm f}\in F$ such that ${\bm f}+E\subseteq F$, i.e.,
	\[
	F^{\circ E}=\{{\bm f}\in F\colon {\bm f}+E\subseteq F\}.
	\]
	In the case $E$ is a discrete ball centered at the origin, meaning $E=[B({\bm 0},r)\cap \Z^{d}]$ for some $r>0$, we will denote $F^{\circ [B({\bm 0}, r) \cap \Z^{d}]}$ simply by $F^{\circ r}$. 
	Note that the F\o lner assumption implies that for any $E\Subset \Z^{d}$,
	\begin{equation}\label{LargerballsFolner}
		\lim\limits_{n\to \infty}\dfrac{|F_{n}^{\circ E}\Delta F_{n}|}{|F_{n}|}=1.
	\end{equation}	
	
	\subsection{Symbolic Dynamics}\label{SubsectionSymbolic Dynamics}    
	Let $\A$ be a finite alphabet, and let $d\geq 1$ be an integer. We define a topology on $\A^{\Z^{d}}$ by endowing $\A$ with the discrete topology and considering on $\A^{\Z^{d}}$ the product topology, generated by cylinders. Since $\A$ is finite, $\A^{\Z^{d}}$ is a metrizable compact space. The additive group $\Z^{d}$ acts on this space by translations (or shifts), defined for every ${\bm n}\in \Z^{d}$ by
	\[
	S^{{\bm n}}(x)_{{\bm k}}=x_{{\bm n}+{\bm k}},\ x\in \A^{\Z^{d}},\ {\bm k}\in \Z^{d}.
	\]
	
	The $\Z^{d}$-action $(\A^{\Z^{d}},S,\Z^{d})$ is called the \emph{full-shift}. 
	
	Let $P\subseteq \Z^{d}$ be a finite set. A \emph{pattern} is an element $\texttt{p}\in \A^{P}$. We say that $P$ is the \emph{support} of $\texttt{p}$, denoted $P=\supp(\texttt{p})$. We say that a pattern $\texttt{p}$ \emph{occurs in} $x\in \A^{\Z^{d}}$ if there exists ${\bm n}\in \Z^{d}$ such that $\texttt{p}=x|_{{\bm n}+P}$ (identifying $P$ with ${\bm n}+P$ by translation). In this case, we denote it $\texttt{p}\sqsubseteq x$, and we call such ${\bm n}$ an \emph{occurrence in} $x$ \emph{of} $\texttt{p}$.
	
	A \emph{subshift} $(X,S,\Z^{d})$ is given by a closed subset $X\subseteq \A^{\Z^{d}}$ that is invariant under the $\Z^{d}$-action. In this article, even if the alphabet changes, $S$ will always denote the shift map, and we usually say that $X$ itself is a subshift. A subshift can also be defined by its language. For $P\Subset \Z^{d}$, we denote
	$$\mathcal{L}_{P}(X)=\{\texttt{p}\in \A^{P}: \exists x \in X,\ \texttt{p}\sqsubseteq x\}.$$
	We define the {\em language} of a subshift $X$ by
	\[
	\mathcal{L}(X)=\bigcup\limits_{P\Subset \Z^{d}}\mathcal{L}_{P}(X).
	\]
	
	We say that the subshift $(X,S,\Z^{d})$ is {\em minimal} if it does not contain proper non-empty subshifts. The subshift is \emph{aperiodic} if there are no nontrivial periods; that is, if $S^{{\bm p}}x=x$ for some ${\bm p}\in \Z^{d}$ and $x\in X$, then ${\bm p}=0$.
	
	Let $\B$ be a finite alphabet, and consider a subshift $Y\subseteq \B^{\Z^{d}}$. A map $\phi:(X,S,\Z^{d}) \to (Y,S,\Z^{d})$ is a {\em factor map} if it is continuous, surjective and commutes with the actions, i.e., $\phi \circ S^{\bm n} = S^{\bm n} \circ \phi$ for all ${\bm n} \in \Z^d$. 
	In this case, we say that $(Y,S,\Z^{d})$ is a \emph{factor} of $(X,S,\Z^{d})$. 
	If $\phi$ is also injective, we say it is a \emph{conjugacy} (or an \emph{isomorphism}).
	When $\phi:(X,S,\Z^{d})\to (Y,S,\Z^{d})$ is a factor map, there exists a finite subset $P\Subset \Z^{d}$ and a $P$-\emph{block map} $\Phi: \mathcal{L}_{P}(X)\to \B$ such that for any ${\bm n}\in \Z^{d}$ and $x\in X$, $\phi(x)_{{\bm n}}= \Phi(x|_{{\bm n}+ P})$. This is known as the Curtis-Hedlund-Lyndon theorem \cite{hedlund1969endomorphism}. We call such $P$ {\em the support} of $\Phi$ and {\em a support} of $\phi$. 
	Observe that if $\phi$ is induced by $\Phi$, we can define another block map $\Phi'$ also inducing $\phi$ and whose support is a discrete ball of the form $[B({\bm 0},r)\cap \Z^{d}]$, for $r\in \N$. 
	We define the \emph{radius} of $\phi$ (and denote it $r(\phi)$) as the infimum of $r\in\N$ such that $\phi$ is induced by a block map with support $[B({\bm 0},r)\cap \Z^{d}]$.
	
	\subsection{Multidimensional constant-shape substitutions}\label{SectionDefinitionsMultidimensional}
	
	We recall some basic definitions and results about multidimensional substitutive subshifts of constant-shape that will be used throughout this article. We refer to \cite{cabezas2023homomorphisms} for basic properties on this topic, where we follow the same notation (see also \cite{frank2022spectral} for spectral properties of these subshifts). Let $L\in \mathcal{M}(d,\Z)$ be an expansion integer matrix, i.e., there exists $\lambda>1$ such that for every ${\bm x}\in \R^{d}\setminus\{0\}$, we have that $\Vert L({\bm x}) \Vert >\lambda \Vert {\bm x}\Vert$. 
	Let $F$ be a fundamental domain of $L(\Z^{d})$ in $\Z^{d}$, meaning a set of representative classes of $\Z^{d}/L(\Z^{d})$ (with ${\bm 0}\in F$), and let $\A$ be a finite alphabet. A \emph{multidimensional constant-shape substitution} is a map $\zeta:\A\to \A^{F}$. 
	The set $F$ is called the \emph{support} of the substitution. The following shows an example of a constant-shape substitution:
	
	\begin{example}[Triangular Thue-Morse substitution]\label{ExampleTriangularThueMorse} The \emph{triangular Thue-Morse substitution} is defined with $L=2\id_{\R^{2}}$, $F=\left\{(0,0),(1,0),(0,1),(-1,-1)\right\}$ and $\A=\{a,b\}$ as
		$$\begin{array}{cccccccccccc}
			&  &  &  & b &  &  &  &  &  & a &  \\ 
			\sigma_{\Delta TM}: & a & \mapsto &  & a & b, &  & b & \mapsto &  & b & a. \\ 
			&  &  & a &  &  &  &  &  & b &  &  \\ 
		\end{array}$$
	\end{example}      
	
	In the literature, constant-shape substitutions with a positive diagonal expansion matrix $L=\diag(l_{i})_{i=1,\ldots,d}$ and support equal to the standard $d$-dimensional parallelepiped $F_{1}=\prod\limits_{i=1}^{d}\llbracket 0,l_{i}-1\rrbracket$ are called \emph{block substitutions}. These substitutions have a characteristic block structure defined by the shape of $F_{1}$. Moreover, when $L=p\id_{\R^{2}}$ is equal to some positive multiple of the identity, and the support is equal to $F=\llbracket 0, p-1\rrbracket^{2}$, we use the term \emph{square substitution} to describe such cases. 
	
	Given a substitution $\zeta$, we let $L_{\zeta}$ denote its expansion matrix, and $F_{1}^{\zeta}$ its support. For any $n>0$, we define the $n$-th iteration of the substitution $\zeta^{n}:\A\to \A^{F_{n}^{\zeta}}$ by induction: $\zeta^{n+1}=\zeta\circ \zeta^{n}$, where the supports of these substitutions satisfy the recurrence 
	\begin{equation}\label{Eq:RecurrenceSets}
		F_{n+1}^{\zeta}=L_\zeta(F_{n}^{\zeta})+F_{1}^{\zeta},\quad \forall n\geq 1.
	\end{equation}
	Observe that we trivially have $L_{\zeta^n} = L_\zeta^n$.
	
	The \emph{language} of a substitution is the set of all patterns that appear in $\zeta^{n}(a)$, for some $n>0$, $a\in \A$, i.e.,
	$$\mathcal{L}_{\zeta}=\{\texttt{p}\colon \texttt{p}\sqsubseteq \zeta^{n}(a),\ \text{for some }n>0,\ a\in \A\}.$$
	
	For such a language to define a subshift, we need that the sets $F_n^\zeta$ contain arbitrarily large balls for $n$ large enough, i.e., for any $r>0$, there exists $n>0$ such that $(F_{n}^{\zeta})^{\circ r}\neq \emptyset$.
	This condition is ensured by the F\o lner property. Moreover, the F{\o}lner assumption on the sequence of fundamental domains $(F_{n})_{n>0}$ is necessary to ensure that the subshift is uniquely ergodic, following the proof given in \cite{lee2003consequences} for substitutive Delone sets. 
	We prove in the next section that the F\o lner property is actually equivalent to the existence of arbitrarily large balls in the sequence of supports. 
	We also show that whether a sequence $(F_n)_{n>0}$ satisfies this property is decidable (see~\cref{thm:decidabilityfolnerporperty}).
	
	The following shows the first three iterations of the substitution given in \cref{ExampleTriangularThueMorse}.
	
	\begin{example}[Iterations of a constant-shape substitution]\label{FirstsIterationsExample} The first three iterations of the substitution $\sigma_{\Delta TM}$ illustrated in \cref{ExampleTriangularThueMorse}.
		$$\begin{array}{ccccccccccccccccc}
			&  &  &  &  &  &  &  &  & a &  &  &  &  &  &  &  \\ 
			&  &  &  &  &  &  &  &  & b & a &  &  &  &  &  &  \\ 
			&  &  & b &  &  &  &  & b & b &  & a &  &  &  &  &  \\ 
			a & \mapsto &  & a & b & \mapsto &  &  &  & a & b & b & a &  &  &  &  \\ 
			&  & a &  &  &  &  & b & a &  & b &  &  &  &  &  &  \\ 
			&  &  &  &  &  &  & a & b &  &  &  &  &  &  &  &  \\ 
			&  &  &  &  &  & a &  &  &  &  &  &  &  &  &  &  \\
		\end{array}$$
		\begin{figure}[H]
			$$\begin{array}{ccccccccccccccccc}
				&  &  &  &  &  &  &  &  &  &  &  &  &  &  &  &  \\ 
				&  &  &  &  &  &  &  &  &  &  &  &  &  &  &  &  \\ 
				&  &  &  &  &  &  &  &  & b &  &  &  &  &  &  &  \\ 
				&  &  &  &  &  &  &  &  & a & b &  &  &  &  &  &  \\ 
				&  &  &  &  &  &  &  & a & a &  & b &  &  &  &  &  \\ 
				&  &  &  &  &  &  &  &  & b & a & a & b &  &  &  &  \\ 
				&  &  &  &  &  &  & a & b & a & a &  &  & b &  &  &  \\ 
				&  &  &  &  &  &  & b & a & b & a &  &  & a & b &  &  \\ 
				&  &  &  &  &  & b &  & b & b &  & a & a & a &  & b &  \\ 
				& \mapsto &  &  &  &  &  &  &  & a & b & b & a & b & a & a & b \\ 
				&  &  &  &  & a &  & b & a &  & b & a & b &  & a &  &  \\ 
				&  &  &  &  & b & a & a & b &  &  & b & a &  &  &  &  \\ 
				&  &  &  & b & b & a & a &  &  & b &  &  &  &  &  &  \\ 
				&  &  &  &  & a & b & b & a &  &  &  &  &  &  &  &  \\ 
				&  &  & b & a &  & b &  &  &  &  &  &  &  &  &  &  \\ 
				&  &  & a & b &  &  &  &  &  &  &  &  &  &  &  &  \\ 
				&  & a &  &  &  &  &  &  &  &  &  &  &  &  &  &  \\ 
			\end{array}$$
			\caption{An example of application of the first three iterations of the substitution illustrated in \cref{ExampleTriangularThueMorse}.}
		\end{figure}
		
	\end{example}
	
	We define the subshift $X_{\zeta}$ associated with the substitution $\zeta$ as the set of all sequences $x\in \A^{\Z^{d}}$ such that every pattern occurring in $x$ is in $\mathcal{L}_{\zeta}$. 
	We call this subshift a \emph{substitutive subshift}.
	
	A substitution $\zeta$ is called \emph{primitive} if there exists a positive integer $n>0$, such that, for every $a,b\in \A$, $b$ occurs in $\zeta^{n}(a)$. Each substitution $\zeta$ can be naturally associated with an \emph{incidence matrix} denoted as $M_{\zeta}$. For any $a,b\in \A$ as $(M_{\zeta})_{a,b}$ is defined as $|\{{\bm f}\in F_{1}^{\zeta}\colon \zeta(a)_{{\bm f}} =b\}|$, i.e., it is equal to the number of occurrences of $b$ in the pattern $\zeta(a)$. The substitution $\zeta$ is primitive if and only if its incidence matrix is primitive. A matrix is primitive when it has a power with strictly positive coefficients. 
	
	If $\zeta$ is a primitive constant-shape substitution, the existence of \emph{periodic points} is well-known, i.e., there exists at least one point $x_{0}\in X_{\zeta}$ such that $\zeta^{p}(x_{0})=x_{0}$ for some $p>0$. 
	In the primitive case, the subshift is preserved by replacing the substitution with a power of it, meaning $X_{\zeta^{n}}$ is equal to $X_{\zeta}$ for any $n>0$. 
	Thus, we may assume that the substitution possesses at least one fixed point. To avoid some problems, we only keep in the alphabet the letters that appear in the fixed points. We recall that in this case, the substitutive subshift is minimal if and only if the substitution is primitive (see \cite{queffelec2010substitution}). 
	
	As in the one-dimensional case, the supports do not need to cover all the space. Nevertheless, up to adding a finite set and taking its images under powers of the expansion map $L$, they cover the space. This property is explained in the following proposition. It is similar to the notion of remainder in numeration theory and will be technically useful.
	
	\begin{proposition}\cite[Proposition 2.10]{cabezas2023homomorphisms}\label{FiniteSubsetFillsZd}
		Let $L\in \mathcal{M}(d,\Z)$ be an expansion matrix, and $F_{1}$ be a fundamental domain of $L(\Z^{d})$ in $\Z^{d}$ (containing ${\bm 0}$). Then, the set $K_{L,F_{1}}=\bigcup\limits_{m>0}((\id -L^{m})^{-1}(F_{m})\cap \Z^{d})$ is finite and satisfies	$$\bigcup\limits_{n\geq 0} L^{n}(K_{L,F_{1}})+F_{n}=\Z^{d},$$
		
		\noindent where $F_{0}^{\zeta}=\{{\bm 0}\}$.
	\end{proposition} 
	
	It is straightforward to check that for any block substitution, the set $K_{L,F_{1}}$ is equal to $\llbracket -1,0\rrbracket^{d}$. If $\zeta$ is a constant-shape substitution, we denote $K_{\zeta}=K_{L_{\zeta},F_{1}^{\zeta}}$. The set $K_{\zeta}$ controls, in some way, the number of periodic points that a constant-shape substitution has. More specifically, it can be proved that a primitive constant-shape substitution has at most $|\mathcal{L}_{K_{\zeta}}(K_{\zeta})|$ $\zeta$-periodic points. 
	
	\begin{remark}
		\label{rem:powers if needed}
		Observing that the sets $\bigcup\limits_{m =1}^n\left((\id -L^{m})^{-1}(F_{m})\cap \Z^{d}\right)$, $n \geq 1$, are nested, Proposition~\ref{FiniteSubsetFillsZd} implies that $K_{L,F_{1}}$ is equal to $\bigcup\limits_{m=1}^j((\id -L^{m})^{-1}(F_{m})\cap \Z^{d})$ for some $j >0$.
		Therefore, whenever $\zeta$ is primitive, up to replacing $\zeta$ by a power of itself, we may assume that $K_\zeta$ is equal to $(\id -L_{\zeta})^{-1}(F_{1}^{\zeta}) \cap \mathbb{Z}^d$. In the latter case, it is straightforward to prove that $\Vert K_{L,F_{1}}\Vert \leq \Vert L^{-1}(F_{1})\Vert/(1-\Vert L^{-1}\Vert)$.
	\end{remark}
	
	For the triangular Thue-Morse substitution, the set $K_{\Delta TM}$ is equal to $\{(-1,0),(0,0),(0,-1),(1,1)\}$. 
	
	The proof of \cref{FiniteSubsetFillsZd} is inspired by the \emph{Euclidean Division algorithm}, which was used to obtain finite sets satisfying particular properties as shown in the following result that we will use in the rest of this article.
	
	\begin{proposition}\cite[Proposition 2.12]{cabezas2023homomorphisms}\label{Prop:FiniteSetSatisfyingParticularProperties} Set $A\Subset \Z^{d}$ containing $\bm 0\in \Z^{d}$ and let $F\Subset \Z^{d}$ containing a fundamental domain $F_{1}$ of an integer expansion matrix $L$. Then, there exists a (computable) finite subset $C \subseteq \Z^{d}$ containing $\mathbf{0}$ and such that
		\begin{enumerate}[label=(\arabic*),ref=\text{(}\arabic*\text{)}]
			\item $A+F \subseteq C+A+F\subseteq L(C)+F_{1}$.
			\item \label{item:MoreGeneralresultFortheSet}More generally, for any $n\geq 0$
			
			\begin{itemize}
				\item $L^{n}(C+A+F)+F_{n}\subseteq L^{n+1}(C)+F_{n+1}$
				\item $C+\sum\limits_{i=0}^{n}L^{i}(A+F)\subseteq L^{n+1}(C)+F_{n+1}$.
			\end{itemize} 
			
			\item \label{item:nestedB_n}
			The sequence of sets $(L^{n}(C)+F_{n})_{n\geq 0}$ is nested.
			\item\label{item:BoundForTheNormforthefiniteset} $\Vert C\Vert\leq (\Vert L^{-1}(A+F)\Vert+\Vert L^{-1}(F_{1})\Vert)/(1-\Vert L^{-1}\Vert)$.
		\end{enumerate} 
	\end{proposition}
	\begin{proof}
		We define the sequence $(C_n)_{n \geq 0}$ of finite sets by $C_0 = \{\mathbf{0}\}$ and, for every $n \geq 0$,
		\[
		C_{n+1} = [L^{-1}(C_n+A+F-F_1)\cap \Z^{d}]. 
		\]
		One easily checks by induction that $C_n \subseteq C_{n+1}$ for all $n\geq 0$ and a quick computation shows that
		\[
		\| C_n\| \leq \frac{\Vert L^{-1}(A+F)\Vert+\Vert L^{-1}(F_{1}) \Vert}{1-\Vert L^{-1}\Vert }.
		\]
		As a consequence, the sequence $(C_n)_{n \geq 0}$ stabilizes and we set $C = C_N$ where $N$ is such that $C_n = C_N$ for every $n \geq N$. This set $C$ is obviously computable.
		
		Let us now check that the set $C$ satisfies all items.
		The set $C$ contains $\mathbf{0}$ by construction, which directly implies that $A+F \subseteq C+A+F$.
		If ${\bm n} \in C+A+F$, then ${\bm n}$ belongs to $C_n +A + F$ for some $n\in\N$.
		We can thus write ${\bm n} = {\bm n}' = {\bm a} + {\bm f}$ as well as ${\bm n} = L({\bm m}) + {\bm f}'$, for some ${\bm n}' \in B_n, {\bm a} \in A, {\bm f} \in F$ and ${\bm f}' \in F_1$.
		We deduce that $L({\bm m}) \in C_n + A + F - F_1$, hence ${\bm m} \in C$.
		Item~\ref{item:MoreGeneralresultFortheSet} follows by induction and implies Item~\ref{item:nestedB_n}, as $\mathbf{0} \in A \cap F$.
	\end{proof}
	
	\begin{remark}\label{rem:GoodUniformlyBounded}
		Note that if we change the pair $(L,F_{1})$ by $(L^{n},F_{n})$ for any $n\geq 1$, then the set $C$ given by \cref{Prop:FiniteSetSatisfyingParticularProperties} is the same for fixed $A$ and $F$. Using the notion of \emph{digit tile} defined in Section 2.7 of \cite{cabezas2023homomorphisms}, we note that $\Vert L^{-n}(F_{n})\Vert \xrightarrow[n\to \infty]\ \Vert T_{(L,F_{1})}\Vert$. Hence, the sequence $(\Vert L^{-n}(F_{n})\Vert/(1-\Vert L^{-n}\Vert))_{n\geq 1}$ is uniformly bounded. Note that, in the block case, these quantities are bounded by $\sqrt{d}$, where $d$ is the dimension of the substitution.
	\end{remark}
	
	From now on, we denote $C_{L,F_{1}}$ to the set given by \cref{rem:GoodUniformlyBounded} using $A=\{{\bm 0}\}$, and $F=F_{1}+F_{1}$. By \cite[Remark 2.13 (2)]{cabezas2023homomorphisms}, we have that for any $n\in \N$, $C_{L,F_{1}}+F_{n}+F_{n}\subseteq L^{n}(C_{L,F_{1}})+F_{n}$. Note that, in the block case, the set $C_{L,F_{1}}$ is equal to $\llbracket 0,1\rrbracket^{d}$.
	
	Every element of $\Z^{d}$ can be expressed in a unique way as ${\bm p}= L({\bm j})+{\bm f}$, with ${\bm j} \in \Z^{d}$ and ${\bm f}\in F_{1}$, so we can consider the substitution $\zeta$ as a map from $X_{\zeta}$ to itself given by
	$$\zeta(x)_{L({\bm j})+{\bm f}}=\zeta(x(\bm{j}))_{{\bm f}}.$$
	
	This map is continuous. Moreover, when the substitution is aperiodic and primitive, \cref{RecognizabilitySecondStep} below ensures that this map is actually a homeomorphism. This property is satisfied, even in the case where the substitution is not \emph{injective on letters}, i.e., when there exist distinct letters $a,b\in \A$ such that $\zeta(a)=\zeta(b)$. This comes from the notion of \emph{recognizability} of a substitution (see \cref{Section:ComputabilityRecognizability}). 
	
	\begin{definition}
		Let $\zeta$ be a substitution and $x\in X_{\zeta}$ be a fixed point. We say that $\zeta$ is \emph{recognizable on $x$} if there exists some constant $R>0$ such that for all ${\bm i}, {\bm j}\in \Z^{d}$,
		$$x|_{[B(L_{\zeta}({\bm i}),R)\cap \Z^{d}]}=x|_{[B({\bm j},R)\cap \Z^{d}]} \implies (\exists {\bm k}\in \Z^{d}) (({\bm j}=L_{\zeta}({\bm k}))\wedge (x_{{\bm i}}=x_{{\bm k}})).$$
	\end{definition}
	
	The recognizability of a substitution $\zeta$ implies that for every $x\in X_{\zeta}$, there exist a unique $x'\in X_{\zeta}$ and a unique ${\bm j} \in F_{1}^{\zeta}$ such that $x=S^{{\bm j}}\zeta(x')$. 
	This implies that the set $\zeta(X_{\zeta})$ is a clopen subset of $X_{\zeta}$, and $\{S^{{\bm j}}\zeta(X_{\zeta})\colon\ {\bm j} \in F_{1}^{\zeta}\}$ forms a clopen partition of $X_{\zeta}$ (The proof is classical and similar to the one-dimensional case~\cite[Section 5.6]{queffelec2010substitution}). Any power of a recognizable substitution is also recognizable, so these properties extend to $\zeta^{n}$, for all $n>0$. The recognizability property was first proved for any aperiodic primitive substitution by B. Moss\'e in the one-dimensional case \cite{mosse1992puissances}, and in the multidimensional case by B. Solomyak \cite{solomyakrecognizability} for aperiodic self-affine tilings with an $\R^{d}$-action. Later, in \cite{cabezas2023homomorphisms} it was established that the aperiodic symbolic factors of primitive substitutive subshifts also satisfy a recognizability property.
	
	\section{Decidability of the F{\o}lner property for fundamental domains of an expansion matrix and computability of the language of constant-shape substitutions}\label{Section:DecidabilityFOlner}
	
	Let $L\in \mathcal{M}(d,\Z)$ be an expansion matrix and $F_{1}\Subset \Z^{d}$ be a fundamental domain of $L(\Z^{d})$ in $\Z^{d}$ containing ${\bm 0}$. 
	Define the sequence of fundamental domains of $L^{n}(\Z^{d})$ as in \eqref{Eq:RecurrenceSets}:
	$$F_{n+1}=L(F_{n})+F_{1},\ \forall n\geq 1.$$
	
	We recall that the sequence $(F_{n})_{n\in \N}$ is said to be F{\o}lner if for any ${\bm n}\in \Z^{d}$,
	\[\lim\limits_{n\to\infty}\dfrac{|F_{n}\Delta ({\bm n}+F_{n})|}{|F_{n}|}=0.
	\]
	As mentioned in the previous section, the F{\o}lner property ensures that the sets $F_n$ eventually contain balls of arbitrarily large radius, which then allows us to define a subshift. 
	However, the converse is not true in general, i.e., having balls of arbitrarily large radius is not enough to guarantee that the sequence is F{\o}lner, since the sets can be sparse at the same time, as the following example shows.
	
	\begin{example} Consider the sequence of finite sets given by $A_{n}=\llbracket 0,n\rrbracket \cup \{-k(k+3)/2\colon 0\leq k \leq n\}$, for any $n\in \N$. Set $r>2$. Let $n\in \N$ be large enough. We note that
		$$(A_{n}+r)\cap A_{n}=\llbracket r, n\rrbracket \cup \left\{r-\frac{k(k+3)}{2} \colon 0\leq k\leq \frac{\sqrt{8r+9}-3}{2}\right\} \cup \left\{-\frac{(r-2)(r+1)}{2}\right\},$$
		
		\noindent which let us conclude that $|(A_{n}+r)\Delta A_{n}|/|A_{n}|\to 1/2$ as $n\to\infty$. Hence, the sequence $(A_{n})_{n\in \N}$ is not F{\o}lner.
		
	\end{example}
	
	In this section, we prove \cref{thm:decidabilityfolnerporperty}, stating that, when the sets $F_n$ are built as in \eqref{Eq:RecurrenceSets}, the F\o lner property is equivalent to the existence of arbitrarily large balls in the sets $F_n$, for large enough $n$. Moreover, we show that we have some control over the index of the sequence, to check an equivalent property, which implies that being F\o lner is decidable.
	\begin{theorem}\label{thm:decidabilityfolnerporperty}
		Let $L\in \mathcal{M}(d,\Z)$ be an expansion matrix and $F_{1}\Subset \Z^{d}$ be a fundamental domain of $L(\Z^{d})$ in $\Z^{d}$ containing ${\bm 0}$. Set $\bar{r} = \|L^{-1}(F_1)\|/(1-\|L^{-1}\|)$.
		The sequence $(F_{n})_{n\in\N}$ defined as \eqref{Eq:RecurrenceSets} is F{\o}lner if and only if there exists $n \leq \frac{(6\bar{r})^{3d}-(6\bar{r})^d}{6}$ such that $F_n$ contains a translation of $C_{L,F_1}+[B({\bm 0},\bar{r})\cap \Z^{d}]$.
		In particular, it is decidable to check whether $(F_{n})_{n\in\N}$ is a F{\o}lner sequence. 
	\end{theorem}
	
	In the block case, i.e., $L=\diag(\ell_{i})_{i=1}^{d}$ and $F_{1}=\prod\limits_{i=1}^{d}\llbracket 0,l_{i}-1\rrbracket$, we note that $C_{L,F_{1}}+[B({\bm 0},\bar{r})\cap \Z^{d}]\subseteq [B({\bm 0},2\sqrt{d})\cap \Z^{d}]$, hence for any $n\in \N$ such that $n\geq \log(4\sqrt{d})/\log(\min \ell_{i})$, contains a translation of $C_{L,F_{1}}+[B({\bm 0},\bar{r})\cap \Z^{d}]$.
	
	Roughly speaking, the idea of the proof of \cref{thm:decidabilityfolnerporperty} is twofold. First, we show that if some $F_N$ contains a large enough ball $B$, then the ``extended image'' $L^m(B)+F_m$ inside $F_{N+m}$ will contain a ball larger than $B$ and this will enforce the F{\o}lner property. 
	The statement of \cref{lemma:EquivalenceFolner} below is actually more precise about the ball $B$, which allows us get a characterization of the F\o lner property. \cref{fig:ballsbecomingbigger} illustrates this idea.
	In the second step (\cref{prop:equivalenceforfolnerproperty}), we define a finite graph and translate the condition of \cref{lemma:EquivalenceFolner} into the existence of a path in this graph. This leads to the decidability of the F\o lner property and to a bound on the index $N$. \cref{automatonforthuemorsetriangular} is an example of the graphs for the triangular Thue-Morse substitution.
	
	Consider the set $K\Subset \Z^{d}$ given by \cref{FiniteSubsetFillsZd}, and the set $C_{L,F_{1}}\Subset \Z^{d}$ given by \cref{Prop:FiniteSetSatisfyingParticularProperties} with $A=\{{\bm 0}\}$ and $F=F_{1}+F_{1}$. 
	We recall that, by \cref{item:MoreGeneralresultFortheSet}, for any $n\geq 1, C_{L,F_{1}}+F_{n}+F_{n}\subseteq L^n(C_{L,F_{1}})+F_n$. 
	We also deduce from \cref{item:BoundForTheNormforthefiniteset} that 
	\begin{equation}\label{eq:boundonC}
		\|C_{L,F_1}\| \leq 2\bar{r}.
	\end{equation}
	Assume, up to replacing $L$ by an appropriate power of it, that $K=(\id-L)^{-1}(F_{1})\cap \Z^{d}$. 
	The following result shows a characterization, for the sequence of fundamental domains $(F_{n})_{n\in \N}$ to be F{\o}lner.

	\begin{proposition}\label{lemma:EquivalenceFolner}
		Let $L\in \mathcal{M}(d,\Z)$ be an expansion matrix, $F_{1}\Subset \Z^{d}$ a fundamental domain of $L(\Z^{d})$ in $\Z^{d}$ containing ${\bm 0}$ and $(F_{n})_{n\in\N}$ be the sequence of fundamental domains defined as \eqref{Eq:RecurrenceSets}. The following conditions are equivalent
		\begin{enumerate}
			\item \label{item:folner} the sequence $(F_{n})_{n\in \N}$ is F{\o}lner;
			\item \label{item:translationofC+K}there exists $n \in \mathbb{N}$ and ${\bm f}\in F_{n}$ such that 
			\[
			{\bm f}+C_{L,F_{1}}+K\subseteq F_{n};
			\]
			\item \label{item:translationofC+B} for every finite set $B \supseteq K$, there exists $m \in \mathbb{N}$, ${\bm f}\in F_{m}$ such that 
			\[
			{\bm f}+C_{L,F_{1}}+B\subseteq F_{m};
			\]
		\end{enumerate}
	\end{proposition}
	
	Note that \cref{item:translationofC+B} implies that, it is equivalent for a substitution $\zeta$ to define the substitutive subshift $X_{\zeta}$ to the sequence $(F_{n})_{n>0}$ is F{\o}lner.
	
	\begin{proof}
		The implications $\ref{item:folner} \Rightarrow \ref{item:translationofC+K}$, $\ref{item:folner} \Rightarrow \ref{item:translationofC+B}$ and $\ref{item:translationofC+B} \Rightarrow \ref{item:translationofC+K}$ are direct. 
		Let us thus assume \cref{item:translationofC+K}. We recall that, by \cref{FiniteSubsetFillsZd} $\bigcup\limits_{p\in\N}L^{p}(K)+F_{p}=\Z^{d}$, and the sequence of finite sets $(L^{p}(K)+F_{p})_{p\in \N}$ is nested. To check that $(F_{n})_{n\in \N}$ is F{\o}lner, it is enough to prove that
		\[
		(\forall p\in \N)\ \lim\limits_{n\to\infty}\dfrac{|F_{n}\Delta (F_{n}+C_{L,F_{1}}+L^{p}(K)+F_{p})|}{|F_{n}|}=0,
		\]
		or, equivalently, that
		\begin{equation}\label{eq:limitofb_n}
			(\forall p\in \N)\ \lim\limits_{n\to\infty}
			\frac{|\{{\bm f}\in F_{n}\colon {\bm f}+C_{L,F_1}+L^{p}(K)+F_{p} \subseteq F_n\}|}{|F_n|} = 1.
		\end{equation}
		For any $n\in\N$ and $p\in\N$ we set 
		\[
		J_{n,p}=\{{\bm f}\in F_{n}\colon {\bm f}+C_{L,F_1}+L^{p}(K)+F_{p} \subseteq F_n\},
		\]
		$a_{n,p}=|J_{n,p}|$ and $b_{n,p}=a_{n,p}/|F_{n}|$.
		By assumption, there exists $N \in \mathbb{N}$ and ${\bm f} \in F_N$ such that ${\bm f}+C_{L,F_{1}}+K\subseteq F_{N}$. In other words, we have $a_{N,0} \geq 1$.
		\cref{Claim:claimforhteiteratesofK} below shows a form of stability of elements of $J_{n,p}$ under the application of $L$. In particular, it implies that $a_{N+p,p} \geq 1$ for all $p \in \mathbb{N}$.
		Since $\bigcup\limits_{p\in\N}L^{p}(K)+F_{p}=\Z^{d}$ and the sequence of finite sets $(L^{p}(K)+F_{p})_{p\in \N}$ is nested, it also implies \cref{item:translationofC+B}.
		
		\begin{claim}\label{Claim:claimforhteiteratesofK}
			For all $n,p \in \mathbb{N}$ and all ${\bm f} \in F_n$, if ${\bm f}+C_{L,F_{1}}+L^p(K)+F_p\subseteq F_{n}$, then all elements ${\bm g} \in L({\bm f})+F_1 \subseteq F_{n+1}$ are such that ${\bm g}+C_{L,F_{1}}+L^{p+1}(K)+F_{p+1}\subseteq F_{n+1}$.
			In particular, we have
			\[
			a_{n+1,p} \geq a_{n+1,p+1} \geq a_{n,p} |\det(L)|.
			\]
		\end{claim}
		
		\begin{proof}[Proof of \cref{Claim:claimforhteiteratesofK}]     
			Indeed, since $C_{L,F_{1}}+F_1+F_1\subseteq L(C_{L,F_{1}})+F_1$, we have
			\begin{align*}
				{\bm f}+C_{L,F_{1}}+L^p(K)+F_p\subseteq F_{n}& \implies L({\bm f})+(L(C_{L,F_{1}})+F_{1})+L^{p+1}(K)+L(F_p)\subseteq F_{n+1}  \\
				& \implies (L({\bm f})+F_{1})+C_{L,F_{1}}+L^{p+1}(K)+F_{p+1}\subseteq F_{n+1}.
			\end{align*}
			This implies that $a_{n+1,p+1} \geq a_{n,p} |\det(L)|$.                        The other inequality follows from the fact that $L^p(K)+F_p \subseteq L^{p+1}(K)+F_{p+1}$.
		\end{proof}
		
		\cref{fig:ballsbecomingbigger} illustrates \cref{Claim:claimforhteiteratesofK}. Note that \cref{Claim:claimforhteiteratesofK} proves that \cref{item:translationofC+K} implies \cref{item:translationofC+B}. 
		
		\begin{figure}[htp]
			\includegraphics[scale=0.66]{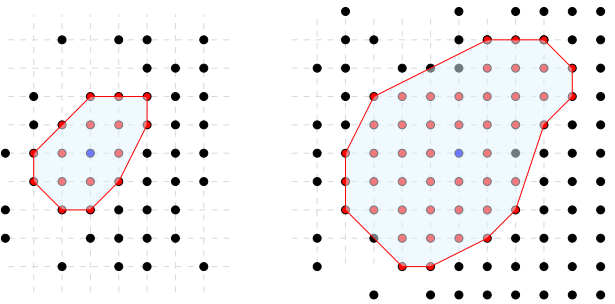}
			\caption{A visual representation of \cref{Claim:claimforhteiteratesofK} with the triangular Thue-Morse substitution (\cref{ExampleTriangularThueMorse}). The black points at the left represent the points in $F_{4}$, and $F_{5}$ at the right. The blue points represent the points $(-2,5)$ in $F_{4}$ (left) and $(-4,10)=2\cdot (-2,5)+(0,0)$ in $F_{5}$ (right). Finally, the red points at the left represent the set $(-2,5)+C_{L,F_{1}}+K$ and $(-4,10)+C_{L,F_{1}}+L(K)+F_{1}$ at the right.}
			\label{fig:ballsbecomingbigger}
		\end{figure}
		
		Since $|F_{n+1}| = |F_n||\det(L)|$, we deduce from \cref{Claim:claimforhteiteratesofK} that the sequence $(b_{n,p})_n$ is non-decreasing for all $p$. Furthermore, it is also bounded from above by 1, and the limit~\eqref{eq:limitofb_n} exists.
		
		For every $p \in \mathbb{N}$, we set $m(p)=\inf\{n \mid a_{n,p}>0\}$. By hypothesis and \cref{Claim:claimforhteiteratesofK} we note that $m(p)\leq N+p$. Recall that for any $k\geq 1$, any element in $F_{k\cdot m(p)}$ can be written as $\sum\limits_{i=0}^{k-1}L^{i\cdot m(p)}({\bm f}_{i})$, with ${\bm f}_{i}\in F_{m(p)}$.
		
		\begin{claim}\label{CharacterizationofSetsJ}
			If there exists $0\leq i\leq k-1$ such that ${\bm f}_{i}\in J_{m(p),p}$, then $\left(\sum\limits_{i=0}^{k-1}L^{i \cdot m(p)}({\bm f}_{i})\right)\in J_{k\cdot m(p),p}$.
		\end{claim}
		
		Using the notion of \emph{invariance} introduced by B. Weiss in \cite{weiss2001monotileable}, \cref{CharacterizationofSetsJ} proves that if $F_{n}$ is $(C+L^{p}(K)+F_{p},\varepsilon)$-invariant, then for any $k>0$, the set $F_{kn}$ is $(C+L^{p}(K)+F_{p},\varepsilon^{k})$-invariant, which will let us conclude that the sequence $(F_{n})_{n>0}$ is F{\o}lner.
		
		\begin{proof}[Proof of \cref{CharacterizationofSetsJ}]
			First, we note that $m(p)\geq p$. For any ${\bm c}\in C_{L,F_{1}}$ and ${\bm n}\in L^{p}(K)+F_{p}$, we need to find $({\bm g}_{i})_{i=0}^{k-1}\subseteq F_{m(p)}$ such that
			$$\left(\sum\limits_{i=0}^{k-1}L^{i\cdot m(p)}({\bm f}_{i})\right)+{\bm c}+{\bm n}=\left(\sum\limits_{i=0}^{k-1}L^{i\cdot m(p)}({\bm g}_{i})\right).$$
			
			Let $0\leq j\leq k-1$ be the minimal such that ${\bm f}_{j}\in J_{m(p),p}$. If $j=0$, then the result follows from \cref{CharacterizationofSetsJ}. 
			Assume that $j>0$. 
			Since $L^p(K)+F_p \subseteq L^{j \cdot m(p)}(K)+F_{j \cdot m(p)}$, there exist ${\bm k}\in K$ and $({\bm h}_{i})_{i=0}^{j-1}\subseteq F_{m(p)}$ such that
			\[
			{\bm n}=L^{j\cdot m(p)}({\bm k})+\left(\sum\limits_{i=0}^{j-1}L^{i\cdot m(p)}({\bm h}_{i})\right).
			\]
			
			Then, using items~\ref{item:MoreGeneralresultFortheSet} and~\ref{item:nestedB_n} of \cref{Prop:FiniteSetSatisfyingParticularProperties}, there exists ${\bm c}_{1}\in C_{L,F_{1}}$ and $({\bm g}_{i})_{i=0}^{j-1}\subseteq F_{m(p)}$ such that
			$${\bm c}+\left(\sum\limits_{i=0}^{j-1}L^{i\cdot m(p)}({\bm f}_{i})\right)+\left(\sum\limits_{i=0}^{j-1}L^{i\cdot m(p)}({\bm h}_{i})\right)=L^{j\cdot m(p)}({\bm c}_{1})+\sum\limits_{i=0}^{j-1}L^{i\cdot m(p)}({\bm g}_{i})$$
			
			Hence
			$$\left(\sum\limits_{i=0}^{k-1}L^{i\cdot m(p)}({\bm f}_{i})\right)+{\bm c}+{\bm n}=\left(\sum\limits_{i=0}^{j-1}L^{i\cdot m(p)}({\bm g}_{i})\right)+\left(\sum\limits_{i=j+1}^{k-1}L^{i\cdot m(p)}({\bm f}_{i})\right)+L^{j\cdot m(p)}({\bm f}_{j}+{\bm c}_{1}+{\bm k}).$$
			
			We conclude noting that ${\bm f}_j+{\bm c}_{1}+{\bm k}\in F_{m(p)}$.
		\end{proof}
		
		Now, \cref{CharacterizationofSetsJ} implies that
		$$|F_{k\cdot m(p)}\Delta (F_{k\cdot m(p)}+C+L^{p}(K)+F_{p})|\subseteq \left\{\sum\limits_{i=0}^{k-1}L^{i \cdot m(p)}({\bm f}_{i})\in F_{k\cdot m(p)}\colon (\forall i)\ {\bm f}_{i}\notin J_{m(p),p}\right\}.$$
		
		Hence, for any $p\in\N$
		\begin{align*}
			\lim\limits_{k\to\infty}b_{k,p} &\geq \lim\limits_{k\to\infty}b_{k\cdot m(p),p} \\
			& \geq 1-\lim\limits_{k\to\infty}\dfrac{\left|\left\{\sum\limits_{i=0}^{k-1}L^{i}({\bm f}_{i})\in F_{k\cdot m(p)}\colon (\forall\ 0\leq i\leq k-1){\bm f}_{i}\notin J_{m(p),p}\right\}\right|}{|F_{k\cdot m(p)}|}\\
			& =1-\lim\limits_{k\to\infty}\dfrac{(|\det(L)|^{m(p)}-|J_{m(p),p}|)^{k}}{|\det(L)|^{k\cdot m(p)}}\\
			& = 1-\lim\limits_{k\to\infty}\left(1-\dfrac{|J_{m(p),p}|}{|\det(L)|^{m(p)}}\right)^{k}\\
			& = 1
		\end{align*}
		
		We conclude that the sequence $(F_{n})_{n>0}$ is F{\o}lner.
	\end{proof}
	
	We now aim to show that satisfying \cref{item:translationofC+K} of \cref{lemma:EquivalenceFolner} is decidable. The next result shows that this condition is equivalent to the existence of a \emph{synchronizing word} in some particular finite graph, a condition which is known to be decidable~\cite{Volkov2008}.
	
	Assume that $B \Subset \mathbb{Z}^d$ satisfies $K \subseteq B \subseteq L(B)+F_1$. 
	Let us consider a labeled directed graph $\mathcal{G}(L,F_{1},B)$, where the vertex set is $C_{L,F_{1}}+B$ and there is an edge from ${\bm a}\in C_{L,F_{1}}+B$ to ${\bm b}\in C_{L,F_{1}}+B$ labeled with ${\bm f}\in F_{1}$ if and only if there exists ${\bm g}\in F_{1}$ such that ${\bm f}+{\bm a}=L({\bm b})+{\bm g}$. \cref{fig:automatonfor1d} shows the graph $\mathcal{G}(L,F_{1},B)$ in the classical one-dimensional case, where $L=\ell\geq2$ and $F_{1}=\llbracket 0,\ell-1\rrbracket$.
	
	\begin{figure}[htp]
		\centering
		\begin{tikzpicture}[auto,scale=0.7]
			\node (s0) [state]  at (0,0) {$0$};
			\node (s1) [state] at (5,0) {$1$};
			\node (s7) [state] at (-5,0) {$-1$};
			
			\path [-stealth, thick]
			(s0) edge [in=80,out=110,loop] node[above] {$\llbracket \texttt{0},\ell-1\rrbracket$} (s0)
			(s1) edge [in=350,out=380,loop] node [right] {$\ell-1$} (s1)
			(s1) edge node[midway,sloped] {$\llbracket \texttt{0},\ell-2\rrbracket$} (s0)
			(s7) edge [loop left] node [left] {$\texttt{0}$} (s7)
			(s7) edge node[midway,sloped] {$\llbracket \texttt{1},\ell-1\rrbracket$} (s0);
		\end{tikzpicture}
		\caption{The graph $\mathcal{G}(\ell,\llbracket 0,\ell-1\rrbracket,\{-1,0\})$ for the classical one-dimensional case.}
		\label{fig:automatonfor1d}
	\end{figure}

	\cref{automatonforthuemorsetriangular} represents the graph $\mathcal{G}(L,F_{1},K)$ for the triangular Thue-Morse (\cref{ExampleTriangularThueMorse}).

	\begin{figure}[htp]
		\centering
		\begin{tikzpicture}[auto,scale=0.7]
			\node (s0) [state]  at (0,0) {$(0,0)$};
			\node (s1) [state] at (6,0) {$(1,2)$};
			\node (s2) [state] at (5.1962,3) {$(1,1)$};
			\node (s3) [state] at (3,5.1962) {$(2,1)$};
			\node (s4) [state] at (0,6) {$(1,0)$};
			\node (s5) [state] at (-3,5.1962) {$(1,-1)$};
			\node (s6) [state] at (-5.1962,3) {$(0,-1)$};
			\node (s7) [state] at (-6,0) {$(-1,-2)$};
			\node (s8) [state] at (-5.1962,-3) {$(-1,-1)$};
			\node (s9) [state] at (-3,-5.1962) {$(-2,-1)$};
			\node (s10) [state] at (0,-6) {$(-1,0)$};
			\node (s11) [state] at (3,-5.1962) {$(-1,1)$};
			\node (s12) [state] at (5.1962,-3) {$(0,1)$};
			\node (s13) [state] at (8.6603,5) {$(2,2)$};
			\node (s14) [state] at (-8.6603,5) {$(0,2)$};
			\node (s15) [state] at (0,-10) {$(-2,0)$};
			
			\path [-stealth, thick]
			(s1) edge [bend right] node[below,sloped]  {$(0,0)$} (s12)
			(s1) edge [bend right] node[sloped,below] {$(1,0)$} (s2)
			(s1) edge [in=350,out=380,loop] node [sloped,above] {$(0,1)$} (s1)
			(s1) edge node[midway,sloped] {$(-1,-1)$} (s0)
			
			(s2) edge [in=50,out=80,loop] node[sloped] {$(0,0)$} (s2)
			(s2) edge [bend right=60,distance=4.2cm] node [above,sloped] {$(1,0)$} (s4)
			(s2) edge [bend right] node [sloped] {$(0,1)$} (s12)
			(s2) edge node[midway,sloped] {$(-1,-1)$} (s0)
			
			(s3) edge [bend right] node [above,sloped] {$(0,0)$} (s4)
			(s3) edge [in=50,out=80,loop] node[sloped] {$(1,0)$} (s3)
			(s3) edge [bend right] node[sloped] {$(0,1)$} (s2)
			(s3) edge node[midway,sloped] {$(-1,-1)$} (s0)
			
			(s4) edge node[midway,sloped] {$(0,0)$} (s0)
			(s4) edge [in=60,out=90,loop] node[sloped] {$(1,0)$} (s4)
			(s4) edge [bend right] node [above,sloped] {$(0,1)$} (s2)
			(s4) edge [bend right=60,distance=4.2cm] node [above,sloped] {$(-1,-1)$} (s6)
			
			(s5) edge [bend right] node [above,sloped] {$(0,0)$} (s4)
			(s5) edge [in=75,out=105,loop] node[sloped] {$(1,0)$} (s5)
			(s5) edge node[midway,sloped] {$(0,1)$} (s0)
			(s5) edge [bend right] node [below,sloped] {$(-1,-1)$} (s6)
			
			(s6) edge [in=90,out=120,loop] node[sloped] {$(0,0)$} (s6)
			(s6) edge [bend right] node [sloped] {$(1,0)$} (s4)
			(s6) edge node[midway,sloped] {$(0,1)$} (s0)
			(s6) edge [bend right=60,distance=4.2cm] node [sloped] {$(-1,-1)$} (s8)
			
			(s7) edge [bend right] node [sloped] {$(0,0)$} (s8)
			(s7) edge [bend right] node [sloped] {$(1,0)$} (s6)
			(s7) edge node[midway,sloped] {$(0,1)$} (s0)
			(s7) edge [loop left] node [above=-.1,sloped] {$(-1,-1)$} (s7)
			
			(s8) edge node[midway,sloped] {$(0,0)$} (s0)
			(s8) edge [bend right] node [sloped] {$(1,0)$} (s6)
			(s8) edge [bend right=60,distance=4.2cm] node [sloped,below] {$(0,1)$} (s10)
			(s8) edge [in=199,out=229,loop] node[below,sloped] {$(-1,-1)$} (s8)
			
			(s9) edge [bend right] node [sloped,below] {$(0,0)$} (s8)
			(s9) edge node[midway,sloped] {$(1,0)$} (s0)
			(s9) edge [bend right] node [above,sloped] {$(0,1)$} (s10)
			(s9) edge [in=230,out=260,loop] node [below=-.2,sloped] {$(-1,-1)$} (s9)
			
			(s10) edge [in=290,out=320,loop] node [below,sloped] {$(0,0)$} (s10)
			(s10) edge node[midway,sloped] {$(1,0)$} (s0)
			(s10) edge [bend right=60,distance=4.2cm] node [below,sloped] {$(0,1)$} (s12)
			(s10) edge [bend right] node [sloped] {$(-1,-1)$} (s8)
			
			(s11) edge [bend right] node [sloped] {$(0,0)$} (s12)
			(s11) edge node[midway,sloped] {$(1,0)$} (s0)
			(s11) edge [in=260,out=290,loop] node[below=-.2,sloped] {$(0,1)$} (s11)
			(s11) edge [bend right] node [below,sloped] {$(-1,-1)$} (s10)
			
			(s12) edge node[midway,sloped] {$(0,0)$} (s0)
			(s12) edge [bend right=60,distance=4.2cm] node [sloped,below] {$(1,0)$} (s2)
			(s12) edge [in=290,out=320,loop] node[below,sloped] {$(0,1)$} (s12)
			(s12) edge [bend right] node [below,sloped] {$(-1,-1)$} (s10)
			(s13) edge node {$F_{1}$} (s2)
			(s14) edge node {$F_{1}$} (s6)
			(s15) edge node {$F_{1}$} (s10);
		\end{tikzpicture}
		\caption{The graph $\mathcal{G}(L,F_{1},K)$ for the triangular Thue-Morse. Outgoing edges from vertex $(0,0)$ are not represented as they are self-loops.}
		\label{automatonforthuemorsetriangular}
	\end{figure}
	
	Note that ${\bm 0}\in C_{L,F_{1}}+B$ has out-degree $|F_{1}|=|\det(L)|$ and any edge from ${\bm 0}$ is a self-loop. Furthermore, if ${\bm a}\in C_{L,F_{1}}+B$ is an \emph{in-neighbor} of ${\bm 0}$ with an edge labeled by ${\bm f}\in F_{1}$, then ${\bm f}+{\bm a}\in F_{1}$. Let $P={\bm a}_{0}\xrightarrow[]{{\bm f}_{0}}\ {\bm a}_{1} \xrightarrow[]{{\bm f}_{1}}\ {\bm a}_{2}$ be a path in $\mathcal{G}(L,F_{1},B)$. By definition, we have that ${\bm f}_{0}+{\bm a}_{0}=L({\bm a}_{1})+{\bm g}_{0}$ and ${\bm f}_{1}+{\bm a}_{1}=L({\bm a}_{2})+{\bm g}_{1}$, for some ${\bm g}_{0}, {\bm g}_{1}\in F_{1}$. This implies that $L({\bm f}_{1})+{\bm f}_{0}+{\bm a}_{0}=L^{2}({\bm a}_{2})+L({\bm g}_{1})+{\bm g}_{0}$.
	
	\begin{proposition}\label{prop:equivalenceforfolnerproperty}
		Let $n \geq 1$ and ${\bm f} = \sum\limits_{i=0}^{n-1}L^{i}({\bm f}_{i}) \in F_{n}$, with ${\bm f}_{i} \in F_1$ for every $i$. 
		We have that ${\bm f}+C_{L,F_{1}}+B\subseteq F_{n}$ if and only if ${\bm f}_{0}{\bm f}_{1} \cdots {\bm f}_{n-1}$ labels a path from every vertex in $\mathcal{G}(L,F_{1},B)$ to ${\bm 0}$.
	\end{proposition}
	
	\begin{proof}
		Assume that ${\bm f}+C_{L,F_{1}}+B\subseteq F_{n}$. Set ${\bm a}\in C_{L,F_{1}}+B$. Then, there exists ${\bm g}({\bm a})=\sum\limits_{i=0}^{n-1}L^{i}({\bm g}_{i}({\bm a}))\in F_{n}$ such that ${\bm f}+{\bm a}={\bm g}({\bm a})$. Note that, there exists ${\bm a}_{1}\in C_{L,F_{1}}+B$ such that ${\bm f}_{0}+{\bm a}=L({\bm a}_{1})+{\bm g}_{0}({\bm a})$, and a straightforward induction shows that there exists a sequence $({\bm a}_{i})_{i=1}^{n-1}$ such that for any $1\leq i\leq n-1$, ${\bm f}_{i}+{\bm a}_{i}=L({\bm a}_{i+1})+{\bm g}_{i}({\bm a})$ and ${\bm f}_{n-1}+{\bm a}_{n-1}={\bm g}_{n-1}({\bm a})$. This implies that, there is a path $P_{{\bm a}}={\bm a}\xrightarrow[]{{\bm f}_{0}}\ {\bm a}_{1} \xrightarrow[]{{\bm f}_{1}}\ \cdots \xrightarrow[]{{\bm f}_{n-2}}\ {\bm a}_{n-1} \xrightarrow[]{{\bm f}_{n-1}}\ {\bm 0}$. All of these paths have the same label. The other direction is direct.
	\end{proof}
	
	The next result is a direct consequence of \cref{lemma:EquivalenceFolner} and \cref{prop:equivalenceforfolnerproperty}. 
	
	\begin{corollary}
		\label{cor:equivalenceFolner}
		Let $L\in \mathcal{M}(d,\Z)$ be an expansion matrix, $F_{1}\Subset \Z^{d}$ a fundamental domain of $L(\Z^{d})$ in $\Z^{d}$ containing ${\bm 0}$ and $(F_{n})_{n\in\N}$ be the sequence of fundamental domains defined as \eqref{Eq:RecurrenceSets}. 
		Let also $B \Subset \mathbb{Z}^d$ satisfying $K \subseteq B \subseteq L(B)+F_1$.
		The sequence $(F_{n})_{n\in \N}$ is F{\o}lner if and only if there exists $n \geq 1$ and a word ${\bm f}_{0}{\bm f}_{1} \cdots {\bm f}_{n-1} \in F_1^*$ that labels a path from any vertex to ${\bm 0}$ in $\mathcal{G}(L,F_{1},B)$.
	\end{corollary}
	
	Note that in \cref{fig:automatonfor1d}, if $\ell\geq3$, the letter $\texttt{1}$ labels paths from every vertex in $\{-1,0,1\}$ to $0$. If $\ell=2$, the word $\texttt{01}$ labels paths from every vertex in $\{-1,0,1\}$ to $0$. Observe that in \cref{automatonforthuemorsetriangular}, the word $(0,1)(1,0)(-1,-1)(0,1)$ labels paths from every vertex to $(0,0)$, which implies by \cref{cor:equivalenceFolner} that the corresponding sequence $(F_n)_{n>0}$ of fundamental domains is F\o lner. In particular, we have 
	\[
	\binom{-2}{5} = \binom{0}{1} + L \binom{1}{0} + L^2 \binom{-1}{-1} +L^3\binom{0}{1},
	\]
	so \cref{prop:equivalenceforfolnerproperty} implies that $(-2,5) \in F_4$ satisfies $(-2,5)+C_{L,F_1}+K \subseteq F_4$. This inclusion is illustrated on the left part of \cref{fig:ballsbecomingbigger}.
	
	\begin{proof}[Proof of \cref{thm:decidabilityfolnerporperty}.]
		Set $B = [B({\bm 0},\bar{r}) \cap \mathbb{Z}^d]$. Since $\|K\| \leq \bar{r}$ (see \cref{rem:powers if needed}), we have that $K \subseteq B$. Assuming that some $F_n$ contains a translation of $C_{L,F_1}+B$, we directly deduce from \cref{lemma:EquivalenceFolner} that the sequence $(F_n)_{n>0}$ is F\o lner.
		
		Let us now assume that the sequence $(F_n)_{n>0}$ is F\o lner. We first prove that we also have that $B \subseteq L(B)+F_1$.
		Indeed, if ${\bm n}, {\bm n}_{1}\in \Z^{d}$ and ${\bm f}_{1}\in F_{1}$ are such that ${\bm n}=L({\bm n}_{1})+{\bm f}_{1}$, then 
		$$\Vert {\bm n}_{1}\Vert \leq \Vert L^{-1}\Vert\cdot \Vert{\bm n}\Vert+\Vert L^{-1}({\bm f}_{1})\Vert.$$
		In particular, we have that
		\begin{equation}\label{eq:ballsapplyingthematrix}
			B
			\subseteq 
			L\left([B\left({\bm 0}, \Vert L^{-1}\Vert \cdot \bar{r}+\Vert L^{-1}(F_{1})\Vert \right)\cap \Z^{d}]\right)+F_{1} 
			= L(B)+F_{1}.
		\end{equation}
		By \cref{cor:equivalenceFolner}, there exists $n \geq 1$ and a word ${\bm f}_{0}{\bm f}_{1} \cdots {\bm f}_{n-1}$ labelling a path $\mathcal{G}(L,F_{1},B)$ from every vertex to ${\bm 0}$.
		When such a word exists, it is known that the length of the shortest one is at most $(N^{3}-N)/6$, where $N$ is the number of vertices of $\mathcal{G}(L,F_{1},B)$~\cite{frankl1982extremal,pin1983two}.
		It thus suffices to observe, using~\eqref{eq:boundonC}, that $N \leq (6r)^d$.
	\end{proof}
	
	Thanks to \cref{lemma:EquivalenceFolner} from now on, we always assume that the sequence of supports of a substitution is F{\o}lner.
	
	The techniques used in the proof of \cref{thm:decidabilityfolnerporperty} are also useful to understand the pattern complexity of a multidimensional substitutive subshift. The \emph{pattern complexity function} (or just the complexity function), denoted by $p_{\zeta}(r)$, is the number of patterns in $\mathcal{L}_{[B({\bm 0},r)\cap \Z^{d}]}(X_{\zeta})$.
	The next result shows that the pattern complexity of multidimensional substitutive subshift is polynomial. This result was already known in the case of aperiodic tillings~\cite[Theorem 7.17]{robinson2004symbolic}, which follows the proof in the dissertation of C. Hansen~\cite{hansen2000dynamics}.

	\begin{proposition}\label{Lemma:UpperBoundComplexityFunction}
		Let $\zeta$ be an aperiodic and primitive constant-shape substitution with expansion matrix $L_{\zeta}$ and support $F_{1}^{\zeta}$. Then, there exists a constant $c>0$, such that
		$$p_{\zeta}(r)\leq c\cdot r^{-\log(|\det(L_\zeta)|)/\log(\Vert L_\zeta^{-1}\Vert)}.$$
	\end{proposition}
	
	\begin{proof}
		Set $B=\left[B\left({\bm 0},\bar{r}\right)\cap \Z^{d}\right]$, where $\bar{r} = \|L_{\zeta}^{-1}(F_1^{\zeta})\|/(1-\|L_{\zeta}^{-1}\|)$.
		The idea of the proof is to show that, for all $r>0$, every pattern in $\mathcal{L}_{[B({\bm 0},r)\cap \Z^{d}]}(X_{\zeta})$ will appear in the image under some power of $\zeta$ of a pattern in $\mathcal{L}_{C_{L_{\zeta},F_1^{\zeta}}+B}(X_{\zeta})$ and that we can control the needed power. The polynomial bound then appears using some counting argument. 
		
		First observe that the inclusion given in~\eqref{eq:ballsapplyingthematrix} holds if we replace $\bar{r}$ by any positive radius $r$, that is, 
		\[
		\left[B\left({\bm 0},r\right)\cap \Z^{d}\right]
		\subseteq 
		L_{\zeta}\left([B\left({\bm 0}, \Vert L_{\zeta}^{-1}\Vert \cdot r+\Vert L_{\zeta}^{-1}(F_{1}^{\zeta})\Vert \right)\cap \Z^{d}]\right)+F_{1}^{\zeta}. 
		\]       
		This implies that for all $r >0$,
		\begin{equation}\label{EquationOfCoveraDiscreteBall}
			[B({\bm 0},r)\cap \Z^{d}]\subseteq L_{\zeta}^{n(r)}(B)+F_{n(r)}^{\zeta},
		\end{equation}
		where $n(r)= \lceil \log(r-\Vert L_{\zeta}^{-1}(F_{1}^{\zeta}))\Vert/\log(\Vert L_{\zeta}^{-1}\Vert)\rceil$.
		Recall that the set $C_{L_{\zeta},F_1^{\zeta}}$ satisfies $C_{L_{\zeta},F_1^{\zeta}} + F_n^{\zeta} + F_n^{\zeta} \subseteq L_{\zeta}^n(C_{L_{\zeta},F_1^{\zeta}})+F_n^{\zeta}$ for all $n \geq 0$.
		Thus, using \eqref{EquationOfCoveraDiscreteBall}, we get that
		$$F_{n(r)}^{\zeta}+[B({\bm 0},r)\cap \Z^{d}]\subseteq L_{\zeta}^{n(r)}(C_{L_{\zeta},F_{1}^{\zeta}}+B)+F_{n(r)}^{\zeta}.$$
		
		Let $x\in X_{\zeta}$ be a fixed point of $\zeta$. Since $x=\zeta(x)$, for every pattern $\texttt{u}\in \mathcal{L}_{[B({\bm 0},r)\cap \Z^{d}]}(X_{\zeta})$, there exists $\texttt{v}\in \mathcal{L}_{C_{L_{\zeta},F_{1}^{\zeta}}+B}(X_{\zeta})$ and ${\bm f}\in F_{n(r)}$ such that $\texttt{u}=\zeta^{n(r)}(\texttt{v})_{{\bm f}+[B({\bm 0},r)\cap \Z^{d}]}$. Indeed, if ${\bm n}\in \Z^{d}$ is such that $\texttt{u}=x_{{\bm n}+[B({\bm 0},r)\cap \Z^{d}]}$, write ${\bm n}=L_{\zeta}^{n(r)}({\bm n}_{1})+{\bm f}$ for some ${\bm n}_{1}\in \Z^{d}$ and ${\bm f}\in F_{n(r)}^{\zeta}$. Consider $\texttt{v}=x|_{{\bm n}_{1}+C_{L_{\zeta},F_{1}^{\zeta}}+B}$. Since $x$ is a fixed point of $\zeta$, we have that $\zeta^{n(r)}(\texttt{v})=x|_{{\bm n}+L_{\zeta}^{n(r)}(C_{L_{\zeta},F_{1}^{\zeta}}+B)+F_{n(r)}^{\zeta}}$. In particular, $\texttt{u}=\zeta^{n(r)}(\texttt{v})_{{\bm f}+[B({\bm 0},r)\cap \Z^{d}]}$. Now, since $|\mathcal{L}_{C_{L_{\zeta},F_{1}^{\zeta}}+B}(X_{\zeta})|$ is at most $|\A|^{|C_{L_{\zeta},F_{1}^{\zeta}}+B|}$, we have that
		\begin{align*}
			p_{\zeta}(r) & \leq |\A|^{\left|C_{L_{\zeta},F_{1}^{\zeta}}+B\right|}|\det(L)|^{n(r)} \\
			& \leq |\A|^{\left|C_{L_{\zeta},F_{1}^{\zeta}}+B\right|}\cdot r^{-\log(|\det(L_{\zeta})|)/\log(\Vert L_{\zeta}^{-1}\Vert)},
		\end{align*}
		which ends the proof.
	\end{proof}
	
	\section{Conjugacy between constant-shape substitutions sharing the expansion matrix}\label{Section:ConjugacySameMatrices}
	
	Constant-shape substitutions in dimension 1 were defined in \cite{fogg2007substitutions} under the name of \emph{pattern substitutions}.
	This notion slightly differs from the one-dimensional constant-shape substitutions by allowing the support associated with each letter to vary.   
	The author proved that every biinfinite sequence which is a fixed point of a pattern substitution is, in fact, substitutive. As a consequence, pattern substitutions do not generate new aperiodic sequences beyond those produced by regular substitutions. This raised the question of whether this fact holds in higher dimensions. In this section, we prove an analog of this result (\cref{ConjugationSubstitutionsDifferentFundamentalDomains}): For a fixed expansion matrix, the conjugacy class of substitutive subshifts is invariant by changing the supports of the substitution.
	
	Let us start with an example. The triangular Thue-Morse substitution has exactly 8 $\sigma_{\Delta TM}$-periodic points of order 2 (or $\sigma_{\Delta TM}^{2}$ has exactly 8 fixed points), which are generated by the 8 patterns in $\mathcal{L}_{K_{\Delta TM}}(X_{\sigma_{\Delta TM}})$:
	
	\begin{figure}[H]
		\centering
		$$\begin{array}{ccccccccccccccc}
			&  & b &  &  &  & a &  &  &  & a &  &  &  & b \\ 
			b & a &  &  & a & b &  &  & b & b &  &  & a & a &  \\ 
			& a &  &  &  & b &  &  &  & a &  &  &  & b &  \\ 
			&  &  &  &  &  &  &  &  &  &  &  &  &  &  \\ 
			&  & a &  &  &  & b &  &  &  & a &  &  &  & b \\ 
			b & a &  &  & a & b &  &  & a & a &  &  & b & b &  \\ 
			& b &  &  &  & a &  &  &  & a &  &  &  & b &  \\ 
		\end{array}$$
		\caption{The 8 patterns in $\mathcal{L}_{K_{\Delta TM}}(X_{\sigma_{\Delta TM}})$.}
		\label{PatternsGeneratedFixedPointsTriangularThueMorse}
	\end{figure}
	
	Now, consider the following square substitution $\sigma_{\thesigmavariable}$, with $L=2\cdot \id_{\R^{2}}$ and $F_{1}=\llbracket 0,1\rrbracket^{2}$, over the alphabet $\A=\{0,1,\ldots,15\}$ defined as
	
	\begin{figure}[H]
		$$\begin{array}{cccccccccccccccccccc}
			&\multicolumn{1}{c}{} & \multicolumn{1}{c}{} & \multicolumn{1}{c}{2} & \multicolumn{1}{c}{3} & \multicolumn{1}{c}{} & \multicolumn{1}{c}{} & \multicolumn{1}{c}{} & \multicolumn{1}{c}{5} & \multicolumn{1}{c}{6} & \multicolumn{1}{c}{} & \multicolumn{1}{c}{} & \multicolumn{1}{c}{} & \multicolumn{1}{c}{6} & \multicolumn{1}{c}{4} & \multicolumn{1}{c}{} & \multicolumn{1}{c}{} & \multicolumn{1}{c}{} & \multicolumn{1}{c}{8} & \multicolumn{1}{c}{0} \\ 
			\sigma_{\thesigmavariable}: & \multicolumn{1}{c}{0} & \multicolumn{1}{c}{\mapsto} & \multicolumn{1}{c}{0} & \multicolumn{1}{c}{1} & \multicolumn{1}{c}{} & \multicolumn{1}{c}{1} & \multicolumn{1}{c}{\mapsto} & \multicolumn{1}{c}{7} & \multicolumn{1}{c}{4} & \multicolumn{1}{c}{} & \multicolumn{1}{c}{2} & \multicolumn{1}{c}{\mapsto} & \multicolumn{1}{c}{7} & \multicolumn{1}{c}{5} & \multicolumn{1}{c}{} & \multicolumn{1}{c}{3} & \multicolumn{1}{c}{\mapsto} & \multicolumn{1}{c}{0} & \multicolumn{1}{c}{8} \\ 
			& \multicolumn{1}{c}{} & \multicolumn{1}{c}{} & \multicolumn{1}{c}{} & \multicolumn{1}{c}{} & \multicolumn{1}{c}{} & \multicolumn{1}{c}{} & \multicolumn{1}{c}{} & \multicolumn{1}{c}{} & \multicolumn{1}{c}{} & \multicolumn{1}{c}{} & \multicolumn{1}{c}{} & \multicolumn{1}{c}{} & \multicolumn{1}{c}{} & \multicolumn{1}{c}{} & \multicolumn{1}{c}{} & \multicolumn{1}{c}{} & \multicolumn{1}{c}{} & \multicolumn{1}{c}{} & \multicolumn{1}{c}{} \\ 
			& \multicolumn{1}{c}{} & \multicolumn{1}{c}{} & \multicolumn{1}{c}{8} & \multicolumn{1}{c}{6} & \multicolumn{1}{c}{} & \multicolumn{1}{c}{} & \multicolumn{1}{c}{} & \multicolumn{1}{c}{10} & \multicolumn{1}{c}{11} & \multicolumn{1}{c}{} & \multicolumn{1}{c}{} & \multicolumn{1}{c}{} & \multicolumn{1}{c}{2} & \multicolumn{1}{c}{4} & \multicolumn{1}{c}{} & \multicolumn{1}{c}{} & \multicolumn{1}{c}{} & \multicolumn{1}{c}{5} & \multicolumn{1}{c}{0} \\ 
			& \multicolumn{1}{c}{4} & \multicolumn{1}{c}{\mapsto} & \multicolumn{1}{c}{0} & \multicolumn{1}{c}{1} & \multicolumn{1}{c}{} & \multicolumn{1}{c}{5} & \multicolumn{1}{c}{\mapsto} & \multicolumn{1}{c}{12} & \multicolumn{1}{c}{9} & \multicolumn{1}{c}{} & \multicolumn{1}{c}{6} & \multicolumn{1}{c}{\mapsto} & \multicolumn{1}{c}{0} & \multicolumn{1}{c}{8} & \multicolumn{1}{c}{} & \multicolumn{1}{c}{7} & \multicolumn{1}{c}{\mapsto} & \multicolumn{1}{c}{7} & \multicolumn{1}{c}{5} \\
			& \multicolumn{1}{c}{} & \multicolumn{1}{c}{} & \multicolumn{1}{c}{} & \multicolumn{1}{c}{} & \multicolumn{1}{c}{} & \multicolumn{1}{c}{} & \multicolumn{1}{c}{} & \multicolumn{1}{c}{} & \multicolumn{1}{c}{} & \multicolumn{1}{c}{} & \multicolumn{1}{c}{} & \multicolumn{1}{c}{} & \multicolumn{1}{c}{} & \multicolumn{1}{c}{} & \multicolumn{1}{c}{} & \multicolumn{1}{c}{} & \multicolumn{1}{c}{} & \multicolumn{1}{c}{} & \multicolumn{1}{c}{} \\ 
			& \multicolumn{1}{c}{} & \multicolumn{1}{c}{} & \multicolumn{1}{c}{14} & \multicolumn{1}{c}{11} & \multicolumn{1}{c}{} & \multicolumn{1}{c}{} & \multicolumn{1}{c}{} & \multicolumn{1}{c}{0} & \multicolumn{1}{c}{10} & \multicolumn{1}{c}{} & \multicolumn{1}{c}{} & \multicolumn{1}{c}{} & \multicolumn{1}{c}{14} & \multicolumn{1}{c}{9} & \multicolumn{1}{c}{} & \multicolumn{1}{c}{} & \multicolumn{1}{c}{} & \multicolumn{1}{c}{0} & \multicolumn{1}{c}{8} \\ 
			& \multicolumn{1}{c}{8} & \multicolumn{1}{c}{\mapsto} & \multicolumn{1}{c}{8} & \multicolumn{1}{c}{13} & \multicolumn{1}{c}{} & \multicolumn{1}{c}{9} & \multicolumn{1}{c}{\mapsto} & \multicolumn{1}{c}{8} & \multicolumn{1}{c}{13} & \multicolumn{1}{c}{} & \multicolumn{1}{c}{10} & \multicolumn{1}{c}{\mapsto} & \multicolumn{1}{c}{8} & \multicolumn{1}{c}{0} & \multicolumn{1}{c}{} & \multicolumn{1}{c}{11} & \multicolumn{1}{c}{\mapsto} & \multicolumn{1}{c}{8} & \multicolumn{1}{c}{0} \\ 
			& \multicolumn{1}{c}{} & \multicolumn{1}{c}{} & \multicolumn{1}{c}{} & \multicolumn{1}{c}{} & \multicolumn{1}{c}{} & \multicolumn{1}{c}{} & \multicolumn{1}{c}{} & \multicolumn{1}{c}{} & \multicolumn{1}{c}{} & \multicolumn{1}{c}{} & \multicolumn{1}{c}{} & \multicolumn{1}{c}{} & \multicolumn{1}{c}{} & \multicolumn{1}{c}{} & \multicolumn{1}{c}{} & \multicolumn{1}{c}{} & \multicolumn{1}{c}{} & \multicolumn{1}{c}{} & \multicolumn{1}{c}{} \\ 
			& \multicolumn{1}{c}{} & \multicolumn{1}{c}{} & \multicolumn{1}{c}{15} & \multicolumn{1}{c}{8} & \multicolumn{1}{c}{} & \multicolumn{1}{c}{} & \multicolumn{1}{c}{} & \multicolumn{1}{c}{15} & \multicolumn{1}{c}{10} & \multicolumn{1}{c}{} & \multicolumn{1}{c}{} & \multicolumn{1}{c}{} & \multicolumn{1}{c}{10} & \multicolumn{1}{c}{9} & \multicolumn{1}{c}{} & \multicolumn{1}{c}{} & \multicolumn{1}{c}{} & \multicolumn{1}{c}{6} & \multicolumn{1}{c}{3} \\ 
			& \multicolumn{1}{c}{12} & \multicolumn{1}{c}{\mapsto} & \multicolumn{1}{c}{12} & \multicolumn{1}{c}{15} & \multicolumn{1}{c}{} & \multicolumn{1}{c}{13} & \multicolumn{1}{c}{\mapsto} & \multicolumn{1}{c}{12} & \multicolumn{1}{c}{9} & \multicolumn{1}{c}{} & \multicolumn{1}{c}{14} & \multicolumn{1}{c}{\mapsto} & \multicolumn{1}{c}{12} & \multicolumn{1}{c}{15} & \multicolumn{1}{c}{} & \multicolumn{1}{c}{15} & \multicolumn{1}{c}{\mapsto} & \multicolumn{1}{c}{7} & \multicolumn{1}{c}{4} \\ 
		\end{array}$$
		\caption{A square substitution conjugate to the triangular Thue-Morse substitution.}
		\label{SquareSubstitutionConjugateTriangularThueMorse}
	\end{figure}
	
	This square substitution is conjugated to the triangular Thue-Morse substitution via the following coding:
	
	$$\begin{array}{cccccccccccccccc}
		\Phi: & 0 & \mapsto & a &  & 1 & \mapsto & b &  & 2 & \mapsto & b &  & 3 & \mapsto & a \\ 
		&  &  &  &  &  &  &  &  &  &  &  &  &  &  &  \\ 
		& 4 & \mapsto & a &  & 5 & \mapsto & a &  & 6 & \mapsto & a &  & 7 & \mapsto & b \\ 
		&  &  &  &  &  &  &  &  &  &  &  &  &  &  &  \\ 
		& 8 & \mapsto & b &  & 9 & \mapsto & b &  & 10 & \mapsto & b &  & 11 & \mapsto & b \\ 
		&  &  &  &  &  &  &  &  &  &  &  &  &  &  &  \\ 
		& 12 & \mapsto & a &  & 13 & \mapsto & a &  & 14 & \mapsto & a &  & 15 & \mapsto & b. \\ 
	\end{array}$$
	
	To see this, we note that $\sigma_{\thesigmavariable}$ also has exactly 8 $\sigma_{1}$-periodic points of order 2 generated by the following patterns in $\mathcal{L}_{\llbracket 0,1\rrbracket^{2}}(X_{\sigma})$:
	
	\begin{figure}[H]
		\centering
		$$\begin{array}{ccccccccccc}
			9 & 0 &  & 4 & 8 &  & 13 & 0 &  & 1 & 8 \\ 
			0 & 10 &  & 8 & 6 &  & 3 & 14 &  & 11 & 2 \\ 
			&  &  &  &  &  &  &  &  &  &  \\ 
			13 & 12 &  & 1 & 7 &  & 9 & 12 &  & 4 & 7 \\ 
			4 & 10 &  & 9 & 6 &  & 6 & 14 &  & 10 & 2. \\ 
		\end{array}$$
		\caption{The patterns that generate the 8 fixed points of $\sigma_{\thesigmavariable}^{2}$.}
		\label{fig:my_label}
	\end{figure}
	
	A standard computation shows that, if we define $\phi:X_{\sigma_{\Delta TM}}\to \phi(X_{\sigma_{\thesigmavariable}})$ by the coding $\phi(x)_{\bm n}=\Phi(x_{\bm n})$ for any ${\bm n}\in \Z^{2}$, then any fixed point of $\sigma_{\thesigmavariable}^{2}$ is mapped, via $\phi$, to a fixed point of $\sigma_{\Delta TM}^{2}$. The minimality of $(X_{\sigma_{\Delta TM}},S,\Z^{2})$ let us conclude that $\phi(X_{\sigma_{\Delta TM}})=X_{\sigma_{\thesigmavariable}}$. It can be shown that, the map $\psi:X_{\sigma_{\Delta TM}}\to X_{\sigma_{\thesigmavariable}}$ induced by the following local map:
	$$\begin{array}{llllllllllllllllllll}
		\Psi: & b & a &  &  &  & a & a &  &  &  & b & a &  &  &  & a & a &  &  \\ 
		& a & b & \mapsto & 0 &  & b & b & \mapsto & 1 &  & b & a & \mapsto & 2 &  & a & a & \mapsto & 3 \\ 
		&  &  &  &  &  &  &  &  &  &  &  &  &  &  &  &  &  &  &  \\ 
		& a & a &  &  &  & a & b &  &  &  & b & a &  &  &  & a & a &  &  \\ 
		& a & b & \mapsto & 4 &  & a & a & \mapsto & 5 &  & a & a & \mapsto & 6 &  & b & a & \mapsto & 7 \\ 
		&  &  &  &  &  &  &  &  &  &  &  &  &  &  &  &  &  &  &  \\ 
		& a & b &  &  &  & b & b &  &  &  & a & b &  &  &  & b & b &  &  \\ 
		& b & a & \mapsto & 8 &  & b & a & \mapsto & 9 &  & b & b & \mapsto & 10 &  & b & b & \mapsto & 11 \\ 
		&  &  &  &  &  &  &  &  &  &  &  &  &  &  &  &  &  &  &  \\ 
		& b & b &  &  &  & b & b &  &  &  & a & b &  &  &  & b & a &  &  \\ 
		& a & b & \mapsto & 12 &  & a & a & \mapsto & 13 &  & a & b & \mapsto & 14 &  & b & b & \mapsto & 15, \\ 
	\end{array}$$
	
	\noindent satisfies $\psi\circ \phi=\id_{X_{\sigma_{\thesigmavariable}}}$, so $(X_{\sigma_{1}},S,\Z^{2}), (X_{\sigma_{TM}},S,\Z^{2})$ are topologically conjugate. The example above generalizes as follows.
	
	\begin{theorem}\label{ConjugationSubstitutionsDifferentFundamentalDomains}
		Let $\zeta$ be an aperiodic primitive constant-shape substitution with an expansion matrix $L$ and support $F_{1}$. Now, consider another fundamental domain $G_1\Subset \Z^{d}$ of $\mathbb{Z}^d/L(\mathbb{Z}^d)$, with $\mathbf{0} \in G_1$, and such that the associated sequences $(F_{n})_{n>0}$, $(G_{n})_{n>0}$ are F\o lner.
		There exists an aperiodic computable primitive constant-shape substitution $\tilde{\zeta}$ with support $G_{1}$ such that $(X_{\zeta},S,\Z^{d})$ and $(X_{\tilde{\zeta}},S,\Z^{d})$ are topologically conjugate.
	\end{theorem}
	
	The proof of \cref{ConjugationSubstitutionsDifferentFundamentalDomains} is an adaptation of the construction of substitution of length $n$ from the one-dimensional case that we now recall. We refer to~\cite{queffelec2010substitution} for more details. Assume that $\zeta : \A^* \to \A^*$ is a primitive one-dimensional substitution. For every $n \geq 1$, we consider the set $\mathcal{L}_n(X_\zeta)$ as an alphabet $\B_n$ and define the substitution $\zeta_n:\B_n^* \to \B_n^*$ as follows. If $w = w_1 \cdots w_n \in \B_n$ and $\zeta(w) = v = v_1 \cdots v_\ell$, where all $w_i$ and $v_j$ are letters, then $\ell \geq |\zeta(w_1)|+n-1$ and we set
	\[
	\zeta_n(w) = (v_1 \cdots v_n)(v_2 \cdots v_{n+1}) \cdots (v_{|\zeta(w_1)|} \cdots v_{|\zeta(w_1)|+n-1}). 
	\]
	It turns out that $\zeta_n$ is a primitive substitution and that the subshifts $X_\zeta$ and $X_{\zeta_n}$ are conjugate, an isomorphism being given by the sliding block code
	\[
	\Phi : \mathcal{L}_{n}(X_\zeta) \to \B_n, w \mapsto (w).  
	\]
	In the proof of \cref{ConjugationSubstitutionsDifferentFundamentalDomains}, we find a set $B$ whose $F_1$-image $L(B)+F_1$ can be cut into blocks with support $B$ along $G_1$. This allows us define the substitution $\tilde{\zeta}$ and we prove that the associated subshifts are conjugate. Note that, the local map $\psi$ in the above example that defines a conjugacy between the triangular Thue-Morse substitution and a square substitution is a coding of the patterns in $\mathcal{L}_{\llbracket 0,1\rrbracket^{2}}(X_{\zeta})$.
	
	\begin{proof} 
		Using Remark~\ref{rem:powers if needed}, we may assume that $\zeta$ has a fixed point $\bar{x}\in X_{\zeta}$. Assume $K_{1}=(\id-L)^{-1}(F_{1})\cap \Z^{d}$ and $K_{2}=(\id-L)^{-1}(G_{1})\cap \Z^{d}$ are the sets given by \cref{FiniteSubsetFillsZd} for $F_{1}$ and $G_{1}$, respectively.   
		
		First, we adapt the proof of \cref{Prop:FiniteSetSatisfyingParticularProperties}, to obtain a finite set $A\Subset \Z^{d}$ such that for any $n\geq 0, K_{2}+A+G_{n}\subseteq L^{n}(K_{2}+A)+F_{n}$. 
		Consider the sequence $(A_n)_{n \geq 0}$ of finite sets in $\mathbb{Z}^d$ as follows: set $A_0 = \{\mathbf{0}\}$, and for $n\geq 0$,
		\begin{align*}
			A_{n+1} 
			&= \{ {\bm p} \in \mathbb{Z}^d \mid \exists {\bm f} \in F_1, {\bm g}_{1}, {\bm g}_{2} \in G_1, {\bm q} \in A_n: L({\bm p})+{\bm f} = {\bm q} + {\bm g}_{1} + {\bm g}_{2}\}\\
			&= L^{-1}(A_n + G_1 + G_1 - F_1) \cap \mathbb{Z}^d.
		\end{align*}
		
		We will prove that this sequence is eventually stationary, i.e., there exists $N>0$ such that for all $m\geq N$, $A_{m}=A_{N}$.
		
		\setcounter{claim}{0}
		\begin{claim}\label{FirstClaimforConjugateSubstitution} For every $n\geq 0$ and every $k>0$, we have the following inclusions: $A_{n}\subseteq A_{n+1}$ and $K_{2}+A_{n}+G_{k}\subseteq L^{k}(K_{2}+A_{n+k})+F_{k}$.
		\end{claim}
		
		This claim shows that if we choose $N>0$ such that for all $m\geq N$, $A_{m}=A_{N}$, then for all $k>0$, we have that $K_{2}+A_{N}+G_{k}\subseteq L^{k}(K_{2}+A_{N})+F_{k}$. Now, if we consider a pattern $\texttt{w}\in \mathcal{L}_{K_{2}+A_{N}}(X_{\zeta})$, the support of $\zeta^{k}(\texttt{w})$ is large enough to be splited into $(K_{2}+A_{N})$-blocks along $G_k$. This will allow us to consider $\mathcal{L}_{K_{2}+A_{N}}(X_{\zeta})$ as an alphabet and to define the substitution $\tilde{\zeta}$ with support $G_1$. 
		
		\begin{proof}[Proof of \cref{FirstClaimforConjugateSubstitution}]
			We first prove that $A_{n}\subseteq A_{n+1}$ for every $n \geq 0$.
			Since ${\bm 0}\in G_1 \cap F_1$, we trivially have that $A_{0} \subseteq A_{1}$. It is then a direct consequence of the definition of the sets $A_n$, that if $A_n \subseteq A_{n+1}$, then $A_{n+1} \subseteq A_{n+2}$.
			
			Now, let us prove the other sequence of inclusions. The inclusion $K_{2}+A_{n}+G_{1}\subseteq L(K_{2}+A_{n+1})+F_{1}$ is direct for any $n\geq 0$. Suppose that $K_{2}+A_{n}+G_{k}\subseteq L^{k}(K_{2}+A_{n+k})+F_{k}$ for some $n \geq 0$ and some $k>0$. 
			Since $G_{k+1} = G_k + L^k(G_{1})$, we get that 
			\[
			K_{2}+A_{n}+G_{k+1} 
			= K_{2}+A_{n}+G_k + L^k(G_{1}) 
			\subseteq L^{k}(K_{2}+A_{n+k}+G_1)+F_{k}. 
			\]
			By the initial case, we have that $K_{2}+A_{n+k}+G_1 \subseteq L(K_{2}+A_{n+k+1})+F_{1}$. Using the equality $F_{k+1}=F_k + L^k(F_1)$,
			\begin{align*}
				L^{k}(K_{2}+A_{n+k}+G_1)+F_{k} 
				& \subseteq L^{k+1}(K_{2}+A_{n+k+1})+L^k(F_1) +F_{k}\\ 
				& =L^{k+1}(K_{2}+A_{n+k+1})+ F_{k+1}.
			\end{align*}
			
			This completes the proof of the claim.
		\end{proof}
		
		Now, we prove that the sequence of finite sets $(A_{n})_{n>0}$ is eventually stationary. Define the sequence $a_{n}=\Vert A_{n}\Vert$. This sequence satisfies 
		$$a_{n+1}\leq \Vert L^{-1}\Vert a_{n} +\Vert L^{-1}(G_{1}+G_{1}-F_1)\Vert,$$
		which implies that
		$$
		a_{n}
		\leq 
		a_{0}\cdot \Vert L^{-1}\Vert^{n}
		+    
		\Vert L^{-1}(G_{1}+G_{1}-F_1)\Vert\dfrac{1-\Vert L^{-1}\Vert^{n}}{1-\Vert L^{-1}\Vert}.$$
		
		Since $\Vert L^{-1}\Vert<1$, the sequence $(a_{n})_{n\geq 0}$ is bounded. Therefore, the nested sequence $(A_n)_{n \geq 0}$ is eventually constant. Let $n \geq 0$ such that $A_n = A_m$ for all $m \geq n$, and set $A = A_n$.
		By~\cref{FirstClaimforConjugateSubstitution}, we have
		\begin{equation}\label{WeCanIterateSubstitution}
			\forall k> 0, K_{2}+A+G_{k}\subseteq L^{k}(K_{2}+A)+F_{k}.
		\end{equation}
		
		To define a substitution $\tilde{\zeta}$ with support $G_{1}$, we consider the set $\mathcal{B}=\mathcal{L}_{K_{2}+A}(X_{\zeta})$ as a new alphabet, and we define the substitution $\tilde{\zeta}$ with support $G_{1}$ on the alphabet $\mathcal{B}$ as follows:
		$$\forall {\bm g}\in G_{1}, (\tilde{\zeta}(\texttt{w}))_{\bm g}=\zeta(\texttt{w})_{{\bm g}+K_{2}+A}.$$
		
		Note that by \eqref{WeCanIterateSubstitution}, the substitution $\tilde{\zeta}$ is well-defined. Using~\cref{FirstClaimforConjugateSubstitution} and the primitivity of $\zeta$, it is straightforward to check that $\tilde{\zeta}$ is primitive. 
		Let us now prove that $(X_{\zeta},S,\Z^{d})$ and $(X_{\tilde{\zeta}},S,\Z^{d})$ are topologically conjugate. Indeed, consider the factor map $\phi:X_{\zeta}\to \mathcal{B}^{\mathbb{Z}^d}$ induced by 
		$$\begin{array}{cccc}
			\Phi: & \mathcal{L}_{K_{2}+A}(X_{\zeta}) & \to &\mathcal{B}\\
			& \texttt{w} & \mapsto & \texttt{w}.
		\end{array}$$
		Thus, for all $x\in X_{\zeta}$ and ${\bm n}\in \Z^{d}$, we have that $\phi(x)_{{\bm n}}=\Phi(x|_{{\bm n}+K_{2}+A})$. We prove that $\phi(\bar{x})$ is a fixed point of $\tilde{\zeta}$. Set ${\bm n}\in \Z^{d}$. There exists a unique ${\bm n}_{1}\in \Z^{d}$ and ${\bm g}\in G_{1}$ such that ${\bm n}=L({\bm n}_{1})+{\bm g}$. Note that
		\begin{align*}
			(\tilde{\zeta}(\phi(\bar{x})))_{{\bm n}} 
			& = (\tilde{\zeta}(\phi(\bar{x})))_{L({\bm n}_{1})+{\bm g}}\\
			& = (\tilde{\zeta}(\phi(\bar{x}))_{{\bm n}_{1}})_{{\bm g}}\\
			& = (\tilde{\zeta}(\Phi(\bar{x}|_{{\bm n}_{1}+K_{2}+A})))_{{\bm g}}\\
			& = (\zeta(\bar{x}|_{{\bm n}_{1}+K_{2}+A}))_{{\bm g}+K_{2}+A}\\
			& = (\zeta(\bar{x}))_{L({\bm n}_{1})+{\bm g}+K_{2}+A}\\
			& = (\zeta(\bar{x}))_{{\bm n}+K_{2}+A}\\
			& = \bar{x}_{{\bm n}+K_{2}+A}\\
			& = \Phi(\bar{x}|_{{\bm n}+K_{2}+A})\\
			& = (\phi(\bar{x}))_{{\bm n}},
		\end{align*}
		so $\phi(\bar{x})\in X_{\tilde{\zeta}}$ is a fixed point of $\tilde{\zeta}$.
		By the minimality of $\phi(X_\zeta)$ and $X_{\tilde{\zeta}}$, we conclude that $\phi(X_\zeta) = X_{\tilde{\zeta}}$. Therefore, $\phi$ is a factor map from $X_\zeta$ to $X_{\tilde{\zeta}}$.
		To prove that it is a conjugacy, we check that the factor map $\psi:X_{\tilde{\zeta}}\to \mathcal{A}^{\mathbb{Z}^d}$ induced by
		$$\begin{array}{cccc}
			\Psi: & \mathcal{L}_{K_{2}+A}(X_{\zeta}) & \to &\mathcal{A}\\
			& \texttt{w} & \mapsto & \texttt{w}_{{\bm 0}}
		\end{array}$$
		is its inverse map. Indeed, for any ${\bm n}\in \Z^{d}$, we get that $\psi(\phi(\bar{x}))_{{\bm n}}=\Psi(\phi(\bar{x})_{{\bm n}})=\Psi(\bar{x}_{{\bm n}+K_{2}+A})=\bar{x}_{{\bm n}}$, i.e., $\psi(\phi(\bar{x}))=\bar{x}$. 
		The minimality of $(X_{\zeta},S,\Z^{d})$ implies that $\psi\circ\phi=\id_{X_{\zeta}}$. Hence, $\phi$, $\psi$ are invertible and $\phi^{-1}=\psi$. We conclude that $(X_{\zeta},S,\Z^{d})$ and $(X_{\tilde{\zeta}},S,\Z^{d})$ are topologically conjugate.
	\end{proof}
	
	Observe that, by construction, the set $B=K_{2}+A$ only depends on $L,F_1$ and $G_1$, not on the combinatorial properties of $\zeta$. For example, the set $B=\llbracket -1,2\rrbracket^{2}$ is enough to obtain a square substitution conjugate to a substitution defined with $L=2\cdot \id_{\R^{2}}$ and $F_{1}=\{(0,0),(1,0),(0,1),(-1,-1)\}$.
	
	\section{Computability of the constant of recognizability}\label{Section:ComputabilityRecognizability}
	
	The recognizability property of substitutions is a combinatorial one that offers a form of invertibility, allowing the unique decomposition of points within the substitutive subshift. Recall that if $\zeta$ is a substitution with a fixed point $x\in X_{\zeta}$, then $\zeta$ is recognizable on $x$ if there exists some constant $R>0$ such that for all ${\bm i}, {\bm j}\in \Z^{d}$,
	\begin{equation} \label{eq:def recognizability}
		x|_{[B(L_{\zeta}({\bm i}),R)\cap \Z^{d}]}=x|_{[B({\bm j},R)\cap \Z^{d}]} \implies (\exists {\bm k}\in \Z^{d}) (({\bm j}=L_{\zeta}({\bm k}))\wedge (x_{{\bm i}}=x_{{\bm k}})).    
	\end{equation}

	This property was initially established for aperiodic primitive substitutions by B. Moss\'e in \cite{mosse1992puissances}. This proof implies the existence of a natural sequence of refining (Kakutani-Rokhlin) partitions, which is a key tool when studying substitutive systems and more general $\mathcal{S}$-adic systems. Subsequently, in \cite{bezuglyi2009aperiodic}, it was extended to cover non-primitive substitutions. Later, F. Durand and J. Leroy proved the computability of the recognizability length for one-dimensional primitive substitutions \cite{durand2017constant}, which was then generalized by M.-P. Beal, D. Perrin, and A. Restivo in \cite{beal2023recognizability} for the most general class of morphisms, including ones with erasable letters. In the multidimensional setting, B. Solomyak showed in \cite{solomyakrecognizability} that aperiodic translationally finite self-affine tilings of $\R^{d}$ satisfy a recognizability property, referred to as the \emph{unique composition property}. Furthermore, in \cite{cabezas2023homomorphisms}, it was demonstrated that aperiodic symbolic factors of constant-shape substitutive subshifts also exhibit a recognizability property. In this section, we provide a computable upper bound for the constant of recognizability of aperiodic primitive constant-shape substitutions (\cref{UpperBoundRecognizability}). This upper bound can be expressed solely in terms of $|\A|,L_{\zeta}, \Vert F_{1}^{\zeta}\Vert$ and $d$. This result will be instrumental in the subsequent section, where we establish the decidability of the factorization problem between minimal substitutive subshifts. To achieve this, we will adapt some of the proofs presented in \cite{cabezas2023homomorphisms} in order to obtain computable bounds. We prove the following.	
	
	\begin{theorem}\label{UpperBoundRecognizability} Let $\zeta$ be an aperiodic primitive constant-shape substitution on an alphabet $\A$, with expansion matrix $L_{\zeta}$ and support $F_{1}^{\zeta}$ admitting a fixed point $x\in X_{\zeta}$.  Define
		
		\begin{itemize}
			\item $t=-\log(\Vert L_{\zeta}\Vert)/\log(\Vert L_{\zeta}^{-1}\Vert)$,
			\item $\bar{r}=\Vert L_{\zeta}^{-1}(F_{1}^{\zeta})\Vert/(1-\Vert L_{\zeta}^{-1}\Vert)$,
			\item $a=\left\lceil(2\Vert F_{1}^{\zeta}\Vert+d)(2\Vert F_{1}\Vert + \Vert L_{\zeta}\Vert^{|\A|^{2}+(|\A|+1)^{(6\bar{r})^{d}}})\right\rceil$,
			\item $\bar{R}=\left\lceil \Vert L_{\zeta}^{-1}\Vert^{|\A|-1}\cdot a\cdot 9^{t}\Vert L_{\zeta}\Vert^{|t(|\A|-1)}\bar{r}^{t})+4\bar{r}\right\rceil$
			\item $\bar{n}=\left\lceil |\A|^{(2\bar{R}+6\bar{r})^{d}}\right\rceil$.
		\end{itemize} Then $\zeta$ is recognizable on $x$ and the constant of recognizability is at most
		$$2\Vert L_{\zeta}\Vert^{|\A|}[2\Vert F_{1}^{\zeta}\Vert + \Vert L_{\zeta}\Vert^{{\bar n}+|\A|}(2\Vert F_{1}^{\zeta}\Vert+7\bar{r}+\Vert L_{\zeta}^{-1}\Vert^{|\A|-1}\cdot a\cdot 9^{t}\Vert L_{\zeta}\Vert^{t(|\A|-1)})\bar{r}^{t}].$$
	\end{theorem}
	
	In B. Moss\'e's original proof, a key argument for the proof of the recognizability property is the existence of an integer $p>0$ with the following property: for all $a,b\in \A$, if $\zeta^{n}(a)=\zeta^{n}(b)$ for some $n\geq 0$, then $\zeta^{p}(a)=\zeta^{p}(b)$. This result was proved in \cite{ehrenfeucht1978simplifcations}. Notably, this property holds true for $p=1$ when the substitution is injective on letters. The original proof concerns only one-dimensional morphisms. Nevertheless, it is possible to adapt the proof to the multidimensional context. The proof is left to the reader.
	
	\begin{theorem}\cite[Theorem 3]{ehrenfeucht1978simplifcations}\label{PropositionForNonInyectiveSubstitutions}. Let $\zeta$ be a constant-shape substitution. Then for any patterns $\texttt{u},\texttt{v}\in \A^{P}$, for some $P\Subset \Z^{d}$, we have that
		$$\zeta^{|\A|-1}(\texttt{u})\neq \zeta^{|\A|-1}(\texttt{v}) \implies \forall n,\ \zeta^{n}(\texttt{u})\neq \zeta^{n}(\texttt{v}).$$
	\end{theorem}
	
	We recall that $\zeta$ is primitive if and only if its \emph{incidence matrix} $M_{\zeta}$ defined for all $a,b\in \A$ as $(M_{\zeta})_{a,b}=|\{{\bm f}\in F_{1}^{\zeta}\colon \zeta(a)_{{\bm f}}=b\}|$ is primitive, i.e., there exists $k>0$ such that $M_{\zeta}^{k}$ only contains positive integer entries. The following is a well-known bound for this $k$.
	
	\begin{lemma}\cite{wielandt1950matrizen}\label{LemmaPrimitivity} 
		A non-negative $d\times d$ matrix $M$ is primitive if, and only if, there is an integer $k\leq d^{2}-2d+2$ such that $M^{k}$ only contains positive entries.
	\end{lemma}
	
	Following the proof of the recognizability property of multidimensional substitutive subshifts in \cite{cabezas2023homomorphisms}, we first study the computability of the growth of the repetitivity function. We recall that the \emph{repetitivity function} of a minimal subshift is the map $R_{X}:\R_{+}\to\R_{+}$ defined for $r>0$ as the smallest radius such that every discrete ball $[B({\bm n},R_{X}(r))\cap \Z^{d}]$ contains an occurrence of every pattern whose support has a diameter less than $r$. 
	
	Like in several proofs of \cref{Section:DecidabilityFOlner}, we will consider the set $K$ of \cref{FiniteSubsetFillsZd}, the set $C_{L_{\zeta},F_{1}^{\zeta}}$ given by \cref{Prop:FiniteSetSatisfyingParticularProperties} with $A=\{{\bm 0}\}$ and $F=F_{1}^{\zeta}+F_{1}^{\zeta}$ and the set $B=\left[B\left({\bm 0},\bar{r}\right)\cap \Z^{d}\right]$, where $\bar{r} = \|L_{\zeta}^{-1}(F_1^{\zeta})\|/(1-\|L_{\zeta}^{-1}\|)$. Recall that $B \subseteq L_{\zeta}(B)+F_1^{\zeta}$ (see Equation~\eqref{eq:ballsapplyingthematrix}) and that $K$ satisfies $\|K\| \leq \bar{r}$ (see \cref{rem:GoodUniformlyBounded}), and so is included in $B$.

	\begin{lemma}\label{GrowthRepetititvtyFunction}
		Let $\zeta$ be an aperiodic primitive constant-shape substitution. Define $t=-\log(\Vert L_{\zeta}\Vert)/\log(\Vert L_{\zeta}^{-1}\Vert)$. Then, 
		\[
		R_{X_{\zeta}}(r)
		\leq 
		(2 \|F_1^{\zeta}\|+ \|L_{\zeta}\|^{|\A|^{2}+(|\A|+1)^{(6\bar{r})^d}} (2 \|F_1^{\zeta}\|+d)) r^t.
		\]
	\end{lemma}
	
	\begin{proof} 
		Set $r>0$. We will show that every pattern $\texttt{u} \in \mathcal{L}_{[B({\bm 0},r) \cap \mathbb{Z}^d]}(X_\zeta)$ occurs in every image $\zeta^{m(r)}(a)$, $a \in \A$, for some $m(r)\in \N$ and then we give an upper bound on $m(r)$, from which we deduce a bound on the repetitivity function.
		
		Following the proof of \cref{Lemma:UpperBoundComplexityFunction}, we know that for every pattern $\texttt{u}\in \mathcal{L}_{[B({\bm 0},r)\cap \Z^{d}]}(X_{\zeta})$, there exists $\texttt{v}\in \mathcal{L}_{C_{L_{\zeta},F_{1}^{\zeta}}+B}(X_{\zeta})$ and ${\bm f}\in F_{n(r)}^{\zeta}$ such that $\texttt{u}=\zeta^{n(r)}(\texttt{v})_{{\bm f}+[B({\bm 0},r)\cap \Z^{d}]}$, where $n(r)=\lceil \log(r-\Vert L_{\zeta}^{-1}(F_{1}^{\zeta}))\Vert/\log(\Vert L_{\zeta}^{-1}\Vert)\rceil$.
		Thus, consider $\texttt{v}\in \mathcal{L}_{C_{L_{\zeta},F_1^{\zeta}}+B}(X_{\zeta})$ and  let $n\in \N$ denote the smallest positive integer such that $\texttt{v}\sqsubseteq \zeta^{n}(a)$ for some $a\in \A$ (such an $n$ exists by primitivity of the substitution).
		
		Let us first prove that $n\leq (|\A|+1)^{\left|C_{L_{\zeta},F_{1}^{\zeta}}+B\right|}$.         Indeed, as $\texttt{v}_1 := \texttt{v}$ occurs in $\zeta^{n}(a)$, there exists ${\bm f}_{n} \in F_n$ such that $\zeta^{n}(a)_{{\bm f}_{n}+C_{L,F_{1}}+B}=\texttt{v}_1$.
		Writing ${\bm f}_{n}=L_{\zeta}({\bm f}_{n-1})+{\bm f}_{1}$, there exists a pattern $\texttt{v}_2$ with minimal support $S_2$ such that $\zeta^{n-1}(a)_{{\bm f}_{n-1}+S_2} = \texttt{v}_2$ and $\texttt{v}_{1} = \zeta(\texttt{v}_{2})_{{\bm f}_{1}+C_{L_{\zeta},F_{1}^{\zeta}}+B}$.
		Note that we have 
		\[
		C_{L_{\zeta},F_1^{\zeta}}+B+{\bm f}_1 \subseteq C_{L_{\zeta},F_1^{\zeta}}+L_{\zeta}(B)+F_1^{\zeta}+F_1^{\zeta} \subseteq L_{\zeta}(B+C_{L_{\zeta},F_1^{\zeta}})+F_1^{\zeta},
		\]
		so the support $S_2$ of $\texttt{v}_{2}$ satisfies $S_2 \subseteq B+C_{L,F_1}$.
		Observe that since $\texttt{v}_1$ does not occur in $\zeta^{n-1}(a)$, we have that $\texttt{v}_2 \neq \texttt{v}_1$.
		Using the same argument, we can find a pattern $\texttt{v}_3$ with minimal support $S_3$ in $\zeta^{n-2}(a)$ such that $\texttt{v}_3$ occurs in $\zeta(v_2)$, $S_3 \subseteq B+C_{L_{\zeta},F_1^{\zeta}}$ and $\texttt{v}_3 \notin \{\texttt{v}_1,\texttt{v}_2\}$. Continuing this way, we inductively construct a sequence of pairwise distinct patterns $\texttt{v}_{1},\texttt{v}_{2},\ldots$ in $\bigcup_{S \subseteq C_{L_{\zeta},F_{1}^{\zeta}}+B}\mathcal{L}_{S}(X_{\zeta})$, such that for any $j\geq 1$, $\texttt{v}_{j}\sqsubseteq \zeta(\texttt{v}_{j+1})$, and $\texttt{v}_{j}$ occurs in $\zeta^{n-j+1}(a)$ but does not occur in $\zeta^{n-j}(a)$. Considering that there are at most $|\A|^{\left|S\right|}$ patterns in $\mathcal{L}_{S}(X_{\zeta})$, we conclude that $n\leq (|\A|+1)^{\left|C_{L_{\zeta},F_{1}^{\zeta}}+B\right|}$.
		
		Now, by \cref{LemmaPrimitivity}, for any pair of letters $a,b\in \A$, we have that $a\sqsubseteq \zeta^{|\A|^{2}}(b)$. Hence, for any letter $a\in \A$, any pattern $\texttt{v}\in \mathcal{L}_{C_{L_{\zeta},F_{1}^{\zeta}}+B}(X_{\zeta})$ occurs in $\zeta^{|\A|^{2}+(|\A|+1)^{\left|C_{L_{\zeta},F_{1}^{\zeta}}+B\right|}}(a)$.
		
		Since for any $n>0$, $L_{\zeta}^{n}(\Z^{d})$ is $d\Vert L_{\zeta}\Vert^{n}$-relatively dense, any ball of radius $d\Vert L_{\zeta}\Vert^{n}+2\Vert F_{1}^{\zeta}\Vert\cdot\Vert L_{\zeta}\Vert^{n}$ contains a set of the form $L_{\zeta}^{n}({\bm m})+F_{n}^{\zeta}$, for some ${\bm m}\in \Z^{d}$. This implies that any pattern of the form $\zeta^{n}(a)$, for some $a\in \A$, occurs in any pattern in $\mathcal{L}_{[B({\bm 0}, \Vert L_{\zeta}\Vert^{n}(d+2\Vert F_{1}^{\zeta}\Vert)\cap \Z^{d}]}(X_{\zeta})$. In particular, for $N=|\A|^{2}+(|\A|+1)^{\left|C_{L_{\zeta},F_{1}^{\zeta}}+B\right|}$, we conclude  that any ball of radius $\|L_{\zeta}\|^N (2 \|F_1^{\zeta}\|+d)$ contains an occurrence of any pattern $\texttt{v}\in \mathcal{L}_{C_{L_{\zeta},F_1^{\zeta}}+B}(X_{\zeta})$. Hence, by the first part of the proof, any ball of radius $\|L_{\zeta}\|^{n(r)} (2 \|F_1^{\zeta}\|+ \|L_{\zeta}\|^N (2 \|F_1^{\zeta}\|+d))$ contains an occurrence of any pattern $\texttt{u}\in \mathcal{L}_{[B({\bm 0},r)\cap \Z^{d}]}(X_{\zeta})$.
		
		To finish the proof, we deduce from Equation~\eqref{eq:boundonC} that $\|C_{L_{\zeta},F_1^{\zeta}}+B\| \leq 3 \bar{r}$. Using classical upper bounds for the cardinality of the discrete balls $[B({\bm 0},r)\cap \Z^{d}]$, we have that $|C_{L_{\zeta},F_1^{\zeta}}+B|\leq (6 \bar{r})^{d}$. With this new bound, we get that for any $r>0$,
		\begin{align*}
			R_{X_{\zeta}}(r) &
			\leq (2 \|F_1^{\zeta}\|+ \|L_{\zeta}\|^N (2 \|F_1^{\zeta}\|+d)) \Vert L\Vert^{\log(r)/\log(\Vert L^{-1}\Vert)}\\
			& \leq 
			(2 \|F_1^{\zeta}\|+ \|L_{\zeta}\|^{|\A|^{2}+(|\A|+1)^{(6\bar{r})^{d}}} (2 \|F_1^{\zeta}\|+d)) r^t.
		\end{align*}
	\end{proof}
	
	As pointed out in \cite{cabezas2023homomorphisms}, the growth of the repetitivity function has a direct consequence on the distance between two occurrences of a pattern in a point $x\in X_{\zeta}$, called \emph{repulsion property}. This is an analog to the \emph{k-power-free} property of one-dimensional primitive substitutions. We add the proof for completeness.
	
	\begin{proposition}[Repulsion property]\label{RepulsionProperty} Let $\zeta$ be an aperiodic primitive constant-shape substitution, $x\in X_{\zeta}$ and set $t=-\log(\Vert L_{\zeta}\Vert)/\log(\Vert L_{\zeta}^{-1}\Vert)$. Then, if a pattern $\texttt{p}\sqsubseteq x$ with $[B({\bm s},r)\cap \Z^{d}]\subseteq \supp(\texttt{p})$, for some ${\bm s}\in \Z^{d}$ and $r>0$, has two occurrences ${\bm j}_{1}, {\bm j}_{2}\in \Z^{d}$ in $x$ such that 
		$r\geq R_{X_{\zeta}}(\Vert {\bm j}_{1}-{\bm j}_{2}\Vert)$, then ${\bm j}_{1}$ is equal to ${\bm j}_{2}$.	
	\end{proposition}
	
	\begin{proof}
		For any ${\bm k}\in \Z^{d}$, we consider the pattern $\texttt{w}_{{\bm k}}=x|_{{\bm k}\cup ({\bm k}+{\bm j}_{2}-{\bm j}_{1})}$. Note that $\diam(\supp(\texttt{w}_{\bm k}))=\Vert {\bm j}_{2}-{\bm j}_{1}\Vert$. Since $r\geq R_{X_{\zeta}}(\Vert {\bm j}_{2}-{\bm j}_{1}\Vert)$, then the support of the pattern $\texttt{p}$ contains an occurrence in $x$ of any pattern $\texttt{w}_{{\bm k}}$. Since ${\bm j}_{1}$ is an occurrence of $\texttt{p}$ in $x$, we get that for any ${\bm k}\in \Z^{d}$, there exists ${\bm n}_{\bm k}\in \Z^{d}$ such that $x_{{\bm j}_{1}+{\bm n}_{{\bm k}}+{\bm k}}=x_{{\bm k}}$ and $x_{{\bm j}_{1}+{\bm n}_{{\bm k}}+({\bm j}_{2}-{\bm j}_{1}+{\bm k})}=x_{{\bm j}_{2}-{\bm j}_{1}+{\bm k}}$, which implies that $x_{{\bm j}_{2}+{\bm n}_{{\bm k}}+{\bm k}}=x_{{\bm j}_{2}-{\bm j}_{1}+{\bm k}}$. The fact that ${\bm j}_{2}$ is an occurrence of $\texttt{p}$ in $x$ let us conclude that for any ${\bm k}\in \Z^{d}$, $x_{{\bm j}_{2}-{\bm j}_{1}+{\bm k}}$ is equal to $x_{{\bm k}}$, i.e., ${\bm j}_{2}-{\bm j}_{1}$ is a period of $x$. Since $\zeta$ is aperiodic, we conclude that ${\bm j}_{1}={\bm j}_{2}$.
	\end{proof}
	
	\begin{figure}[H]
		\centering
		\begin{tikzpicture}[scale=0.3]
			\node(b1) at (8, 8.5) [scale=1] {$\texttt{p}$};
			
			\path[thin](0,0) edge (16,0);
			\path[thin](16.,0.) edge (16.,8.);
			\path[thin](16.,8.) edge (0.,8.);
			\path[thin](0.,8.) edge (0.,0.);
			
			\node(b8) at (10, 6.5) [scale=1,blue] {$\texttt{p}$};
			
			\path[thin,blue](1,-1) edge (17,-1);
			\path[thin,blue](17.,-1.) edge (17.,7);
			\path[thin,blue](17.,7.) edge (1.,7.);
			\path[thin,blue](1.,7.) edge (1.,-1.);

			\draw[->,thin] (0,0) -- (1,-1);
			
			\node(b15) at (-0.3,-0.3)[scale=1]{${\bm j}_{1}$};
			
			\node(b16) at (0.7,-1.3)[scale=1]{${\bm j}_{2}$};	
			
			\draw[red] (9,4) arc (0:360:3.5cm);
			
			\node(b17) at (12.5,4)[scale=1,red]{$B({\bm s},R_{X_{\zeta}}(r_{1}))$};
			
			\draw[brown] (1.6,-0.3) arc (0:360:1.9cm);
			
			\node(b18) at (-5,-0.4)[scale=1,brown]{$B({\bm j_{1}}, r_{1})$};		
		\end{tikzpicture}
		\caption{Illustration of a forbidden situation given by the repulsion property (\cref{RepulsionProperty}).}
		\label{illustrationrepulsionproperty}
	\end{figure}
	
	Now, we proceed to give a computable upper bound for the constant of recognizability of constant-shape substitutions. As mentioned in \cite{durand2017constant}, the proof of the recognizability property has two steps. Using the notation of Equation~\ref{eq:def recognizability}, the first step is to show that ${\bm j}$ belongs to $L_{\zeta}(\mathbb{Z}^d)$ and the second on is to show that $x_{{\bm i}} = x_{{\bm k}}$. 
	Here, we adapt the proofs in \cite{cabezas2023homomorphisms}.
	
	\begin{proposition}[First step of recognizability: positions of ``cutting bars'']\label{ComputabilityRecognizability}
		Let $\zeta$ be an aperiodic primitive constant-shape substitution from an alphabet $\A$ with expansion matrix $L_{\zeta}$ and support $F_{1}^{\zeta}$. Let $x\in X_{\zeta}$ be a fixed point of $\zeta$. Consider the constants
		\begin{itemize}
			\item $t=-\log(\Vert L_{\zeta}\Vert)/\log(\Vert L_{\zeta}^{-1}\Vert)$,
			\item $\bar{r}=\Vert L_{\zeta}^{-1}(F_{1}^{\zeta})\Vert/(1-\Vert L_{\zeta}^{-1}\Vert)$,
			\item $\bar{R}=\left\lceil \Vert L_{\zeta}^{-1}\Vert^{|\A|-1}\cdot R_{X_{\zeta}}(9\Vert L_{\zeta}\Vert^{|\A|-1}\bar{r})+4\bar{r}\right\rceil$.
			\item $\bar{n}=\lceil |\A|^{(2\bar{R}+6\bar{r})^{d}}\rceil$.
			
		\end{itemize}
		Then, $R=\Vert L_{\zeta}\Vert^{\bar{n}+|\A|}(\bar{R}+3\bar{r})+2\Vert L_{\zeta}\Vert^{\bar{n}+|\A|}\cdot \Vert F_{1}^{\zeta}\Vert$ is such that for all ${\bm i}, {\bm j}\in \Z^{d}$,
		$$x|_{L_{\zeta}({\bm i})+[B({\bm 0},R)\cap \Z^{d}]}=x|_{{\bm j}+[B({\bm 0},R)\cap \Z^{d}]} \implies  {\bm j}\in L_{\zeta}(\Z^{d}).$$
	\end{proposition} 
	
	The proof of \cref{ComputabilityRecognizability} is an adaptation of the proof of Proposition 3.7 from \cite{cabezas2023homomorphisms}, using some ideas in \cite{durand2017constant}. This bound is far from being sharp, but does not depend on the combinatorics of the substitution. The idea of the proof is to construct sequences of patterns $\zeta^{n}(\texttt{w}_{n})$, $\zeta^{n}(\texttt{u}_{n})$ and $\zeta^{n}(\texttt{v}_{n})$ around two points ${\bm i}_{n}\in L_{\zeta}(\Z^{d})$ and ${\bm j}_{n}\notin L_{\zeta}(\Z^{d})$ such that $\zeta^{n}(\texttt{w}_{n})\sqsubseteq\zeta^{n}(\texttt{u}_{n})\sqsubseteq \zeta^{n}(\texttt{v}_{n})$, where $\supp(\texttt{w}_{n})$, $\supp(\texttt{u}_{n})$ and $\supp(\texttt{v}_{n})$ are fixed for every $n>0$ and big enough to ensure that if two occurrences of $\texttt{w}_{n}$ occur in $\texttt{u}_{n}$, they must be the same. The constants given in \cref{ComputabilityRecognizability} ensure that the arguments are true. 
	
	\begin{proof} Using \cref{Prop:FiniteSetSatisfyingParticularProperties} with $A = \{\bm{0}\}$ and $F = F_1^{\zeta} - F_1^{\zeta}$, there exists a finite set $D\subseteq \Z^{d}$ such that for every $n>0$, $F_{n}^{\zeta}-F_{n}^{\zeta}\subseteq L_{\zeta}^{n}(D)+F_{n}^{\zeta}$. Observe that, by \cref{item:BoundForTheNormforthefiniteset} of \cref{Prop:FiniteSetSatisfyingParticularProperties} we have that $\Vert D\Vert \leq 3\bar{r}$. We prove the statement by contradiction. Assume the contrary. Then, for every $|\A|\leq n\leq\bar{n}+|\A|$ there exist ${\bm i}_{n}\in L_{\zeta}(\Z^{d})$ and ${\bm j}_{n}\notin L_{\zeta}(\Z^{d})$ such that 
		$$x|_{{\bm i}_{n}+L_{\zeta}^{n}(D+[B({\bm 0},\bar{R})\cap \Z^{d}])+F_{n}^{\zeta}]}=x|_{{\bm j}_{n}+L_{\zeta}^{n}(D+[B({\bm 0},\bar{R})\cap \Z^{d}])+F_{n}^{\zeta}}.$$
		
		For any $|\A|\leq n\leq |\A|+\bar{n}$, we consider ${\bm a}_{n}\in \Z^{d}$ and ${\bm f}_{n}\in F_{n}^{\zeta}$ such that ${\bm i}_{n}=L_{\zeta}^{n}({\bm a}_{n})+{\bm f}_{n}$. Note that, by definition of the set $D\Subset \Z^{d}$, we have that $$L_{\zeta}^{n}({\bm a}_{n})+L_{\zeta}^{n}([B({\bm 0},\bar{R})\cap \Z^{d}])+F_{n}^{\zeta}\subseteq {\bm i}_{n}+L_{\zeta}^{n}(D+[B({\bm 0},\bar{R})\cap \Z^{d}])+F_{n}^{\zeta}.$$
		
		Let $\texttt{u}_{n}\in \mathcal{L}_{[B({\bm 0},\bar{R})\cap \Z^{d}]}(X_{\zeta})$ be such that
		$$x|_{L_{\zeta}^{n}({\bm a}_{n})+L_{\zeta}^{n}([B({\bm 0},\bar{R})\cap \Z^{d}])+F_{n}^{\zeta}}=\zeta^{n}(\texttt{u}_{n})=x|_{({\bm j}_{n}-{\bm f}_{n})+L_{\zeta}^{n}([B({\bm 0},\bar{R})\cap \Z^{d}])+F_{n}^{\zeta}}.$$

		\begin{figure}[H]
			\centering
			\begin{tikzpicture}[scale=0.3]
				\node(b1) at (5, 6.5) [scale=1] {$\zeta^{n}(\texttt{u}_{n}))$};
				
				\node(b2) at (5,5.5) [scale=1]{$\cdots$};
				
				\node(b3) at (5,4.5) [scale=1]{$\cdots$};
				\node(b4) at (5,3.5) [scale=1]{$\cdots$};
				\node(b5) at (5,2.5) [scale=1]{$\cdots$};
				\node(b6) at (5,1.5) [scale=1]{$\cdots$};
				
				\node(b7) at (5,0.5) [scale=1]{$\cdots$};
				
				\path[thin](0,0) edge (10,0);
				\path[thin](10.,0.) edge (10.,6.);
				\path[thin](10.,6.) edge (0.,6.);
				\path[thin](0.,6.) edge (0.,0.);
				\path[thin](2.,6.) edge (2.,0.);
				\path[thin](8.,6.) edge (8.,0.);
				\path[thin](0.,1.) edge (10.,1.);
				\path[thin](0.,2.) edge (10.,2.);
				\path[thin](0.,4.) edge (10.,4.);
				\path[thin](0.,5.) edge (10.,5.);
				
				\node(b8) at (25, 1.5) [scale=1,blue] {$\zeta^{n}(\texttt{u}_{n}))$};
				
				\node(b9) at (25,0.5) [scale=1]{$\cdots$};
				
				\node(b10) at (25,-0.5) [scale=1]{$\cdots$};
				\node(b11) at (25,-1.5) [scale=1]{$\cdots$};
				\node(b12) at (25,-2.5) [scale=1]{$\cdots$};
				\node(b13) at (25,-3.5) [scale=1]{$\cdots$};
				
				\node(b14) at (25,-4.5) [scale=1]{$\cdots$};
				
				\path[thin,blue](20,-5) edge (30,-5);
				\path[thin,blue](30.,-5.) edge (30.,1);
				\path[thin,blue](30.,1.) edge (20.,1.);
				\path[thin,blue](20.,1.) edge (20.,-5.);
				\path[thin,blue](22.,1.) edge (22.,-5.);
				\path[thin,blue](28.,1.) edge (28.,-5.);
				\path[thin,blue](20.,-4.) edge (30.,-4.);
				\path[thin,blue](20.,-3) edge (30.,-3.);
				\path[thin,blue](20.,-1) edge (30.,-1.);
				\path[thin,blue](20.,0) edge (30.,0.);
				
				\draw[->,thin] (0,0) -- (20,-5);
				
				\node(b15) at (6,3.4)[scale=1]{$L_{\zeta}^{n}({\bm a}_{n})$};
				\node(b15) at (5.6,2.8)[scale=0.5]{$\times$};
				
				\node(b16) at (26,-1.8)[scale=1]{${\bm j}_{n}-{\bm f}_{n}$};
				\node(b15) at (25.6,-2.2)[scale=0.5]{$\times$};
				
				\node(b17) at (11,-2.3)[scale=1,rotate=346]{${\bm j}_{n}-{\bm i}_{n}$};				
			\end{tikzpicture}
			\caption{Illustration of the pattern $\zeta^{n}(\texttt{u}_{n})$ around the coordinates $L_{\zeta}^{n}({\bm a}_{n})$ (black) and ${\bm j}_{n}-{\bm f}_{n}$ (blue).}
			\label{tauzetauandtauzetav}
		\end{figure}	
		
		Observe that ${\bm j}_{n}-{\bm f}_{n}$ is not necessarily in $L_{\zeta}^{n}(\Z^{d})$, so we set ${\bm b}_{n}\in \Z^{d}$ and ${\bm g}_{n}\in F_{n}^{\zeta}$ such that ${\bm j}_{n}-{\bm f}_{n}=L_{\zeta}^{n}({\bm b}_{n})+{\bm g}_{n}$. Now, for any $n>0$ and $E\subseteq \Z^{d}$ we define the following sets
		\begin{align*}
			G_{n,E} & =\{{\bm n}\in \Z^{d}\colon (L_{\zeta}^{n}({\bm n})+F_{n}^{\zeta})\cap ({\bm j}_{n}-{\bm f}_{n})+L_{\zeta}^{n}(E)+F_{n}^{\zeta}\neq \emptyset\}\\
			H_{n,E} & = \{{\bm n}\in \Z^{d}\colon (L_{\zeta}^{n}({\bm n})+F_{n}^{\zeta})\subseteq ({\bm j}_{n}-{\bm f}_{n})+L_{\zeta}^{n}(E)+F_{n}^{\zeta}\}.
		\end{align*}	
		
		By definition, we have for any $n>0$ and any $E\subseteq \Z^{d}$, that $H_{n,E}$ is included in $G_{n,E}$. Since $x=\zeta(x)$, there exist patterns $\texttt{v}_{n}\in \mathcal{L}_{G_{n,[B({\bm 0},\bar{R})\cap \Z^{d}]}-{\bm b}_{n}}(X_{\zeta})$, $\texttt{w}_{n}\in \mathcal{L}_{H_{n,[B({\bm 0},\bar{R})\cap \Z^{d}]}-{\bm b}_{n}}(X_{\zeta})$, with $L_{\zeta}^{n}({\bm b}_{n})$ being an occurrence of $\zeta^{n}(\texttt{v}_{n})$ and $\zeta^{n}(\texttt{w}_{n})$ in $x$, such that $\zeta^{n}(\texttt{w}_{n})$ occurs in $\zeta^{n}(\texttt{u}_{n})$ and $\zeta^{n}(\texttt{u}_{n})$ occurs in $\zeta^{n}(\texttt{v}_{n})$ as illustrated in \cref{zetavyzetau}:
		
		\begin{figure}[H]
			\centering
			\begin{tikzpicture}[scale=0.6]
				\node(b1) at (5, 6.5) [scale=1] {$\zeta^{n}(\texttt{v}_{n})$};
				
				\node(b2) at (9.8, 3) [scale=1] {\textcolor{blue}{$\zeta^{n}(\texttt{u}_{n})$}};
				
				\node(b3) at (5,5.5) [scale=1]{$\cdots$};
				
				\node(b4) at (5,4.5) [scale=1]{$\cdots$};
				\node(b5) at (5,3.5) [scale=1]{$\cdots$};
				\node(b6) at (5,2.5) [scale=1]{$\cdots$};
				\node(b7) at (5,1.5) [scale=1]{$\cdots$};
				
				\node(b8) at (5,0.5) [scale=1]{$\cdots$};
				
				\path[thin](0,0) edge (10,0);
				\path[thin](10.,0.) edge (10.,6.);
				\path[thin](10.,6.) edge (0.,6.);
				\path[thin](0.,6.) edge (0.,0.);
				\path[thin](2.,6.) edge (2.,0.);
				\path[thin](8.,6.) edge (8.,0.);
				\path[thin](0.,1.) edge (10.,1.);
				\path[thin](0.,2.) edge (10.,2.);
				\path[thin](0.,4.) edge (10.,4.);
				\path[thin](0.,5.) edge (10.,5.);
				\path[thin,blue](0.6,0.3) edge (8.6,0.3);
				\path[thin,blue](8.6,0.3) edge (8.6,5.3);
				\path[thin,blue](8.6,5.3) edge (0.6,5.3);
				\path[thin,blue](0.6,5.3) edge (0.6,0.3);
				\path[thin,blue](0.6,1.3) edge (8.6,1.3);
				\path[thin,blue](0.6,2.3) edge (8.6,2.3);
				\path[thin,blue](0.6,4.3) edge (8.6,4.3);
				\path[thin,blue](2.6,0.3) edge (2.6,5.3);
				\path[thin,blue](6.6,0.3) edge (6.6,5.3);
				
				\node(b10) at (5.5,3.2) [scale=1]{${\bm j}_{n}-{\bm f}_{n}$};
				\node(b11) at (5.25,2.85)[scale=1]{$\times$};
				
				\node(b9) at (0.9, 3) [scale=1] {\textcolor{red}{$\zeta^{n}(\texttt{w}_{n})$}};
				
				\path[line width=1.5pt,red](2,1) edge (2,5);
				\path[line width=1.5pt,red](2,1) edge (8,1);
				\path[line width=1.5pt,red](2,5)edge(8,5);
				\path[line width=1.5pt,red](8,1)edge(8,5);
				
			\end{tikzpicture}
			\caption{Illustration of the patterns $\zeta^{n}(\texttt{w}_{n})$, $\zeta^{n}(\texttt{u}_{n})$ and $\zeta^{n}(\texttt{v}_{n})$ around ${\bm j}_{n}-{\bm f}_{n}$.}
			\label{zetavyzetau}
		\end{figure}
		
		\setcounter{claim}{0}
		\begin{claim}\label{ClaimForRecognizability1}
			For any $n>0$, ${\bm b}_{n}$ belongs to $H_{n,[B({\bm 0},\bar{R})\cap \Z^{d}]}$, and $(G_{n,[B({\bm 0},\bar{R})\cap \Z^{d}]}-{\bm b}_{n})$ is a bounded set.
		\end{claim}
		
		\begin{proof}[Proof of \cref{ClaimForRecognizability1}]
			Note that ${\bm b}_{n}\in H_{n,[B({\bm 0},\bar{R})\cap \Z^{d}]}$ if and only if
			$$F_{n}^{\zeta}-{\bm g}_{n}\subseteq L_{\zeta}^{n}([B({\bm 0},\bar{R})\cap \Z^{d}])+F_{n}^{\zeta},$$ 
			
			\noindent which is true since $\bar{R}\geq \Vert D\Vert$. Now, set ${\bm m}\in (G_{n,[B({\bm 0},\bar{R})\cap \Z^{d}]}-{\bm b}_{n})$, i.e., there exists ${\bm h}_{n}\in F_{n}^{\zeta}$, ${\bm r}_{n}\in [B({\bm 0},\bar{R})\cap \Z^{d}]$ and ${\bm l}_{n}\in F_{n}^{\zeta}$ such that $$L_{\zeta}^{n}({\bm m})+{\bm h}_{n}=L_{\zeta}^{n}({\bm b}_{n})+{\bm g}_{n}+L_{\zeta}^{n}({\bm r}_{n})+{\bm l}_{n},$$
			
			\noindent i.e., ${\bm m}-{\bm b}_{n}={\bm r}_{n}+L_{\zeta}^{-n}({\bm g}_{n}+{\bm l}_{n}-{\bm h}_{n})$, which implies that $\Vert {\bm m}-{\bm b}_{n}\Vert \leq \bar{R}+\Vert L_{\zeta}^{-n}({\bm g}_{n}+{\bm l}_{n}-{\bm h}_{n})\Vert$. Since $\Vert L_{\zeta}^{-n}({\bm g}_{n}+{\bm l}_{n}-{\bm h}_{n})\Vert\leq 3\bar{r}$, we conclude that $\Vert {\bm m}-{\bm b}_{n}\Vert \leq \bar{R}+3\bar{r}$.
		\end{proof}
		
		Since $\Vert G_{n,[B({\bm 0},\bar{R})\cap \Z^{d}]}-{\bm b}_{n}\Vert \leq \bar{R}+3\bar{r}$, as in the proof of \cref{Lemma:UpperBoundComplexityFunction}, we have that $|\mathcal{L}_{G_{n,[B({\bm 0},\bar{R})\cap \Z^{d}]-{\bm b}_{n}}}(X_{\zeta})|\leq |\A|^{(2\bar{R}+6\bar{r})^{d}}$. 
		Since $\bar{n}\geq |\mathcal{L}_{G_{n,[B({\bm 0},\bar{R})\cap \Z^{d}]-{\bm b}_{n}}}(X_{\zeta})|$, there are two indices $|\A|<n<m\leq \bar{n}+|\A|$, some finite sets $G,H \Subset \Z^{d}$ such that
		\begin{align*}
			G=(G_{n,[B({\bm 0},r)\cap \Z^{d}]}-{\bm b}_{n})=(G_{m,[B({\bm 0},r)\cap \Z^{d}]}-{\bm b}_{m})\\
			H=(H_{n,[B({\bm 0},r)\cap \Z^{d}]}-{\bm b}_{n})=(H_{m,[B({\bm 0},r)\cap \Z^{d}]}-{\bm b}_{m})
		\end{align*}
		
		\noindent and some patterns $\texttt{u}\in \mathcal{L}_{[B({\bm 0},\bar{R})\cap \Z^{d}]}(X_{\zeta}),\texttt{v}\in \mathcal{L}_{G}(X_{\zeta})$ and $\texttt{w}\in \mathcal{L}_{H}(X_{\zeta})$ such that $\texttt{u}=\texttt{u}_{n}=\texttt{u}_{m}$, $\texttt{v}=\texttt{v}_{n}=\texttt{v}_{m}$ and $\texttt{w}=\texttt{w}_{n}=\texttt{w}_{m}$. Set $$\texttt{a}_{n}=x|_{({\bm j}_{n}-{\bm f}_{n})+L_{\zeta}^{n}([B({\bm 0},\bar{R})\cap \Z^{d}])+F_{n}^{\zeta}\setminus(L_{\zeta}^{n}({\bm b}_{n})+L_{\zeta}^{n}(H)+F_{n}^{\zeta})},$$
		
		\noindent i.e., $\texttt{a}_{n}$ is the pattern whose support is equal to $\supp(\zeta^{n}(\texttt{u}))\setminus \supp(\zeta^{n}(\texttt{w}))$,  as illustrated in \cref{zetavyzetauyzetaw}:
		
		\begin{figure}[H]
			\centering
			\begin{tikzpicture}[scale=0.6]
				\node(b1) at (5, 6.5) [scale=1] {$\zeta^{n}(\texttt{v})$};
				
				\node(b2) at (9.9, 3) [scale=1] {\textcolor{blue}{$\zeta^{n}(\texttt{u})$}};
				
				\node(b9) at (0.9, 3) [scale=1] {\textcolor{red}{$\zeta^{n}(\texttt{w})$}};
				
				\node(b3) at (5,5.5) [scale=1]{$\cdots$};
				
				\node(b4) at (5,4.5) [scale=1]{$\cdots$};
				\node(b5) at (5,3.5) [scale=1]{$\cdots$};
				\node(b6) at (5,2.5) [scale=1]{$\cdots$};
				\node(b7) at (5,1.5) [scale=1]{$\cdots$};

				\path[thin](0,0) edge (10,0);
				\path[thin](10.,0.) edge (10.,6.);
				\path[thin](10.,6.) edge (0.,6.);
				\path[thin](0.,6.) edge (0.,0.);
				\path[thin](2.,6.) edge (2.,0.);
				\path[thin](8.,6.) edge (8.,0.);
				\path[thin](0.,1.) edge (10.,1.);
				\path[thin](0.,2.) edge (10.,2.);
				\path[thin](0.,4.) edge (10.,4.);
				\path[thin](0.,5.) edge (10.,5.);
				\path[line width=1.5pt,red](2,1) edge (2,5);
				\path[line width=1.5pt,red](2,1) edge (8,1);
				\path[line width=1.5pt,red](2,5)edge(8,5);
				\path[line width=1.5pt,red](8,1)edge(8,5);
				
				\draw[fill=zzttqq,fill opacity=0.10000000149011612] (0.6,0.3) rectangle (2,5.3);
				\draw[fill=zzttqq,fill opacity=0.10000000149011612] (2,5) rectangle (8.6,5.3);
				\draw[fill=zzttqq,fill opacity=0.10000000149011612] (8,0.3) rectangle (8.6,5);
				\draw[fill=zzttqq,fill opacity=0.10000000149011612] (2,0.3) rectangle (8,1);
				
				\node(b8) at (9.4,0.6) [scale=1]{\textcolor{zzttqq}{$\texttt{a}_{n}$}};
				
				\path[thin,blue](0.6,0.3) edge (8.6,0.3);
				\path[thin,blue](8.6,0.3) edge (8.6,5.3);
				\path[thin,blue](8.6,5.3) edge (0.6,5.3);
				\path[thin,blue](0.6,5.3) edge (0.6,0.3);
				\path[thin,blue](0.6,1.3) edge (8.6,1.3);
				\path[thin,blue](0.6,2.3) edge (8.6,2.3);
				\path[thin,blue](0.6,4.3) edge (8.6,4.3);
				\path[thin,blue](2.6,0.3) edge (2.6,5.3);
				\path[thin,blue](6.6,0.3) edge (6.6,5.3);
				
				\node(b9) at (3.5,2.6)[scale=1]{$L_{\zeta}^{n}({\bm b}_{n})$};
				\node(b12) at (4.15,3)[scale=1]{$\times$};
				
				\node(b10) at (6.5,3) [scale=1]{${\bm j}_{n}-{\bm f}_{n}$};
				\node(b11) at (5.25,2.85)[scale=1]{$\times$};
			\end{tikzpicture}
			\caption{Illustration of the patterns $\zeta^{n}(\texttt{w})$ and $\texttt{a}_{n}$ around $({\bm j}_{n}-{\bm f}_{n})$.}
			\label{zetavyzetauyzetaw}
		\end{figure}

		Applying $\zeta^{m-n}$ to $\zeta^{n}(\texttt{u})$, we obtain the patterns $\zeta^{m}(\texttt{a}_{n})$ and $\zeta^{m-n}(\zeta^{n}(\texttt{w}))=\zeta^{m}(\texttt{w})$. Note that 
		$$\supp(\zeta^{m}(\texttt{u}))=\supp(\texttt{a}_{m})\cupdot\supp(\zeta^{m}(\texttt{w}))=\supp(\zeta^{m-n}(\texttt{a}_{n}))\cupdot\supp(\zeta^{m}(\texttt{w})).$$
		
		If $\zeta^{m-n}(\texttt{a}_{n})$ and $\texttt{a}_{m}$ are different, then the pattern $\zeta^m(\texttt{u})$ contains two occurrences of $\zeta^{m}(\texttt{w})$. We will use the repulsion property (\cref{RepulsionProperty}) to get the contradiction, i.e., these two occurrences are the same. To do this, we need the following result
		
		\begin{claim}\label{ClaimAnotherRecognizability}
			For any $n>0$ and any $E\subseteq \Z^{d}$, the set $G_{n,E}$ is included in $H_{n,E}+C_{L_{\zeta},F_{1}^{\zeta}}+C_{L_{\zeta}+F_{1}^{\zeta}}+D$.
		\end{claim}
		
		\begin{proof}[Proof of \cref{ClaimAnotherRecognizability}]
			First, we prove that for any $n>0$ and $E\Subset \Z^{d}$, we have that $G_{n,E}\subseteq H_{n,\left(E+C_{L_{\zeta},F_{1}^{\zeta}}+C_{L_{\zeta},F_{1}^{\zeta}}\right)}+D$. Indeed, set ${\bm m}\in G_{n,E}$. Then, there exists ${\bm h}_{n}\in F_{n}^{\zeta}$, ${\bm e}_{n}\in E$, ${\bm l}_{n}\in F_{n}^{\zeta}$ such that $L_{\zeta}^{n}({\bm m})+{\bm h}_{n}=L_{\zeta}^{n}({\bm b}_{n})+{\bm g}_{n}+L_{\zeta}^{n}({\bm e}_{n})+{\bm l}_{n}$. Set ${\bm d}_{n}\in D$ such that ${\bm l}_{n}-{\bm h}_{n}+{\bm g}_{n}=L_{\zeta}^{n}({\bm d}_{n})$. Hence ${\bm m}={\bm b}_{n}+{\bm e}_{n}+{\bm d}_{n}$. We prove that ${\bm m}-{\bm d}_{n}\in H_{n,E+C_{L_{\zeta},F_{1}^{\zeta}}+C_{L_{\zeta},F_{1}^{\zeta}}}$.
			
			Now, set ${\bm o}_{n}\in F_{n}^{\zeta}$. Then
			\begin{align*}
				L_{\zeta}^{n}({\bm m}-{\bm d}_{n})+{\bm o}_{n} & =L_{\zeta}^{n}({\bm m})+{\bm h}_{n}-L_{\zeta}^{n}({\bm d}_{n})-{\bm h}_{n}+{\bm o}_{n}\\
				& = L_{\zeta}^{n}({\bm b}_{n})+{\bm g}_{n}+L_{\zeta}^{n}({\bm e}_{n})+{\bm l}_{n}-L_{\zeta}^{n}({\bm d}_{n})-{\bm h}_{n}+{\bm o}_{n}\\
				& = L_{\zeta}^{n}({\bm b}_{n})+L_{\zeta}^{n}({\bm e}_{n})+{\bm o}_{n}\\
			\end{align*} 
			Let ${\bm q}_{n}\in F_{n}^{\zeta}$ and ${\bm c}_{n}\in C_{L_{\zeta},F_{1}^{\zeta}}$ be such that ${\bm g}_{n}+{\bm q}_{n}=L_{\zeta}^{n}({\bm c}_{n})$. We get that $L_{\zeta}^{n}({\bm m}-{\bm d}_{n})+{\bm o}_{n}=L_{\zeta}^{n}({\bm b}_{n})+{\bm g}_{n}+L_{\zeta}^{n}({\bm e}_{n}+{\bm c}_{n})+({\bm o}_{n}+{\bm q}_{n})$. Since $F_{n}^{\zeta}+{\bm q}_{n}\subseteq L_{\zeta}^{n}(C_{L_{\zeta},F_{1}^{\zeta}})+F_{n}^{\zeta}$, we conclude that ${\bm m}\in H_{n,E+C_{L_{\zeta},F_{1}^{\zeta}}+C_{L_{\zeta},F_{1}^{\zeta}}}+D$.
			
			To finish the proof, we note that a straightforward computation shows that for any $n>0$ and $A,B\Subset \Z^{d}$, we have that $H_{n,A+B}\subseteq H_{n,A}+B$. We then, conclude that $G_{n,E}\subseteq H_{n,E}+C_{L_{\zeta},F_{1}^{\zeta}}+C_{L_{\zeta},F_{1}^{\zeta}}+D$.
			
		\end{proof}

		By \cref{PropositionForNonInyectiveSubstitutions}, these patterns come from two patterns $\texttt{w}_{1},\texttt{w}_{2}\in \mathcal{L}_{H}(X_{\zeta})$ such that $\zeta^{|\A|-1}(\texttt{w}_{1})=\zeta^{|\A|-1}(\texttt{w}_{2})=\zeta^{|\A|-1}(\texttt{w})$, occurring in $\zeta^{|\A|-1}(\texttt{v})$. Thanks to \cref{ClaimAnotherRecognizability} and the fact that $\Vert C_{L_{\zeta},F_{1}^{\zeta}}+C_{L_{\zeta},F_{1}^{\zeta}}+D\Vert\leq 9\bar{r}$, the difference between these two occurrences is included in $L_{\zeta}^{|\A|-1}([B({\bm 0},9\bar{r})\cap \Z^{d}])$, so the distance is smaller than $9\Vert L_{\zeta}\Vert^{|\A|-1}\bar{r}$.
		
		\begin{claim}\label{ClaimFOrIterationsOfRadius}
			For any $r>0$ and any $n>0$, we have that $[B({\bm 0},r-3\bar{r})\cap \Z^{d}]\subseteq H_{n,[B({\bm 0},r)\cap \Z^{d}]}$ and $[L_{\zeta}^{n}(B({\bm 0},r))\cap \Z^{d}]\subseteq L_{\zeta}^{n}([B({\bm 0},r+\bar{r})\cap \Z^{d}])+F_{n}^{\zeta}$.
		\end{claim}

		\begin{proof}[Proof of \cref{ClaimFOrIterationsOfRadius}]
			Set ${\bm n}\in [B({\bm 0},r-3\bar{r})\cap \Z^{d}]$ and ${\bm h}_{n}, {\bm l}_{n}\in F_{n}^{\zeta}$. Then, we write
			\begin{align*}
				L_{\zeta}^{n}({\bm n})+{\bm h}_{n} & =({\bm j}_{n}-{\bm f}_{n})+L_{\zeta}^{n}({\bm r}_{n}-{\bm b}_{n})+{\bm l}_{n}\\
				& = L_{\zeta}^{n}({\bm r}_{n})+{\bm g}_{n}+{\bm l}_{n}
			\end{align*}
			
			\noindent for some ${\bm r}_{n}\in \Z^{d}$. We note that there exists ${\bm c}\in C_{L_{\zeta},F_{1}^{\zeta}}$ such that $L_{\zeta}^{n}({\bm c})+{\bm h}_{n}={\bm g}_{n}+{\bm l}_{n}$ and we conclude that ${\bm m}={\bm r}_{n}+{\bm c}$. Since $\Vert {\bm m}\Vert \leq r-3\bar{r}$ and $\Vert {\bm c}\Vert \leq 3\bar{r}$, we conclude that ${\bm r}_{n}\leq r$. Now we prove the second inclusion.
			
			Set ${\bm n}\in L_{\zeta}^{n}(B({\bm 0},r))\cap \Z^{d}$. Then, there exists ${\bm m}_{1}\in \Z^{d}$ and ${\bm f}\in F_{n}^{\zeta}$ such that ${\bm m}=L_{\zeta}^{n}({\bm m}_{1})+{\bm f}$, which implies that $\Vert {\bm m}_{1}+L_{\zeta}^{-n}({\bm f})\Vert \leq r$. We then get that 
			$$\Vert {\bm m}_{1}\Vert\leq r+\Vert L_{\zeta}^{-n}({\bm f})\Vert \leq r+\bar{r}.$$
		\end{proof}
		
		By \cref{ClaimFOrIterationsOfRadius}, we have that
		$$[L_{\zeta}^{|\A|-1}(B({\bm 0},\bar{R}-4{\bar{r}}))\cap \Z^{d}] \subseteq L_{\zeta}^{|\A|-1}([B({\bm 0},\bar{R}-3{\bar{r}})\cap \Z^{d}])+F_{|\A|-1}^{\zeta}\subseteq \supp(\zeta^{|\A|-1}(\texttt{w})),$$
		
		\noindent so $\supp(\zeta^{|\A|-1}(\texttt{w}))$ contains a discrete ball of radius $(\bar{R}-4\bar{r})/\Vert L_{\zeta}^{-1}\Vert^{|\A|-1}$. By the repulsion property (\cref{RepulsionProperty}), this is a contradiction. Indeed, by the choice of $\bar{R}$, we note that
		\begin{align*}
			\dfrac{1}{\Vert L^{-1}\Vert^{|\A|-1}}(\bar{R}-4{\bar r})\geq \dfrac{1}{\Vert L^{-1}\Vert^{|\A|-1}}9^{t}(\Vert L_{\zeta}\Vert\cdot R_{X_{\zeta}}(9\Vert L_{\zeta}\Vert^{|\A|-1}\bar{r})).
		\end{align*}
		
		\noindent so  $\zeta^{m-n}(\texttt{a}_{n})=\texttt{a}_{m}$ as illustrated in \cref{zetavyzetauyzetawlm-njn}:
		
		\begin{figure}[H]
			\centering
			\begin{tikzpicture}[scale=0.6]
				\node(b1) at (5, 6.5) [scale=1] {$\zeta^{m}(\texttt{v})$};
				
				\node(b2) at (9.9, 3) [scale=1] {\textcolor{blue}{$\zeta^{m}(\texttt{u})$}};
				
				\node(b9) at (0.9, 3) [scale=1] {\textcolor{red}{$\zeta^{m}(\texttt{w})$}};
				
				\node(b3) at (5,5.5) [scale=1]{$\cdots$};
				
				\node(b4) at (5,4.5) [scale=1]{$\cdots$};
				\node(b5) at (5,3.5) [scale=1]{$\cdots$};
				\node(b6) at (5,2.5) [scale=1]{$\cdots$};
				\node(b7) at (5,1.5) [scale=1]{$\cdots$};
				
				\node(b8) at (11.5,0.6) [scale=1]{\textcolor{zzttqq}{$\zeta^{m-n}(\texttt{a}_{n})=\texttt{a}_{m}$}};
				
				\path[thin](0,0) edge (10,0);
				\path[thin](10.,0.) edge (10.,6.);
				\path[thin](10.,6.) edge (0.,6.);
				\path[thin](0.,6.) edge (0.,0.);
				\path[thin](2.,6.) edge (2.,0.);
				\path[thin](8.,6.) edge (8.,0.);
				\path[thin](0.,1.) edge (10.,1.);
				\path[thin](0.,2.) edge (10.,2.);
				\path[thin](0.,4.) edge (10.,4.);
				\path[thin](0.,5.) edge (10.,5.);
				\path[line width=1.5pt,red](2,1) edge (2,5);
				\path[line width=1.5pt,red](2,1) edge (8,1);
				\path[line width=1.5pt,red](2,5)edge(8,5);
				\path[line width=1.5pt,red](8,1)edge(8,5);
				
				\draw[fill=zzttqq,fill opacity=0.10000000149011612] (0.6,0.3) rectangle (2,5.3);
				\draw[fill=zzttqq,fill opacity=0.10000000149011612] (2,5) rectangle (8.6,5.3);
				\draw[fill=zzttqq,fill opacity=0.10000000149011612] (8,0.3) rectangle (8.6,5);
				\draw[fill=zzttqq,fill opacity=0.10000000149011612] (2,0.3) rectangle (8,1);
				
				\path[thin,blue](0.6,0.3) edge (8.6,0.3);
				\path[thin,blue](8.6,0.3) edge (8.6,5.3);
				\path[thin,blue](8.6,5.3) edge (0.6,5.3);
				\path[thin,blue](0.6,5.3) edge (0.6,0.3);
				\path[thin,blue](0.6,1.3) edge (8.6,1.3);
				\path[thin,blue](0.6,2.3) edge (8.6,2.3);
				\path[thin,blue](0.6,4.3) edge (8.6,4.3);
				\path[thin,blue](2.6,0.3) edge (2.6,5.3);
				\path[thin,blue](6.6,0.3) edge (6.6,5.3);
				
				\node(b10) at (6,3) [scale=0.8]{$L_{\zeta}^{m-n}({\bm j}_{n})$};
				
				\node(b12) at (1.85,0.8)[scale=1]{$\times$};
				\node(b11) at (5.25,2.85)[scale=1]{$\times$};
			\end{tikzpicture}
			\caption{Illustration of the patterns $\zeta^{m-n}(\texttt{a}_{n})$ in $L_{\zeta}^{m-n}({\bm j}_{n})$.}
			\label{zetavyzetauyzetawlm-njn}
		\end{figure}
		
		To finish the proof, we note that since $\zeta^{n}(\texttt{u})\sqsubseteq \zeta^{n}(\texttt{v})$, there exists ${\bm p}_{m}\in L_{\zeta}^{n}({\bm b}_{m})+L_{\zeta}^{n}(G)+F_{n}^{\zeta}$ such that $x|_{{\bm p}_{m}+L_{\zeta}^{n}([B({\bm 0},\bar{R})\cap \Z^{d}])+F_{n}^{\zeta}}=\zeta^{n}(\texttt{u})$, which implies that $x|_{L_{\zeta}^{m-n}({\bm p}_{m})+L_{\zeta}^{m}([B({\bm 0},\bar{R})\cap \Z^{d}])+F_{m}^{\zeta}}=\zeta^{m}(\texttt{u})$. Using the fact that $\zeta^{m-n}(\texttt{a}_{n})=\texttt{a}_{m}$, we get that $L_{\zeta}^{m-n}({\bm p}_{m})+L_{\zeta}^{m}([B({\bm 0},\bar{R})\cap \Z^{d}])+F_{m}^{\zeta}=({\bm j}_{m}-{\bm f}_{m})+L_{\zeta}^{m}([B({\bm 0},\bar{R})\cap \Z^{d}])+F_{m}^{\zeta}$, i.e., ${\bm j}_{m}-{\bm f}_{m}=L_{\zeta}^{m-n}({\bm p}_{m})\in L_{\zeta}^{m}(\Z^{d})$. Since ${\bm i}_{m}, {\bm f}_{m}\in L_{\zeta}(\Z^{d})$ we conclude that ${\bm j}_{m}\in L_{\zeta}(\Z^{d})$.
	\end{proof}
	
	In \cref{ComputabilityRecognizability} we compute a constant such that we can recognize patterns of the form $\zeta(a)$ for $a\in \A$. But it does not give information on the letter such that the pattern $\zeta(a)$ comes from. If the substitution is injective in letters, this result is enough to get a complete recognizability property. To finish the proof of \cref{UpperBoundRecognizability}, we need to deal with the case where the substitution is not injective. For this case, we prove the second step of the recognizability property.
	
	\begin{proposition}\label{RecognizabilitySecondStep}[Second step of recognizability: uniqueness of pre-images]
		Let $\zeta$ be an aperiodic primitive substitution and $x\in X_{\zeta}$ be a fixed point of $\zeta$. 
		Consider $R_{|\A|}>0$ the constant of \cref{ComputabilityRecognizability} associated with $\zeta^{|\A|}$, i.e.,
		$$x|_{L^{|\A|}_{\zeta}({\bm i})+[B({\bm 0},R_{|\A|})\cap \Z^{d}]}=x|_{{\bm j}+[B({\bm 0},R_{|\A|})\cap \Z^{d}]} \implies  {\bm j}\in L^{|\A|}_{\zeta}(\Z^{d}).$$
		Consider also $R=R_{|\A|}+2\Vert F_{1}^{\zeta}\Vert\cdot  \Vert L_{\zeta}\Vert^{|\A|}$. Then, for any ${\bm i},{\bm j}\in \Z^{d}$
		$$x|_{[B(L_{\zeta}({\bm i}),R)\cap \Z^{d}]}=x|_{[B(L_{\zeta}({\bm j}),R)\cap \Z^{d}]} \implies x_{{\bm i}}=x_{{\bm j}}.$$
	\end{proposition}
	
	\begin{proof}
		Let ${\bm k}\in \Z^{d}$ and ${\bm f}=\sum\limits_{i=1}^{|\A|-1}L_{\zeta}^{i}({\bm f}_{i})\in F_{|\A|}^{\zeta}$ be such that $L_{\zeta}^{|\A|}({\bm k})+{\bm f}=L_{\zeta}({\bm i})$. Hence, we have that $L_{\zeta}^{|\A|-1}({\bm k})+\sum\limits_{i=1}^{|\A|-1}L_{\zeta}^{i-1}({\bm f}_{i})={\bm i}$. By the definition of $R_{|\A|}>0$, we have the existence of ${\bm m}\in \Z^{d}$ such that $L_{\zeta}^{|\A|}({\bm m})+{\bm f}=L_{\zeta}({\bm j})$, which implies that ${\bm j}= L_{\zeta}^{|\A|-1}({\bm m})+\sum\limits_{i=1}^{|\A|-1}L_{\zeta}^{i-1}({\bm f}_{i})$. Note that, by the definition of $R>0$, we get that $x|_{L_{\zeta}^{|\A|}({\bm k})+F_{|\A|}^{\zeta}}=x|_{L_{\zeta}^{|\A|}({\bm m})+F_{|\A|}^{\zeta}}$. Hence $\zeta^{|\A|}(x_{{\bm k}})=\zeta^{|\A|}(x_{{\bm m}})$, and by \cref{PropositionForNonInyectiveSubstitutions} we also have that $\zeta^{|\A|-1}(x_{{\bm k}})=\zeta^{|\A|-1}(x_{{\bm m}})$. This implies that $x|_{L_{\zeta}^{|\A|-1}({\bm k})+F_{|\A|-1}^{\zeta}}=x|_{L_{\zeta}^{|\A|-1}({\bm m})+F_{|\A|-1}^{\zeta}}$, which let us conclude that $x_{{\bm i}}=x_{{\bm j}}$.
	\end{proof}
	
	\begin{remark}\label{RecognizabilityForPowers}
		We recall that if $R_{\zeta}$ is a constant of recognizability for the first step for $\zeta$, then $2\Vert L_{\zeta}\Vert R_{\zeta}$ is a constant of recognizability for the first step for $\zeta^{2}$. Indeed, note that if 
		$$x|_{L_{\zeta}^{2}({\bm i})+L_{\zeta}([B({\bm 0}),R_{\zeta}]\cap \Z^{d})+[B({\bm 0}),R_{\zeta}]\cap \Z^{d}}=x|_{{\bm m}+L_{\zeta}([B({\bm 0}),R_{\zeta}]\cap \Z^{d})+[B({\bm 0}),R_{\zeta}]\cap \Z^{d}},$$
		
		\noindent then ${\bm m}\in L_{\zeta}^{2}(\Z^{d})$. A straightforward induction shows that for any $n>0$, $2\Vert L_{\zeta}\Vert^{n}R_{\zeta}$ is a constant of recognizability for the first step for $\zeta^{n}$. 
	\end{remark}
	
	\begin{proof}[Proof of \cref{UpperBoundRecognizability}]
		Finally, to get an upper bound we just need to make the computation. By \cref{RecognizabilityForPowers}, we have that if $R_{\zeta}$ is given by \cref{ComputabilityRecognizability} for $\zeta$, then $2\Vert L_{\zeta}\Vert^{|\A|}(R_{\zeta}+\Vert F_{1}^{\zeta}\Vert)$, is a constant of recognizability, for the second step, for $\zeta$. By \cref{ComputabilityRecognizability}, we then get that $\zeta$ is recognizable on $x$ and the constant of recognizability is at most
		
		$$2\Vert L_{\zeta}\Vert^{|\A|}[2\Vert F_{1}^{\zeta}\Vert + \Vert L_{\zeta}\Vert^{{\bar n}+|\A|}(2\Vert F_{1}^{\zeta}\Vert+7\bar{r}+\Vert L_{\zeta}^{-1}\Vert^{|\A|-1}\cdot a\cdot 9^{t}\Vert L_{\zeta}\Vert^{t(|\A|-1)})\bar{r}^{t}].$$
	\end{proof}
	
	\section{Rigidity properties of topological factors between aperiodic minimal substitutive subshifts}\label{Section:RigidityFactorsSubstitutiveSubshifts}
	
	In this section, we study factor maps between multidimensional substitutive subshifts. The main theorem (\cref{TopologicalRigidityOfFactors}) shows that if $\zeta_{1},\zeta_{2}$ are two aperiodic substitutions with the same expansion map $L$ and there exists a factor map $\pi:(X_{\zeta_{1}},S,\Z^{d})\to (X_{\zeta_{2}},S,\Z^{d})$ between two substitutive subshifts, then, there exists another factor map $\phi:(X_{\zeta_{1}},S,\Z^{d})\to (X_{\zeta_{2}},S,\Z^{d})$ given by a local map of a computable bounded radius that only depends on the support and the constant of recognizability. 
	We recall that, by \cref{ConjugationSubstitutionsDifferentFundamentalDomains}, if two substitutions are defined with the same expansion matrix, we can assume, up to conjugacy, that they also share the same support. We then deduce the following consequences: every aperiodic minimal substitutive subshift is coalescent (\cref{Coalescence}) and the quotient group $\Aut(X_{\zeta},S,\Z^{d})/\left\langle S\right\rangle$ is finite, extending the results in \cite{cabezas2023homomorphisms} for the whole class of aperiodic minimal substitutive subshifts, given by constant-shape substitutions. Next, we prove the decidability of the factorization and the isomorphism problem between aperiodic minimal substitutive subshifts (\cref{Thm:DecidabilityFactorizationProblem} and \cref{cor:DecidabilityIsomorphismProblem}). We finish this section proving that aperiodic minimal substitutive subshifts have finitely many aperiodic symbolic factors, up to conjugacy (\cref{Lemma:FinitelyManyAperiodicSymbolicFactors}) and providing an algorithm to obtain a list of the possible injective substitutions factors of an aperiodic substitutive subshift.
	
	\subsection{Factor maps between substitutive subshifts} We prove a multidimensional analog of \cite[Theorem 8.1]{durand2022decidability}: There exists a computable upper bound $R$ such that any factor map between two aperiodic minimal substitutive subshifts is equal to a shift map composed with a sliding block code of radius less than $R$. This was first done in the measurable setting under some extra combinatorial assumptions for the substitutions (in particular for bijective substitutions) in the one-dimensional case in \cite{host1989homomorphismes} and in the multidimensional case in \cite{cabezas2023homomorphisms}. A similar result was also proved for factor maps between two minimal substitutive subshifts of constant-length and Pisot substitutions in \cite{salo2015blockmaps}. This last result was extended in \cite{durand2022decidability} for the whole class of aperiodic minimal substitutive subshifts. They also gave a computable bound for the radius of a factor map between two minimal substitutive subshifts. We prove a similar result of \cite{durand2022decidability} for the multidimensional constant-shape case, where the expansion matrices of the constant-shape substitutions are the same. We start with the following property about factor maps between substitutive subshifts.
	
	\begin{proposition}\label{ConstantModuloMinimality}
		Let $\zeta_{1},\zeta_{2}$ be two aperiodic primitive constant-shape substitutions with the same expansion matrix $L$ and the same support $F_{1}$, and $\phi:(X_{\zeta_{1}},S,\Z^{d})\to(X_{\zeta_{2}},S,\Z^{d})$ be a factor map. Then, for any $n>0$, there exists a unique ${\bm f}_n\in F_{n}$ such that if $x\in \zeta_{1}^n(X_{\zeta_{1}})$, then $\phi(x)\in S^{-{\bm f}_{n}} \zeta_{2}^{n}(X_{\zeta_{2}})$.
	\end{proposition}
	
	\begin{proof}
		Set $x\in \zeta_{1}^{n}(X_{\zeta_{1}})$. Then, there exists ${\bm f}_{n}(x)$ such that $\phi(x)\in S^{-{\bm f}_{n}} \zeta_{2}^{n}(X_{\zeta_{2}})$. We prove that ${\bm f}_{n}(x)$ is constant on $\zeta_{1}^{n}(X_{\zeta_{1}})$. Note that $S^{{\bm n}}x\in \zeta_{1}^{n}(X_{\zeta_{1}})$ if and only if ${\bm n}=L_{\zeta}^{n}({\bm m})$, for some ${\bm m}\in \Z^{d}$. 
		Hence, $\phi(S^{L_{\zeta}^{n}({\bm m})}(x))=S^{L_{\zeta}^{n}({\bm m})}\phi(x)\in S^{-{\bm f}_{n}}\zeta_{2}^{n}(X_{\zeta_{2}})$. Now, if $y\in \zeta_{1}^{n}(X_{\zeta_{1}})$, we use the minimality of $(X_{\zeta},S,\Z^{d})$ to obtain a sequence $(S^{L_{\zeta}^{n}({\bm m}_{p})}(x))_{p>0}$ converging to $y$. By continuity of $\phi$, we conclude that $\phi(y)\in S^{-{\bm f}_{n}(x)}\phi(x)$. 
	\end{proof}
	
	The following theorem states the rigidity properties that factor maps between substitutive subshifts with the same expansion map satisfy. As in the previous sections, we define ${\bar r}=\Vert L^{-1}(F_{1})\Vert/(1-\Vert L^{-1}\Vert)$.
	Observe that in the proofs of the next results, we sometimes need to replace a substitution $\zeta$ by a power $\zeta^n$ of itself. 
	However, in the computations, this modification does not impact the bound $\bar{r}$, i.e., we do not have to replace $\bar{r}$ by $\Vert L^{-n}(F_{n})\Vert/(1-\Vert L^{-n}\Vert)$.
	
	\begin{theorem}\label{TopologicalRigidityOfFactors}
		Let $\zeta_{1},\zeta_{2}$ be two aperiodic primitive constant-shape substitutions with the same expansion matrix $L$ and support $F_{1}$. Suppose there exists a factor map $\phi:(X_{\zeta_{1}},S,\Z^{d})\to(X_{\zeta_{2}},S,\Z^{d})$ with radius $r$. Then, there exists ${\bm j}\in \Z^{d}$ and a factor map $\psi:(X_{\zeta_{1}},S,\Z^{d})\to (X_{\zeta_{2}},S,\Z^{d})$ such that $S^{{\bm j}}\phi=\psi$, satisfying the following two properties:
		
		\begin{enumerate}
			\item The factor map $\psi$ is a sliding block code of radius at most $2\bar{r}+R_{\zeta_{2}}+1$, where $R_{\zeta_{2}}$ is a constant of recognizability for $\zeta_{2}$.
			\item\label{CommutationPropertySlidingBlickCodeSubstitution} There exists an integer $n>0$ and ${\bm f}\in F_{n}$ such that $S^{{\bm f}}\psi \zeta_{1}^{n}=\zeta_{2}^{n}\psi$.
		\end{enumerate}
	\end{theorem}
	
	Our proof follows a similar strategy than in \cite{durand2022decidability} and \cite{cabezas2023homomorphisms}. We then use \cref{TopologicalRigidityOfFactors} to deduce some topological and combinatorial properties of aperiodic primitive constant-shape substitutions.
	
	\begin{proof}
		For any $n>0$, we denote ${\bm f}_{n}(\phi)$ the element of $F_n$ given by \cref{ConstantModuloMinimality}. We recall that the recognizability property implies that the substitution maps $\zeta_{1}^{n},\zeta_{2}^{n}$ are homeomorphisms from $X_{\zeta_{1}}$ to $\zeta_{1}^n(X_{\zeta_{1}})$ and from $X_{\zeta_{2}}$ to $\zeta_{2}^n(X_{\zeta_{2}})$, respectively. Hence, for any $x\in X_{\zeta_{1}}$, there exists a unique point $\phi_{n}(x)\in X_{\zeta_{2}}$ such that
		\begin{equation}\label{eq:definition of phi_n}
			S^{{\bm f}_{n}(\phi)}\phi \zeta_{1}^{n}(x)=\zeta_{2}^{n}(\phi_{n}(x)).        
		\end{equation}
		
		Let us show that $\phi_n$ is also a factor map and that, for large enough $n$, it has a ``small'' radius.
		The map $\phi_{n}$ is obviously continuous, so let us prove that it commutes with the shift. Take ${\bm m}\in \Z^{d}$. We note that
		\begin{align*}
			\zeta_{2}^{n}(\phi_{n}(S^{{\bm m}}x)) & = S^{{\bm f}_{n}(\phi)}\phi\zeta_{1}^{n}(S^{{\bm m}}x)\\
			& = S^{{\bm f}_{n}(\phi)}\phi S^{L^{n}({\bm m})}\zeta_{1}^{n}(x)\\
			& = S^{L^{n}({\bm m})+{\bm f}_{n}(\phi)}\phi \zeta_{1}^{n}(x)\\
			& =	S^{L^{n}({\bm m})}\zeta_{2}^{n}(\phi_{n}(x))\\
			& = \zeta_{2}^{n}(S^{{\bm m}}\phi_{n}(x)).
		\end{align*}
		The map $\zeta_2^n:X_{\zeta_2} \to \zeta_2^n(X_{\zeta_2})$ being injective, proves that $\phi_{n}(S^{{\bm m}}x) = S^{{\bm m}}\phi_{n}(x)$, so the map $\phi_{n}$ is a factor map between $(X_{\zeta_{1}},S,\Z^{d})$ and $(X_{\zeta_{2}},S,\Z^{d})$. 
		
		Let us now show that Equation~\eqref{eq:definition of phi_n} allows to bound the radius of $\phi_n$ for large enough $n$. Roughly, given $x|_{[B({\bm 0},R)\cap \Z^{d}]}$, we first apply $\zeta_1^n$ to compute $\zeta_1^n(x)$ on some big support. Then we lose some information using $\phi$ and the shift $S^{{\bm f}_n}$. Our aim is thus to find $R$ such that, for large enough $n$, the remaining information is large enough to be ``desubstituted'' by $\zeta_2^n$.

		Let $P_{n}\Subset \Z^{d}$ be such that if $x,y\in X_{\zeta_1}$, then 
		\begin{equation}\label{eq:def_of_factor}
			x|_{P_{n}}=y|_{P_{n}} \implies \phi_{n}(x)_{{\bm 0}}=\phi_{n}(y)_{{\bm 0}}.
		\end{equation}
		By definition, we have that $\zeta_{1}^{n}(x)|_{L^{n}(P_{n})+F_{n}}$ is equal to $\zeta_{1}^{n}(y)|_{L^{n}(P_{n})+F_{n}}$. We then obtain that
		\begin{equation}\label{Eq:EqualityOfPatternsRigidty}
			(S^{{\bm f}_{n}(\phi)}\phi\zeta_{1}^{n}(x))|_{(L^{n}(P_{n})+F_{n})^{\circ r}-{\bm f}_{n}(\phi)}=(S^{{\bm f}_{n}(\phi)}\phi\zeta_{1}^{n}(y))|_{(L^{n}(P_{n})+F_{n})^{\circ r}-{\bm f}_{n}(\phi)},
		\end{equation}
		which implies that 
		\[
		\zeta_{2}^{n}(\phi_{n}(x))|_{(L^{n}(P_{n})+F_{n})^{\circ r}-{\bm f}_{n}(\phi)}=\zeta_{2}^{n}(\phi_{n}(y))|_{(L^{n}(P_{n})+F_{n})^{\circ r}-{\bm f}_{n}(\phi)}.
		\]
		Let $R_{\zeta_{2}}$ be a constant of recognizability for $\zeta_{2}$. By \cref{RecognizabilitySecondStep} and \cref{RecognizabilityForPowers} we get that $\phi_n(x)$ and $\phi_n(y)$ coincide on the support
		\[
		[L^{-n}(((L^{n}(P_{n})+F_{n})^{\circ r}-{\bm f}_{n}(\phi))^{\circ \mathcal{R}_{n}})]\cap \Z^{d},
		\]
		where $\mathcal{R}_{n}=\sum\limits_{i=0}^{n-1}L^{i}([B({\bm 0},R_{\zeta_{2}})\cap \Z^{d}])$. By Equation~\eqref{eq:def_of_factor}, for $\phi_{n}$ to be a sliding block code induced by a $P_{n}$-block map, we need that
		\[
		{\bm 0}\in [L^{-n}(((L^{n}(P_{n})+F_{n})^{\circ r}-{\bm f}_{n}(\phi))^{\circ \mathcal{R}_{n}})]\cap \Z^{d}.
		\]
		We have that
		\begin{align}\label{ConditionForAToBeBlockMap}
			{\bm 0}\in [L^{-n}(((L^{n}(P_{n})+F_{n})^{\circ r}-{\bm f}_{n}(\phi))^{\circ \mathcal{R}_{n}})]\cap \Z^{d} & \Leftrightarrow {\bm 0}\in ((L^{n}(P_{n})+F_{n})^{\circ r }-{\bm f}_{n}(\phi))^{\circ \mathcal{R}_{n}}\notag\\
			& \Leftrightarrow \mathcal{R}_{n} \subseteq (L^{n}(P_{n})+F_{n})^{\circ r}-{\bm f}_{n}(\phi)\notag\\
			& \Leftrightarrow {\bm f}_{n}(\phi)+\mathcal{R}_{n} \subseteq (L^{n}(P_{n})+F_{n})^{\circ r}.
		\end{align}
		
		Consider $R = 2\bar{r}+R_{\zeta_{2}}+1$ and let us prove that $P_n = [B({\bm 0},R)\cap \Z^{d}]$ satisfies Equation~\eqref{ConditionForAToBeBlockMap} for large enough $n$.
		Indeed, for all ${\bm r}_n \in \mathcal{R}_n$ and all ${\bm r} \in [B({\bm 0},r)\cap \Z^{d}]$, we can write ${\bm f}_n(\phi)+{\bm r}_n+{\bm r} = L^n({\bm a})+{\bm g}_n$ for some ${\bm a} \in \mathbb{Z}^d$ and ${\bm g}_n \in F_n$. We then get
		\begin{align}
			\nonumber \|{\bm a}\| 
			&\leq
			\|L^{-n}({\bm f}_n(\phi)-{\bm g}_n+{\bm r}_n+{\bm r})\| \\
			&\leq
			\|L^{-n}({\bm f}_n(\phi))\|+\|L^{-n}({\bm g}_n)\|+\|L^{-n}({\bm r}_n)\|+\|L^{-n}({\bm r})\| \label{eq:reductionofradius}.     
		\end{align}

		\noindent and the fact that for $n$ large enough, $\Vert L^{-1}\Vert^{n+1}r\leq 1$, we conclude that for any $n$ large enough,  $\Vert P_{n}\Vert\leq 2\bar{r}+R_{\zeta_{2}}+1$.
		
		As in \cite{cabezas2023homomorphisms,host1989homomorphismes} we note that for any $n\geq 1$, ${\bm f}_{n+1}(\phi)={\bm f}_{n}(\phi)\ (\text{mod}\ L^{n}(\Z^{d}))$. Hence, there exists ${\bm g}\in F_{1}$ such that ${\bm f}_{n+1}(\phi)={\bm f}_{n}(\phi)+L^{n}({\bm g})\ (\text{mod}\ L^{n+1}(\Z^{d}))$. We note that ${\bm g}={\bm f}_{1}(\phi_{n})$. Indeed
		\begin{align}
			\nonumber    S^{{\bm f}_{n+1}(\phi)}\phi \zeta_{1}^{n+1} & = \zeta_{2}^{n+1}(\phi_{n+1}) \\
			\nonumber    S^{{\bm f}_{n+1}(\phi)-{\bm f}_{n}}S^{{\bm f}_{n}(\phi)}\phi \zeta_{1}^{n}\zeta & = \zeta_{2}^{n+1}(\phi_{n+1})\\
			\nonumber    S^{L^{n}({\bm g})}\zeta_{2}^{n}\phi_{n}\zeta_{1} & = \zeta_{2}^{n+1}(\phi_{n+1})\\
			S^{{\bm g}}\phi_{n}\zeta_{1} & =\zeta_{2}(\phi_{n+1}). \label{eq:relation between phi_n's}
		\end{align}
		
		By definition of ${\bm g}$, ${\bm f}_{1}(\phi_{n})$, $\phi_{n+1}$ and $(\phi_{n})_{1}$, we conclude that ${\bm g}={\bm f}_{1}(\phi_{n})$ and $(\phi_{n})_{1}=\phi_{n+1}$. By recurrence, we conclude that for any $n,k\geq 0$, $(\phi_{n})_{k}=\phi_{n+k}$.
		
		To finish the proof, observe that for fixed alphabets $\A$ and $\B$, there exists a finite number of sliding block codes of radius $2\bar{r}+R_{\zeta_{2}}+1$. Thus, there exist two different integers $m,k\geq 0$ such that $\phi_{m}=\phi_{m+k}$.
		
		Let $n\geq m$ be a multiple of $k$. Note that $(\phi_{n})_{k}=\phi_{n+k}=(\phi_{m+k})_{n-m}=(\phi_{m})_{n-m}=\phi_{n}$. This implies that for all $r\in \N$, $\phi_{n}$ is equal to $(\phi_{n})_{rk}$. Since $\phi_{n}$ is equal to $\phi_{2n}$ we denote $\psi=\phi_{n}$ and ${\bm f}={\bm f}_{n}(\psi)$. By definition of ${\bm f}$, we have that $S^{{\bm f}}\psi\zeta_{1}^{n}=\zeta_{2}^{n}\psi$. 
		
		Set ${\bm j}={\bm f}_{n}(\phi)-{\bm f}$, then
		$$S^{{\bm j}}\phi\zeta_{1}^{n}=S^{{\bm f}_{n}(\phi)-{\bm f}}\phi\zeta_{1}^{n}=S^{-{\bm f}}\zeta_{2}^{n}\psi=\psi\zeta_{1}^{n},$$
		this implies that $S^{{\bm j}}\phi$ and $\psi$ coincides on $\zeta_{1}^{n}(X_{\zeta_{1}})$, hence on the whole set $X_{\zeta_{1}}$ by minimality.
	\end{proof}
	
	\subsection{Consequences of \cref{TopologicalRigidityOfFactors}}\label{SectionApplicationsOfMainTheorem}
	
	As a consequence of \cref{TopologicalRigidityOfFactors}, we extend the results on the coalescence and the automorphism group of substitutive subshifts proved in \cite{cabezas2023homomorphisms}. Specifically, we get rid of the reducibility condition for constant-shape substitutions. We recall that a system is \emph{coalescent} if any factor map between $X$ and itself is invertible. This was first proved in \cite{durand2000linearly}
	for one-dimensional linearly recurrent subshifts (in particular aperiodic primitive substitutive subshifts). Multidimensional linearly recurrent substitutive subshifts (such as the self-similar ones) are also coalescent as
	a consequence of a result in \cite{cortez2010linearly}. For an aperiodic primitive constant-shape substitution, we denote $R_{\zeta}$ to be a constant of recognizability for $\zeta$.
	
	Since the set of sliding block codes $2\bar{r}+R_{\zeta}+1$ is finite, we will assume here (up to considering a power of $\zeta$) that if a factor map $\psi$ satisfies Property \ref{CommutationPropertySlidingBlickCodeSubstitution} in \cref{TopologicalRigidityOfFactors}, then it does so for $n=1$, i.e., there exists ${\bm p}\in F_{1}^{\zeta}$ such that $S^{{\bm p}}\psi \zeta=\zeta \psi$.
	
	\subsubsection{Coalescence of substitutive subshifts}\label{CoalescenceAndFactors} The coalescence of aperiodic substitutive subshifts was proved in \cite{cabezas2023homomorphisms} for reduced aperiodic primitive constant-shape substitutions. Here, we proved it for the whole class of aperiodic primitive constant-shape substitutions. As in \cite{cabezas2023homomorphisms}, we use the notion of $\zeta$-\emph{invariant orbits}. An orbit $\mathcal{O}(x,\Z^{d})$ is called $\zeta$-\emph{invariant} if there exists ${\bm j}\in \Z^{d}$ such that $\zeta(x)=S^{{\bm j}}x$, i.e., the orbit is invariant under the action of $\zeta$ in $X_{\zeta}$. Since for every ${\bm n}\in \Z^{d}$ we have that $\zeta\circ S^{{\bm n}}=S^{L_{\zeta}{\bm n}}\circ \zeta$, the definition is independent of the choice of the point in the $\Z^{d}$-orbit of $x$. As an example, the orbit of a fixed point of the substitution map is an example of an invariant orbit. We recall that in \cite{cabezas2023homomorphisms} it was proved that there exist finitely many invariant orbits.
	
	\begin{proposition}\cite[Proposition 3.9]{cabezas2023homomorphisms}\label{FinitelyManyInvariantOrbits}
		Let $\zeta$ be an aperiodic primitive constant-shape substitution. Then, there exist finitely many $\zeta$-invariant orbits in the substitutive subshift $X_{\zeta}$. The bound is explicit and depends only on $d$, $|\A|$, $\Vert L_{\zeta}^{-1}\Vert$, $\Vert F_{1}^{\zeta}\Vert$ and $\det(L_{\zeta}-\id)$.
	\end{proposition} 
	
	We now prove that substitutive subshifts are coalescent.
	
	\begin{proposition}\label{Coalescence}
		Let $\zeta$ be an aperiodic primitive constant-shape substitution. Then, the substitutive subshift $(X_{\zeta},S,\Z^{d})$ is coalescent.
	\end{proposition}
	
	\begin{proof}
		Set $\phi\in \End(X_{\zeta},S,\Z^{d})$. \cref{TopologicalRigidityOfFactors} ensures that there exists ${\bm j}\in \Z^{d}$ such that $S^{{\bm j}}\phi$ is equal to a sliding block code $\psi$ of a fixed radius satisfying $S^{{\bm p}}\psi\zeta=\zeta \psi$, for some ${\bm p}\in F_{1}^{\zeta}$. Let $\bar{x}\in X_{\zeta}$ be in a $\zeta$-invariant orbit, i.e., there exists ${\bm j}\in \Z^{d}$ such that $\zeta(\bar{x})=S^{{\bm j}}\bar{x}$. Note that
		$$\zeta \psi(\bar{x}) = S^{{\bm p}}\psi\zeta(\bar{x})=S^{{\bm p}+{\bm j}}\psi(\bar{x}),$$
		
		\noindent so, if the orbit of $x$ is in a $\zeta$-invariant orbit, then $\psi(x)$ is also in a $\zeta$-invariant orbit. By Proposition \ref{FinitelyManyInvariantOrbits}, there exist finitely many $\zeta$-invariant orbits, hence for $n$ large enough, we can find $x\in X_{\zeta}$ with $x$ and $\psi^{n}(x)$ being in the same orbit, i.e., there exists ${\bm m}\in \Z^{d}$ such that $S^{{\bm m}}\psi^{n}(x)=x$. The minimality of $(X_{\zeta},S,\Z^{d})$ allows us to conclude that $\psi^{n}=S^{-{\bm m}}$. Hence $\psi$ is invertible, which implies that $\phi$ is invertible.
	\end{proof}
	
	\subsubsection{The automorphism group of substitutive subshifts.} The rigidity properties of the topological factors between substitutive subshifts also allow us to conclude that the group $\Aut(X_{\zeta},S,\Z^{d})/\left\langle S\right\rangle$ is finite, since any element in $\Aut(X_{\zeta},S,\Z^{d})$ can be represented as an automorphism with radius $2\bar{r}+R_{\zeta}+1$.
	
	\begin{proposition}\label{AutomoprhismVirtuallyZd}
		Let $(X_{\zeta},S,\Z^{d})$ be a substitutive subshift from an aperiodic primitive reduced  constant-shape substitution $\zeta$. Then, the quotient group $\Aut(X_{\zeta},S,\Z^{d})/\left\langle S\right\rangle$ is finite. A bound for $|\Aut(X_{\zeta},S,\Z^{d})/\left\langle S \right\rangle|$ is given by an explicit formula depending only on $d$, $|\A|$, $\Vert L_{\zeta}^{-1}\Vert$, $\Vert F_{1}^{\zeta}\Vert$.
	\end{proposition}
	
	We recall that this was proved in the one-dimensional case as a consequence of the works \cite{coven2016computing,cyr2015automorphism,donoso2016lowcomplexity}, where they proved that the automorphism group of a minimal subshift with non-superlinear complexity is virtually generated by the shift action.
	
	\subsubsection{Decidability of the factorization problem between aperiodic substitutive subshifts}
	
	In this section, we prove that the factorization problem is decidable for aperiodic substitutive subshifts, given by substitutions sharing the same expansion matrix (\cref{Thm:DecidabilityFactorizationProblem}). This extends the one proved by I. Fagnot \cite{fagnot1997facteurs} and F. Durand, J. Leroy \cite{durand2022decidability} in the one-dimensional constant-length case as follows. 
	
	\begin{theorem}\label{Thm:DecidabilityFactorizationProblem}
		Let $\zeta_{1},\zeta_{2}$ be two aperiodic primitive constant-shape substitutions with the same expansion matrix $L$. It is decidable to know whether there exists a factor map between $(X_{\zeta_{1}},S,\Z^{d})$ and $(X_{\zeta_{2}},S,\Z^{d})$.
	\end{theorem}

	First, we note that, by \cref{ConjugationSubstitutionsDifferentFundamentalDomains}, we may assume that the substitutions share the same support. Let $\zeta_{1}, \zeta_{2}$ be two aperiodic primitive constant-shape substitutions with the same expansion matrix $L$ and support $F_{1}$. 
	
	\begin{proposition}\label{prop:caract factor}
		Let $\zeta_{1}:\A \to \A^{F_1},\zeta_{2}:\B \to \B^{F_1}$ be two aperiodic primitive constant-shape substitutions with the same expansion matrix $L$ and the same support $F_{1}$. 
		Let $\Phi:\A^{[B({\bm 0},r)\cap \Z^{d}]} \to \B$ be a block map of radius $r$ and let $\phi$ be the associated factor map from $\A^{\mathbb{Z}^d}$ to $\phi(\A^{\mathbb{Z}^d})$. 
		If there exists ${\bm f} \in \mathbb{Z}^d$ such that $S^{\bm f} \phi \zeta_1(x) = \zeta_2 \phi(x)$ for all $x \in X_{\zeta_1}$, then $\phi$ is a factor map from $X_{\zeta_1}$ to $X_{\zeta_2}$.
	\end{proposition}
	
	\begin{proof}
		A straightforward induction shows that for all $n >0$, there exists ${\bm f_n} \in \mathbb{Z}^d$ such that $S^{{\bm f}_n} \phi \zeta_1^n(x) = \zeta_2^n \phi(x)$ for all $x \in X_{\zeta_1}$.
		Let $x \in X_{\zeta_1}$ be a fixed point of $\zeta_1$.
		By minimality, and using the F\o lner property, it suffices to show that any pattern of the form $\phi(x)_{{\bm f}_n+F_n}$ belongs to $\mathcal{L}(X_{\zeta_2})$.
		Since $x$ is a fixed point of $\zeta_1$, we have that
		$\phi(x)_{F_n+{\bm f}_n} = (\zeta_2^n(\phi(x)))_{F_n} = \zeta_2^n(\phi(x)_{\bm 0})$, which concludes the proof.
	\end{proof}

	\begin{corollary}\label{corollary:decidegivenblockmap}
		Let $\zeta_{1}:\A \to \A^{F_1},\zeta_{2}:\B \to \B^{F_1}$ be two aperiodic primitive constant-shape substitutions with the same expansion matrix $L$ and the same support $F_{1}$. Let $r>0$ and $\Phi:\A^{[B({\bm 0},r)\cap \Z^{d}]} \to \B$ be a block map of radius $r$ and let $\phi$ be the associated factor map from $\A^{\mathbb{Z}^d}$ to $\phi(\A^{\mathbb{Z}^d})$. 
		For any ${\bm f} \in \mathbb{Z}^d$, it is decidable whether $S^{\bm f} \phi \zeta_1(x) = \zeta_2 \phi(x)$ for all $x \in X_{\zeta_1}$. 
	\end{corollary}
	\begin{proof}
		Note that 
		\begin{align}\label{ConditionForDecidability1}
			S^{{\bm f}}\phi\zeta_{1}(x)=\zeta_{2}(\phi(x)) 
			& \Leftrightarrow (\forall {\bm n}\in \Z^{d})(\forall {\bm g}\in F_{1}), \phi\zeta_{1}(x)_{L({\bm n})+{\bm g}+{\bm f}} =\zeta_{2}(\phi(x))_{L({\bm n})+{\bm g}}\notag\\
			& \Leftrightarrow (\forall {\bm n}\in \Z^{d})(\forall {\bm g}\in F_{1}), \phi\zeta_{1}(x)_{L({\bm n})+{\bm g}+{\bm f}} =\zeta_{2}(\phi(x)_{{\bm n}})_{{\bm g}}\notag\\
			& \Leftrightarrow (\forall {\bm n}\in \Z^{d})(\forall {\bm g}\in F_{1}), \phi\zeta_{1}(x)_{L({\bm n})+{\bm g}+{\bm f}} =\zeta_{2}(\Phi(x|_{{\bm n}+[B({\bm 0},r)\cap \Z^{d}]}))_{{\bm g}}.
		\end{align}
		
		Let ${\bm c}_{{\bm g}}\in C$ and ${\bm h}_{{\bm g}}\in F_{1}$ be such that ${\bm g}+{\bm f}=L({\bm c}_{{\bm g}})+{\bm h}_{{\bm g}}$, then \eqref{ConditionForDecidability1} is equivalent to
		\begin{align}\label{ConditionForDecidability2}(\forall {\bm n}\in \Z^{d})(\forall {\bm g}\in F_{1}), \Phi\left(\zeta_{1}(x)|_{L({\bm n})+L({\bm c}_{\bm g})+{\bm h}_{{\bm g}}+[B({\bm 0},r)\cap \Z^{d}]}\right) =\zeta_{2}(\Phi(x|_{{\bm n}+[B({\bm 0},r)\cap \Z^{d}]}))_{{\bm g}}.
		\end{align}
		
		Using \cref{Prop:FiniteSetSatisfyingParticularProperties} with $A = [B({\bm 0},r)\cap \Z^{d}]$ and $F = F_1$, we find a set $E \Subset \Z^{d}$ such that $\Vert E\Vert\leq \Vert L^{-1}\Vert r/(1-\Vert L^{-1}\Vert)+3\bar{r}$ and
		$$F_1 + [B({\bm 0},r)\cap \Z^{d}]\subseteq L(E)+F_{1},$$
		
		\noindent so we can rewrite \eqref{ConditionForDecidability2} as
		\begin{align}\label{ConditionForDecidability3}(\forall {\bm n}\in \Z^{d})(\forall {\bm g}\in F_{1})& , \Phi\left(\zeta_{1}(x)|_{L({\bm n})+L({\bm c}_{\bm g})+L(E)+F_{1}}\right) =\zeta_{2}(\Phi(x|_{{\bm n}+[B({\bm 0},r)\cap \Z^{d}]}))_{{\bm g}}\notag\\
			(\forall {\bm n}\in \Z^{d})(\forall {\bm g}\in F_{1})& , \Phi\left(\zeta_{1}(x|_{{\bm n}+{\bm c}_{\bm g}+E})\right) =\zeta_{2}(\Phi(x|_{{\bm n}+[B({\bm 0},r)\cap \Z^{d}]}))_{{\bm g}}.
		\end{align}
		
		We note that $\diam(\supp(x|_{{\bm n}+{\bm c}_{g}+E}))\leq 2 \Vert E\Vert\leq 2\Vert L^{-1}\Vert r/(1-\Vert L^{-1}\Vert)+6\bar{r}$ and $\diam(\supp(x|_{{\bm n}+[B({\bm 0},R)\cap \Z^{d}]}))\leq 2r$. Consider $\mathfrak{R}=\max\{2r,2\Vert L^{-1}\Vert r/(1-\Vert L^{-1}\Vert)+6\bar{r}\}$. To test if a sliding block code satisfy \eqref{ConditionForDecidability3}, we consider any pattern $\texttt{w}\in \mathcal{L}(X_{\zeta_1})$ with support
		$[B({\bm 0},R_{X_{\zeta_1}}(\mathfrak{R}))\cap \Z^{d}]$, where $R_{X_{\zeta_{1}}}(\cdot)$ denote the repetitivity function of the substitutive subshift $X_{\zeta_{1}}$. By \cref{GrowthRepetititvtyFunction}, we note that this discrete ball contains an occurrence of any pattern of the form $x|_{{\bm n}+{\bm c}_{\bm g}+E}$ and $x|_{{\bm n}+[B({\bm 0},r)\cap \Z^{d}]}$ for any ${\bm n}\in \Z^{d}$ and ${\bm g}\in F_{1}$.
	\end{proof}
	
	\begin{proof}[Proof of \cref{Thm:DecidabilityFactorizationProblem}.]
		Let $R_{\zeta_{2}}$ be a constant of recognizability for $\zeta_{2}$.
		Using \cref{TopologicalRigidityOfFactors} and \cref{prop:caract factor}, there is a factor map from $X_{\zeta_1}$ to $X_{\zeta_2}$ if and only there exist $n >0$, ${\bm f} \in F_n$ and a factor map $\phi$ with radius $R=2\bar{r}+R_{\zeta_{2}}+1$ that satisfies $S^{{\bm f}}\phi \zeta_{1}^{n}(x)=\zeta_{2}^{n}\phi(x)$ for all $x \in X_{\zeta_1}$.
		
		We first show that if such a factor map exists, we can find another factor map $\psi$ with the same radius such that $S^{{\bm f}}\phi \zeta_{1}^{n}(x)=\zeta_{2}^{n}\phi(x)$ for all $x \in X_{\zeta_1}$ and for some $n \leq |\mathcal{B}|^{|\A|^{R}}$.  
		Indeed, Equation~\eqref{eq:definition of phi_n} defines other factor maps $\phi_n$ that, by \eqref{eq:reductionofradius}, for $n\geq \lfloor\log_{2}(R)\rfloor$, also have radius $R$.
		We can thus find two indices $m,n\in \llbracket \lfloor\log_{2}(\bar{r})\rfloor, \lfloor\log_{2}(\bar{r})\rfloor+|\mathcal{B}|^{|\A|^{\bar{r}}}\rrbracket$ such that $\phi_{n}=\phi_{m} = \psi$ and $m<n$. By Equation~\eqref{eq:relation between phi_n's}, the factor maps $\phi_n$ satisfy $S^{{\bm g}}\phi_{n}\zeta_{1} =\zeta_{2}(\phi_{n+1})$ for some ${\bm g} \in F_1$. Hence, there exists ${\bm f} \in F_{n-m}$ such that $S^{{\bm f}}\psi \zeta_{1}^{n-m}(x)=\zeta_{2}^{n-m}\psi(x)$ for all $x \in X_{\zeta_1}$.
		
		Finally, for every $n \leq |\mathcal{B}|^{|\A|^{R}}$, every ${\bm f} \in F_n$ and every block map $\Phi:\A^{[B(\mathbf{0},R)\cap \Z^{d}]} \to \B$, \cref{corollary:decidegivenblockmap} allows us to decide whether that block map defines a factor map from $X_{\zeta_1}$ to $X_{\zeta_2}$. Since there is only a finite number of possibilities, this completes the proof.
	\end{proof}

	Using the fact that minimal substitutive subshifts from aperiodic primitive constant-shape substitutions are coalescent (\cref{Coalescence}), we can decide if two aperiodic primitive constant-shape substitutions with the same expansion matrix are conjugate.
	
	\begin{corollary}\label{cor:DecidabilityIsomorphismProblem}
		Let $\zeta_{1},\zeta_{2}$ be two aperiodic primitive constant-shape substitutions with the same expansion matrix $L$. It is decidable to know whether $(X_{\zeta_{1}},S,\Z^{d})$ and $(X_{\zeta_{2}},S,\Z^{d})$ are topologically conjugate.
	\end{corollary}
	
	\subsubsection{Aperiodic symbolic factors of substitutive subshifts}\label{Subsection:Listing}
	
	The radius of the sliding block codes given by \cref{TopologicalRigidityOfFactors} can be improved when the substitutions involved are injective on letters. Indeed, under the assumption of injectivity, we can deduce from \eqref{Eq:EqualityOfPatternsRigidty} that, by injectivity of the substitution $\zeta_{2}$, $\phi_{n}(x)$ and $\phi_{n}(y)$ coincide on the set $\{{\bm m}\in \Z^{d}\colon L^{n}({\bm m})+F_{n}\subseteq (L^{n}(P_{n})+F_{n})^{\circ r}-{\bm f}_{n}(\phi)\}$. Moreover, if $\phi_{n}$ is given by an $P_{n}$-local map we need that
	$$
	F_{n}\subseteq (L^{n}(P_{n})+F_{n})^{\circ r}-{\bm f}_{n}(\phi),
	$$
	which is true if $F_{n}+F_{n}+[B({\bm 0},r)\cap \Z^{d}]\subseteq L^{n}(P_{n})+F_{n}$. Proceeding as in the proof of \cref{TopologicalRigidityOfFactors} we get that $\phi_{n}$ has radius at most $3\bar{r}+1$. As mentioned in \cref{rem:GoodUniformlyBounded}, this bound is independent of the choices of the powers $\zeta_{1}$ and $\zeta_{2}$. 
	We thus have the following result.
	\begin{corollary}
		\label{corollary:radiusdependindonlyonL}
		Let $\zeta_{1},\zeta_{2}$ be two aperiodic primitive constant-shape substitutions with the same expansion matrix $L$ and support $F_{1}$, such that $\zeta_2$ is injective.
		If $\phi: X_{\zeta_1} \to X_{\zeta_2}$ is a factor map, then there is a factor map $\psi:X_{\zeta_1} \to X_{\zeta_2}$ with radius at most $3\bar{r}+1$ such that $\phi = S^{\bm j} \psi$ for some ${\bm j} \in \mathbb{Z}^d$.
	\end{corollary}
	
	Now, as in \cite{durand2022decidability}, we prove that we can always assume that the aperiodic substitutions are injective on letters. Indeed, let $\zeta:\A\to\A^{F}$ be an aperiodic primitive constant-shape substitution and consider $\mathcal{B}\subseteq \A$ such that for any $a\in \A$ there exists a unique $b\in \mathcal{B}$ satisfying $\zeta(a)=\zeta(b)$. We define the map $\Phi:\A\to\B$ such that $\Phi(a)=b$ if and only if $\zeta(a)=\zeta(b)$. Then, there exists a unique constant-shape substitution $\tilde{\zeta}:\mathcal{B}\to\mathcal{B}^{F}$ defined by $\tilde{\zeta}\circ \phi=\phi\circ \zeta$. It is clear that $\tilde{\zeta}$ is primitive and $\Phi$ induces a topological factor from $(X_{\zeta},S,\Z^{d})$ to $(X_{\tilde{\zeta}},S,\Z^{d})$ denoted by $\phi$. The substitution $\tilde{\zeta}$ defined this way is called the \emph{injectivization} of $\zeta$.
	
	\begin{proposition}\label{prop:ConjugatetoInjectivesubstitution}\cite{blanchard2004scrambled}
		Let $\zeta$ be an aperiodic primitive constant-shape substitution and $\tilde{\zeta}$ be its injectivization. Then $(X_{\zeta},S,\Z^{d})$ and $(X_{\tilde{\zeta}},S,\Z^{d})$ are topologically conjugate.
	\end{proposition}
	
	\begin{proof}
		Since $\phi$ is a factor map, we prove that it is one-to-one. Let $x,y\in X_{\zeta}$ be such that $\phi(x)=\phi(y)$. Note that $\zeta(x)=\zeta(\phi(x))=\zeta(\phi(y))=\zeta(y)$. The fact that $\zeta:X_{\zeta}\to \zeta(X_{\zeta})$ is a homeomorphism let us conclude that $x=y$.
	\end{proof}
	
	By construction, $\tilde{\zeta}$ may not be injective on letters, so we proceed in the same way to obtain an injectivization of $\tilde{\zeta}$ (and of $\zeta$). Since in each step, the cardinality of the alphabet is decreasing, we will obtain, in finite steps, an injective substitution $\overline{\zeta}$ such that $(X_{\zeta},S,\Z^{d})$ and $(X_{\overline{\zeta}},S,\Z^{d})$ are topologically conjugate.
	
	Now, let $\zeta$ be an aperiodic primitive constant-shape substitution. From \cite[Theorem 3.26]{cabezas2023homomorphisms}, we know that any aperiodic symbolic factor of $(X_{\zeta},S,\Z^{d})$ is conjugate to an aperiodic primitive constant-shape substitution with the same expansion matrix and support of some power of $\zeta$. This substitution may not be injective, but using \cref{prop:ConjugatetoInjectivesubstitution}, we can find an injective substitution which is conjugate to the symbolic factor. Then, thanks to \cref{corollary:radiusdependindonlyonL}, we only need to test finitely many sliding block codes, which are the ones of radius $3\bar{r}+1$. In particular, we prove the following result, extending what is known~\cite{durand1999substitutional} for linearly recurrent shifts (in particular, aperiodic primitive substitutions):
	
	\begin{lemma}\label{Lemma:FinitelyManyAperiodicSymbolicFactors}
		Let $\zeta$ be an aperiodic primitive constant-shape substitution. The substitutive subshift $(X_{\zeta},S,\Z^{d})$ has finitely many aperiodic symbolic factors, up to conjugacy.
	\end{lemma}

	\section{Future works and discussions}
	
	\subsection{Listing the factors and decidability of the aperiodicity of multidimensional substitutive subshifts.} 
	
	Since substitutions are defined by finite objects, it is natural to ask about the decidability of some properties about them. 
	\cref{Lemma:FinitelyManyAperiodicSymbolicFactors} states that any aperiodic substitutive subshift $X_\zeta$ has finitely many aperiodic symbolic factors.
	Furthermore, from \cite[Theorem 3.26]{cabezas2023homomorphisms} and~\cref{prop:ConjugatetoInjectivesubstitution}, any such factor is conjugate to a substitutive subshift defined by an injective constant-shape substitution with the same support as $\zeta$.
	Thus we would like to give a list $\zeta_1,\dots,\zeta_k$ of injective constant-shape substitutions that define all aperiodic symbolic factors of $X_\zeta$.
	
	\cref{corollary:radiusdependindonlyonL} allows to give a bound on the size of the alphabet of any injective constant-shape substitution that would define an aperiodic symbolic factor of $X_\zeta$.
	Hence, we can produce a finite list of candidates.
	Furthermore, if we know that two substitutions from that list are aperiodic, then~\cref{TopologicalRigidityOfFactors} allows us to decide whether they are conjugate. 
	Therefore, positively answering to the following question would allow us to list all possible aperiodic symbolic factors of $X_\zeta$.
	
	\begin{question}
		Is it decidable whether a primitive constant-shape substitution is aperiodic?
	\end{question}
	
	\subsection{Toward a Cobham's theorem for constant-shape substitutions} In \cite{fagnot1997facteurs} it was proved that if $\zeta_{1}, \zeta_{2}$ are two one-dimensional aperiodic primitive constant-length substitutions, and $(X_{\zeta_{2}},S,\Z)$ is a symbolic factor of $(X_{\zeta_{1}},S,\Z)$, then their lengths have a common
	power (greater than 1), generalizing a well-known result proved by A. Cobham \cite{cobham1969recognizable}. 
	In the multidimensional framework, Theorem~\ref{ConjugationSubstitutionsDifferentFundamentalDomains} and~\cite{durand2008cobhamsemenov} imply that this still remains true when the expansion matrix is equal to a multiple of the identity, but there is no generalized version for all constant-shape substitutions, which raises the following question.
	
	\begin{question}
		Is there a version of Cobham's theorem for constant-shape substitutions?
	\end{question}
	
	\subsection{Connections with first-order logic theory} 
	In the one-dimensional case, the decidability of the factorization problem between two constant-length substitutions was proved in \cite{durand2022decidability} using automata theory and its connexion with first order logic. 
	We describe the proof in the following. 
	Let $k\geq 2$. 
	An infinite sequence is called \emph{k-automatic} if there is a finite state automaton with output (we refer to~\cite{allouche2003automatic} for definitions) that, reading the base-$k$ representation of a natural number $n$, outputs the letter $x(n)$.

	Now, consider the first-order logical structure $\left\langle \N,+,V_{k},=\right\rangle$, where $V_{k}$ corresponds to $k$-\emph{valuation function} $V_{k}:\N\to\N$ by $V_{k}(0)=1$ and
	$$V_{k}(n)=\max\{p\colon q^{r}\ \text{divides}\ n\}.$$
	
	A subset $E\subseteq \N$ is called $k$-\emph{definable} if there exists a first-order formula $\phi$ in $\left\langle \N,+,V_{k},=\right\rangle$ such that $E=\{n\in \N\colon \phi(n)\}$.  
	An infinite sequence $x\in \A^{\N}$ is called $k$-\emph{definable} if for all $a\in \A$ the set $\{n\colon x_{n}=a\}$ is $k$-definable.
	
	\begin{theorem}\cite{buchi1960weak,cobham1972uniform}
		\label{thm:automaticiffdefinable}
		Let $\A$ be a finite alphabet and $k\geq 2$. An infinite sequence $x\in \A^{\N}$ is $k$-automatic if and only if it is $k$-definable and if and only if it is the image under a letter-to-letter substitution of a fixed point of a substitution of constant length $k$.  
	\end{theorem}
	
	\begin{theorem}
		The theory $\left\langle \N,+,V_{k},=\right\rangle$ is decidable, i.e, for each closed formula expressed in this first order logical structure, there is an algorithm deciding whether it is true or not.
	\end{theorem}
	
	\cref{thm:automaticiffdefinable} extends to multidimensionnal constant-shape substitutions for which the expansion matrix is proportional to the identity.
	Hence, in that case, the decidability of the factorization problem can be deduced by the Büchi-Bruy\`ere theorem. 
	In the general case, it is not clear that these tools can be applied. 
	For an integer expansion matrix $L\in \mathcal{M}(d,\Z)$, we don't have the analogous notions of $L$-automaticity and of $L$-definability.
	This raises the following questions:
	
	\begin{question}
		\begin{enumerate}
			\item Can we extend the notion of $k$-automaticity to the general case of integer expansion matrices ? More precisely, can we define $L$-automatic sequences so that theses sequences are exactly the image under a letter-to-letter substitution of a constant-shape substitution with expansion matrix $L$ ? 
			
			\item Can we define a logical structure depending on $L$ and a notion of $L$-definability to obtain an analog of the Büchi-Bruy\`ere theorem for constant-shape substitutions?
			
			\item Assuming that the logical structure exists, is this theory decidable ?
		\end{enumerate}
	\end{question}
	
	\subsection{Topological Cantor factors of substitutive subshifts} In the one-dimensional case, the topological Cantor factors of aperiodic primitive substitutions are either expansive or equicontinuous \cite{durand1999substitutional}. This classification result is no
	longer true in the multidimensional framework (\cite[Example 4.3]{cabezas2023homomorphisms} is an example of an aperiodic primitive constant-shape substitution with a Cantor factor which is neither expansive neither equicontinuous). Moreover, we only study the aperiodic topological factors of substitutive subshifts, leaving open the study of the topological factors with non-trivial periods. The following are open questions:
	
	\begin{question}
		\begin{enumerate}
			\item Are the expansive factors of aperiodic substitutive subshifts also substitutive?
			\item Is there a classification theorem for topological Cantor factors of constant-shape substitutions?
			\item Do aperiodic primitive constant-shape substitutions have a finite number of aperiodic topological Cantor factors, up to conjugacy?
		\end{enumerate}
	\end{question}

	\bibliographystyle{plain}
	\bibliography{bibliography}
	
\end{document}